\documentclass{article}
\usepackage[utf8]{inputenc}
\usepackage{graphicx}
\usepackage{amssymb, amsmath, amsthm}
\usepackage{bbm}
\usepackage{enumitem}
\usepackage{tikz}
\usepackage{changepage}
\usepackage{color,soul}
\usepackage{authblk}

\usepackage[draft,authormarkup=none]{changes}
\setdeletedmarkup{{\color{red}\sout{#1}}}

\definechangesauthor[name=DLS, color=purple]{ds}

\usepackage{todonotes}
\newcommand{\dls}[2][]{\todo[color=green!40,#1]{DLS: #2}}

\usepackage{verbatim}

\usepackage[backend=biber, style = authoryear, 
citestyle=authoryear]{biblatex}
\usepackage[margin=1.0in]{geometry}
\newcommand{\N}{{\mathbb{N}}}

\newcommand{\R}{{\mathbb{R}}}

\newcommand{\E}{{\mathbb{E}}}

\newcommand{\prob}{{\mathbb{P}}}
\newcommand{\bvar}[1]{\mathbf{#1}} 
\newcommand{\ErdosRenyi}{{Erd{\H o}s-R{\'e}nyi }}

\theoremstyle{definition}
\newtheorem{theorem}{Theorem}

\newtheorem{Lemma}{Lemma}
\newtheorem{ex}{Example}
\newtheorem{definition}{Definition}
\numberwithin{definition}{section}

\numberwithin{Prop}{section}
\newtheorem{Corollary}{Corollary}
\newtheorem{Remark}{Remark}

\newenvironment{excont}
{\addtocounter{ex}{-1}\begin{ex}{\textit{\textbf{{(continued) }}}}}
  {\end{ex}}

\title{Bias-Variance Tradeoffs in Joint Spectral Embeddings}
\author{Benjamin Draves}
\author{Daniel L. Sussman}
\affil[1]{\small Department of Math. and Stat., Boston University}

\begin{document}

\maketitle

\begin{abstract}
Joint spectral embeddings facilitate analysis of multiple network data by simultaneously mapping vertices in each network to points in Euclidean space where statistical inference is then performed.
In this work, we consider one such joint embedding technique, the omnibus embedding of \textcite{OmniCLT}, which has been successfully used for community detection, anomaly detection, and hypothesis testing tasks. 
To date the theoretical properties of this method have only been established under the strong assumption that the networks are conditionally i.i.d.\ random dot product graphs.
Herein, we take a first step in characterizing the theoretical properties of the omnibus embedding in the presence of heterogeneous network data.
Under a latent position model, we show the omnibus embedding implicitly regularizes its latent position estimates which induces a finite-sample bias-variance tradeoff for latent position estimation.
We establish an explicit bias expression, derive a uniform concentration bound on the residual, and prove a central limit theorem characterizing the distributional properties of these estimates.
These explicit bias and variance expressions enable us to state sufficient conditions for exact recovery in community detection tasks and develop a pivotal test statistic to determine whether two graphs share the same set of latent positions; demonstrating that accurate inference is achievable despite the estimator's inconsistency.
These results are demonstrated in several experimental settings where statistical procedures utilizing the omnibus embedding are competitive, and oftentimes preferable, to comparable embedding techniques.
These observations accentuate the viability of the omnibus embedding for multiple graph inference beyond the homogeneous network setting. 
\end{abstract}

\section{Introduction}

Multiplex networks describe a set of entities, with multiple types of relationships among them, as a collection of networks over a common vertex set \cite{Kivela2014, Battiston2016_review}.
There is a growing demand for inferential frameworks for multiplex networks in a diverse variety of fields such as neuroscience \parencite{Battiston2016, DeDomenico2017,ginestet2017, kong2021multiplex}, transportation systems \parencite{Cardillo2013,Kaluza2010}, and the social sciences \parencite{Coscia2013, Goldblum2019, Takes2018, Szell2013, Stella_2017, Lazega2016, kim2021link}. 
While developing a principled paradigm for random graph inference has been of great interest for individual networks \parencite{Kolaczyk2009, RDPGSurvey}, lesser attention has been given to multiplex networks.
This data structure provides a more detailed representation of complex systems by viewing the collection of networks as being drawn from a multivariate network distribution.
However, it poses novel challenges when developing a formal statistical framework that requires new insights.

Several recent works have focused on extending familiar descriptive statistics such as clustering coefficients \parencite{Battiston2014, Cozzo2013, Baxter_2016} and node centrality scores \parencite{Tudisco2017, Taylor2019,bergermann2021matrix}, tools for network visualization \parencite{MuxViz2014, Fatemi2016}, and community detection algorithms \parencite{MA2018, Hmimida2015} to multilayer network data.
Multi-graph models inspired by individual network models have been proposed in an attempt to capture multilayer network structure \parencite{Bianconi2013, Nicosia2015, Murase2014} and corresponding approaches to estimation and subsequent inference include likelihood approaches, tensor decompositions, and variational methods \parencite{Paez201, Paul2018, Gligorijevic2019}.
While some of these frameworks are constructed for general multilayer networks, we restrict our attention to multiplex networks; that is we study collections of node-aligned networks over a common vertex set.

A class of models that have seen success in capturing multiplex network phenomena, such as multilayer and time varying community structure, while remaining analytically tractable are latent position models (LPM) extended to multiplex data \cite{OmniCLT, wang2017, nielsen2018, arroyo2019, jones2020multilayer, macdonald2021latent}.
LPMs for single networks posit that the vertices are each associated with a {\em latent position} in a low dimensional space \parencite{Hoff2002, CLS, rubindelanchy2017}, and
one of the most ubiquitous examples is the Random Dot Product Graph (RDPG) \parencite{RDPG}.
Under this model, edge connection probabilities are given by the inner product of the incidents nodes' latent positions.
Estimates of these vectors, called \textit{node embeddings}, are then amenable to statistical analysis using familiar techniques from multivariate statistics and machine learning \parencite{Luxburg2007, RDPGSurvey, PerfectClustering}.
In characterizing the behavior of these estimates, one can derive guarantees on the statistical algorithms that use these node representations \parencite{PerfectClustering, Semi-Par, EigenCLT, OmniCLT}. 

In multiplex networks, we anticipate different layers of the network to share common structure while maintaining layer-specific deviations. 
To date, little work has been dedicated to understanding the finite sample properties and asymptotic distributions of latent position estimators under (layer)-heterogeneous network models, save a few notable exceptions \parencite{arroyo2019, jones2020multilayer}.
While the utility of each method depends on the interplay between the embedding technique and the inference task at hand, embedding methods that share strength across networks regularly outperform techniques that do not use this common structure.

In this work, we consider a joint embedding technique, the omnibus embedding of \textcite{OmniCLT}, for latent position estimation.
The omnibus embedding has proven useful in a wide array of inference tasks including community detection~\parencite{pantazis2021importance, arroyo2019}, hypothesis testing~\parencite{OmniCLT, jones2020multilayer, Vogelstein2019, arroyo2019}, anomaly detection~\parencite{chen2020multiple}, and vertex classification~\parencite{pantazis2021importance}.
To date, however, the theoretical properties of the latent position estimates produced by the omnibus embedding have only been established under a (layer-)homogeneous network model where each random adjacency matrix is marginally distributed according to a RDPG with the same latent positions.
In what follows, we take a first step in characterizing the asymptotic behavior of the omnibus embedding estimates under a simple heterogeneous network model. 
With these results, we are equipped to assess the viability of omnibus embedding as an analysis technique beyond the homogeneous network setting.

To initiate this analysis, we propose the \textit{Eigen-Scaling Random Dot Product Graph} (ESRDPG) as a model that extends the RDPG to multiplex network data.
The ESRDPG is similar to models proposed in \textcite{wang2017}, \textcite{arroyo2019}, and \textcite{nielsen2018}. 
In studying the omnibus embedding under the ESRDPG, we are able to extend the analysis completed in \textcite{OmniCLT} to the (layer-)heterogeneous network setting. 
Under the ESRDPG, we provide an analytic expression for the bias of the latent position estimates, along with concentration and distributional results.
Moreover, we establish an explicit covariance between these latent position estimates that enable for the rigorous analysis of algorithms that use linear combinations of these estimates. 
These results shed light on an implicit finite-sample bias-variance tradeoff induced by the omnibus embedding for heterogeneous networks. 
Next, we analyze the impact of the bias-variance tradeoff in subsequent statistical inference tasks such as multiplex community detection and network hypothesis testing. 
We theoretically and empirically demonstrate that applying common clustering techniques (e.g. k-means clustering, Gaussian Mixture Models) to the estimated latent positions produced by the omnibus embedding is competitive with state of the art methods for multiplex community detection and substantially outperforms unbiased embedding approaches that do not pool information across layers.
Moreover, we propose a pivotal test statistic that can be estimated directly from the data that allow for parametric testing procedures to be utilized in multiple network hypothesis testing.

This paper will be organized as follows.
In Section~\ref{Section: Background}, we present the omnibus embedding and consider examples where accurate inference is attained using the omnibus embedding in the presence of heterogeneous network data.
We then introduce the ESRDPG model and discuss its properties.
In Section~\ref{Section: Main Results}, we provide theoretical results that establish the asymptotic bias and distribution of the omnibus embedding's latent position estimates and 
explore this bias-variance tradeoff in latent position estimation.
In Section~\ref{Section: Statistical Consequences}, we analyze the ramifications of these theoretical results on statistical tasks such as multiplex community detection and two-graph hypothesis testing and empirically demonstrate these results in rigorous simulation settings. 
Finally, in Section~\ref{Section: Extensions and Discussion} we discuss several extensions to the current work.

\section{Background \& the ESRDPG}\label{Section: Background}

In this section, we first revisit the omnibus embedding and demonstrate its efficacy in two multiplex inference tasks.
We then review the Random Dot Product Graph (RDPG) for single-layer network data and introduce our extension to multiplex network data, the Eigen-Scaling RDPG (ESRDPG) for modeling heterogeneous networks.
Finally, we collect notation used throughout this paper in Table~\ref{tab:notation_table}.

\subsection{Omnibus Embedding}
Suppose we have a multiplex network represented by a collection of $m$ adjacency matrices $\{\bvar{A}^{(g)}\}_{g=1}^m\subset \R^{n\times n}$ over a common vertex set $\mathcal{V}$ of size $|\mathcal{V}| = n$.
The omnibus embedding \textit{simultaneously} embeds the vertices of each network into a common $d$-dimensional Euclidean space $\R^d$ by decomposing the $nm\times nm$ omnibus matrix \parencite{OmniCLT}.
We introduce this embedding procedure in Definition~\ref{def:omnibus_embedding}.

\begin{definition}\label{def:omnibus_embedding}
Let $\{\bvar{A}^{(g)}\}_{g=1}^m\in\R^{n\times n}$ be a a set of undirected, vertex-aligned, adjacency matrices.
Let $\tilde{\bvar{A}}\in\R^{nm\times nm}$ be the omnibus matrix of $\{\bvar{A}^{(g)}\}_{g=1}^m$ given by 
\begin{equation*}
    \tilde{\bvar{A}} = \begin{bmatrix}
    \bvar{A}^{(1)} & \frac{\bvar{A}^{(1)} + \bvar{A}^{(2)}}{2} & \dots & \frac{\bvar{A}^{(1)} + \bvar{A}^{(m)}}{2}\\
    \frac{\bvar{A}^{(2)} + \bvar{A}^{(1)}}{2} & \bvar{A}^{(2)} & \dots & \frac{\bvar{A}^{(2)} + \bvar{A}^{(m)}}{2}\\
    \vdots & \vdots & \ddots & \vdots\\
    \frac{\bvar{A}^{(m)} + \bvar{A}^{(1)}}{2} & \frac{\bvar{A}^{(m)} +\bvar{A}^{(2)}}{2} & \dots & \bvar{A}^{(m)}
    \end{bmatrix}.
\end{equation*}
Denote the eigendecomposition of $\tilde{\mathbf{A}}$ as
$\tilde{\bvar{A}} = [\bvar{U}_{\tilde{\bvar{A}}}|\bvar{U}_{\tilde{\bvar{A}}}^{\perp}][\bvar{S}_{\tilde{\bvar{A}}}\oplus\bvar{S}_{\tilde{\bvar{A}}}^{\perp}][\bvar{U}_{\tilde{\bvar{A}}}|\bvar{U}_{\tilde{\bvar{A}}}^{\perp}]^T$
where the columns of $\bvar{U}_{\tilde{\bvar{A}}}\in\R^{nm\times d}$ are the $d$ eigenvectors of $\tilde{\bvar{A}}$ corresponding to the $d$ largest positive eigenvalues of $\tilde{\bvar{A}}$ and $\bvar{S}_{\tilde{\bvar{A}}}\in\R^{d\times d}$ is a diagonal matrix containing these $d$ eigenvalues in non-increasing order. 
Then the \textit{d-dimensional omnibus embedding} of $\{\bvar{A}^{(g)}\}_{g=1}^m$, denoted as $ \text{Omni}\left(\{\bvar{A}^{(g)}\}_{g=1}^m,d\right)$, is given by the $d$-dimensional spectral embedding of $\tilde{\bvar{A}}$
\begin{equation}
    \text{Omni}\left(\{\bvar{A}^{(g)}\}_{g=1}^m,d\right) = \bvar{U}_{\tilde{\bvar{A}}}\bvar{S}_{\tilde{\bvar{A}}}^{1/2}\in\R^{nm\times d}.
\end{equation}
\end{definition}
Note that the omnibus embedding is of dimension $nm\times d$ yielding $m$ separate node embeddings for each vertex $v\in\mathcal{V}$.
Indeed, the $n(g-1) + i$ row of $\text{Omni}(\{\bvar{A}^{(g)}\}_{g=1}^m, d)$, which we denote $\hat{\bvar{X}}^{(g)}_i$, is the omnibus node embedding for vertex $i\in[n]$ in graph $g\in[m]$.
These separate node embeddings for each vertex in each graph offer a clear path to address multiplex network inference tasks. 
For instance, graph level hypothesis testing can be carried out by considering the test statistic $\hat{\delta}^{(g,k)}=n^{-1}\sum_{i=1}^n\|\hat{\bvar{X}}_i^{(g)} - \hat{\bvar{X}}_i^{(k)}\|_2^2$ for each $g,k\in[m]$, community detection can be completed by applying a clustering algorithm to the points $\{\bar{\bvar{X}}_i := m^{-1}\sum_{g=1}^m\hat{\bvar{X}}_i^{(g)}\}_{i=1}^n$, and vertex anomaly detection can be carried out by analyzing the vertex level variation $\hat{\tau}_i^{(g)} = \hat{\bvar{X}}_i^{(g)} - \bar{\bvar{X}}_i$.

To date, the omnibus embedding has only been theoretically analyzed under a homogeneous network model where the $\{\bvar{A}^{(g)}\}_{g=1}^m$ are drawn i.i.d.\ from the same distribution 
so each $\bvar{A}^{(g)}$ has the same conditional expectation given the latent positions.
Several recent works, however, present empirical evidence where the omnibus embedding offers competitive performance on network inference tasks in the \textit{heterogeneous} network setting \parencite{arroyo2019, pantazis2021importance, jones2020multilayer, chen2020multiple}.
We add to this empirical evidence in Example~\ref{Ex:omni_under_heterogenous_comm_det} and Example~\ref{Ex:omni_under_heterogenous_anomly_detection} where the omnibus embedding offers competitive performance in both community detection and vertex anomaly detection tasks where the networks are drawn from different distributions.

\begin{ex}\label{Ex:omni_under_heterogenous_comm_det}
Suppose there are $m=2$ networks over $n = 300$ vertices where the vertices are partitioned into three communities of size $n/3$.
We sample $\bvar{A}^{(1)}$ from a three group, stochastic block model (SBM) with connectivity matrix $\bvar{B}$.
Next, using the same community assignments as in $\bvar{A}^{(1)}$, we sample $\bvar{A}^{(2)}$ from a three group SBM with connectivity matrix $\bvar{B} + \lambda\bvar{J}$, where $\lambda\in[0,1]$.
The parameters are given by 
\begin{equation}\label{motivation_ex_parameters}
    \bvar{B} = \begin{bmatrix}
    0.2 & 0.1 & 0.1 \\
    0.1 & 0.2 & 0.1 \\
    0.1 & 0.1 & 0.2 
    \end{bmatrix}\quad \quad \bvar{J} = 
    \begin{bmatrix}
        0 & 0 & 0 \\
        0 & 0 & 0.1 \\
        0 & 0.1 & 0 
    \end{bmatrix} \quad \quad \bvar{Z} = \begin{bmatrix}
        \bvar{1}_{n/3} & \bvar{0}_{n/3} & \bvar{0}_{n/3}\\
        \bvar{0}_{n/3} & \bvar{1}_{n/3} & \bvar{0}_{n/3}\\
        \bvar{0}_{n/3} & \bvar{0}_{n/3} & \bvar{1}_{n/3}
    \end{bmatrix}. 
\end{equation}
Notice $\lambda = 0$ corresponds to a $K=3$ group structure, $\lambda = 1$ corresponds to a $K=2$ group SBM where the second and third community are merged into a single community of size $2n/3$, and $\lambda\in(0,1)$ denotes a convex combination of the $K=2$ and $K=3$ block probability matrices. 

Varying $\lambda\in[0,1]$, we first sample $\bvar{A}^{(1)}\sim \text{Bern}(\bvar{Z}\bvar{B}\bvar{Z}^T)$ and $\bvar{A}^{(2)}\sim\text{Bern}( \bvar{Z}(\bvar{B} + \lambda\bvar{J})\bvar{Z}^T)$ and then attain node embeddings using the omnibus embedding for both networks, denoted $\hat{\bvar{X}}^{(1)}, \hat{\bvar{X}}^{(2)}\in\R^{n\times 3}$. 
We then apply $k$-means clustering to the average node embeddings found in the rows of $2^{-1}(\hat{\bvar{X}}^{(1)} + \hat{\bvar{X}}^{(2)})$.
Finally we compute the missclassification rate of these estimated community labels and complete $T = 1000$ monte carlo iterations of this process.

We compare the classification accuracy of this approach with $k$-means clustering applied to several other spectral node embeddings which we denote as ASE1, ASE2, ASEbar,  Abar, and MASE.
ASE1 and ASE2 refer to adjacency spectral embedding of \textcite{ASE} of $\bvar{A}^{(1)}$ and $\bvar{A}^{(2)}$, respectively, and ASEbar refers to their average, after alignment. 
Abar refers to the adjacency spectral embedding of $2^{-1}(\bvar{A}^{(1)} + \bvar{A}^{(2)})$ and MASE refers to a joint spectral embedding technique introduced in \textcite{arroyo2019}.
The results of this simulation can be found in Figure~\ref{fig:motivating_ex_comm_det}.
\begin{figure}[h!]
    \centering
    \includegraphics[width = 0.6\linewidth]{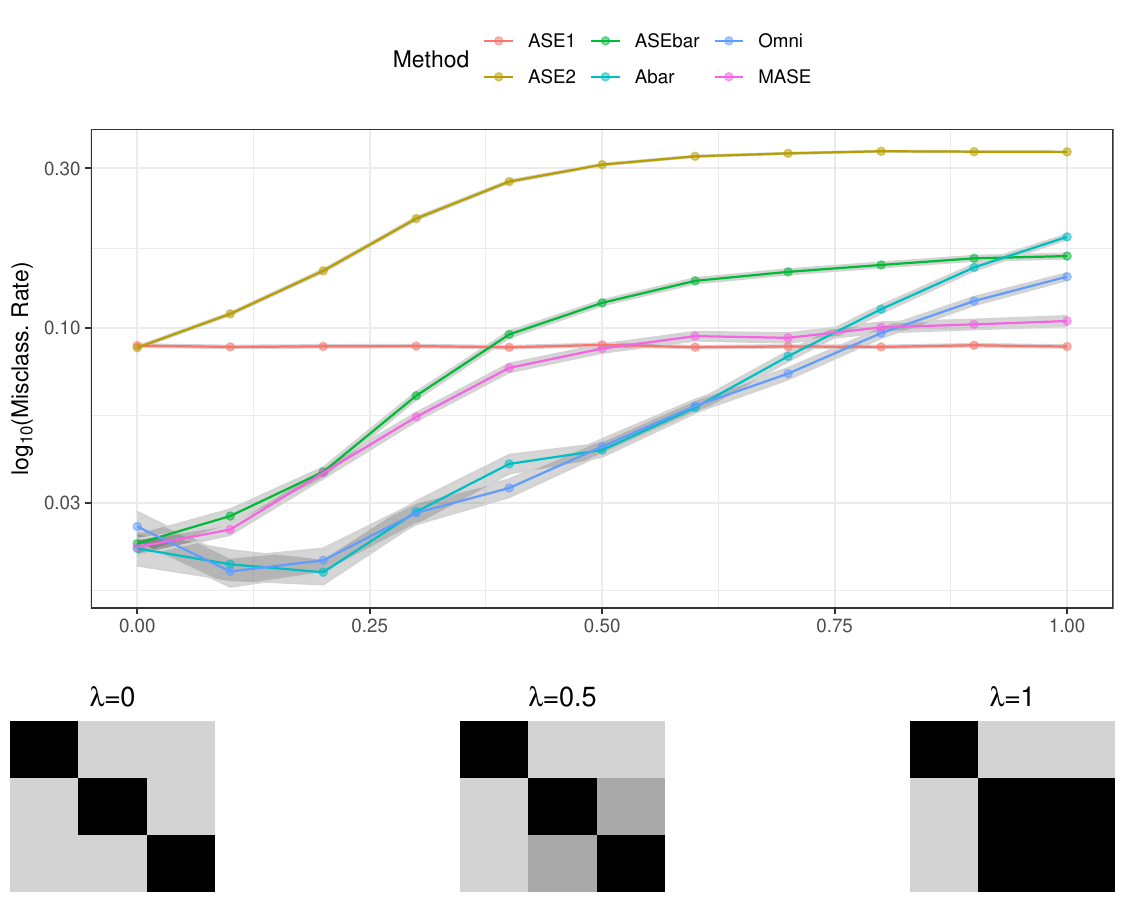}
    \caption{The missclassification rate of the $k$-means clustering algorithm applied to the node embeddings for each embedding technique and for each $\lambda \in[0,1]$.
    Block probabilities vary with $\lambda$ and are visualized across the horizontal axis.}
    \label{fig:motivating_ex_comm_det}
\end{figure}

$\lambda = 0$ represents the homogeneous network setting where ASEbar, Omni, Abar, and MASE all demonstrate comparable performance. 
These methods are all rely on information from both networks and hence have significantly lower error than the single network approaches ASE1, ASE2.
For $\lambda \geq 0.75$, ASE1 has the lowest error as the second graph introduces noise in the detection of the third community and by ignoring this network, ASE1 performs the best.
For $\lambda\in[0.1, 0.7]$ Omni and Abar offer the best performance by a non-trivial margin.
\qed
\end{ex}

Example~\ref{Ex:omni_under_heterogenous_comm_det} demonstrates the omnibus embedding offers a robust approach for community detection.  
Next, we consider the omnibus embedding as a tool for identifying anomalous vertices. 
Similar observations made here were first presented in \textcite{chen2020multiple}.

\begin{ex}\label{Ex:omni_under_heterogenous_anomly_detection}
Consider $m = 2$ networks over $n = 300$ vertices and sample $\bvar{A}^{(1)}\sim\text{Bern}(\bvar{ZBZ}^T)$ from the three group SBM presented in equation~\eqref{motivation_ex_parameters}.
Let $\bvar{M} = \bvar{ZB}^{1/2}$ be the matrix square root of $\bvar{ZBZ}^T$.
Then, for $\alpha\in[0, 0.1]$, we replace $(100\times \alpha)\%$ of the rows of $\bvar{M}$ with an i.i.d.\ sample from $\text{Dir}(1,1,1)$ which we denote as $\bvar{M}_{\alpha}$.
For the sake of this example, we regard the vertices corresponding to these altered rows as anomalous vertices.
We then sample $\bvar{A}^{(2)}\sim\text{Bern}(\bvar{M}_{\alpha}\bvar{M}_{\alpha}^T)$.

In an attempt to identify these anomalous vertices, we first embed $(\bvar{A}^{(1)}, \bvar{A}^{(2)})$ into $\R^3$ using the omnibus embedding which yields $(\hat{\bvar{X}}^{(1)}, \hat{\bvar{X}}^{(2)})$.
We then compute $\hat{\tau}_i = \|(\hat{\bvar{X}}^{(1)} - \hat{\bvar{X}}^{(2)})_i\|_2$ and visualize the sorted $\{\hat{\tau}_i\}_{i=1}^n$.
We complete a similar approach using MASE \textcite{arroyo2019} and ASE \textcite{ASE}.
Note we cannot use the Abar method from above in this context as the method collapses all networks into a single representation making graph-wise differencing impossible. 
The results of this simulation can be found in Figure~\ref{fig:motivating_ex_anomaly_detection}.

\begin{figure}[h!]
    \centering
    \includegraphics[width = 0.75\linewidth]{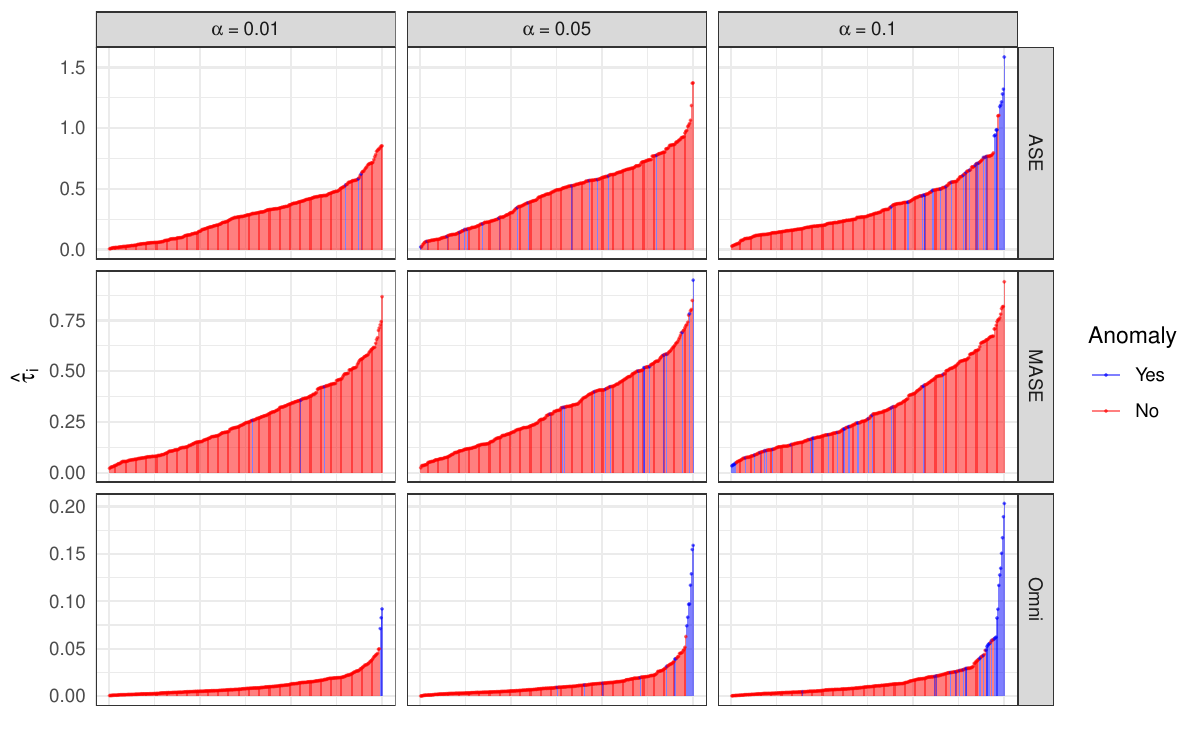}
    \caption{The test statistic $\hat{\tau}_i$ for each vertex $i\in[n]$ colored by if the node is anomalous (blue) or not (red). 
    The columns represent the proportion, $\alpha\in\{0.01, 0.05, 0.1\}$ of rows altered in $\bvar{M}$ while the rows represent each embedding method.  
    The vertical axis visualizes $\hat{\tau}_i$ while the horizontal axis is the index corresponding to the ordering $\hat{\tau}_{[i]}$. }
    \label{fig:motivating_ex_anomaly_detection}
\end{figure}

For $\alpha = \{0.01, 0.05\}$, $\text{Omni}$ begins to successfully separate the bulk of the anomalous vertices from the standard vertices while the MASE and ASE methods are not able to rank these vertices the highest. 
For $\alpha = 0.1$, Omni appears to separate the majority of anomalous vertices by a larger margin than in the $\alpha = 0.05$ setting and ASE begins to separate these vertices but not to the extent of Omni.
MASE does not rank the majority of anomalous vertices higher than the standard vertices for any value of $\alpha$.

As we detail in Section~\ref{Section: Main Results}, we anticipate the omnibus node embedding differences $\{\hat{\tau}_i\}_{i=1}^n$ will be poor estimates of the true distances $\|\bvar{X}_i^{(1)} - \bvar{X}_i^{(2)}\|_2$ but note that these estimates correctly rank the distances between latent positions.
This suggests that while the omnibus embedding may not yield accurate latent position estimates, the method nonetheless provides node embeddings whose \textit{relative} positions elicit vertex-level differences between graphs.
Finally, while this analysis is exploratory in nature, a formal semi-parametric testing approach was proposed in \textcite{chen2020multiple}.
Their findings are consistent with the observations made above and suggest that the omnibus embedding can serve as an anomaly detection tool in the heterogeneous network setting. 
\qed
\end{ex}

Example~\ref{Ex:omni_under_heterogenous_comm_det} and Example~\ref{Ex:omni_under_heterogenous_anomly_detection} provide empirical evidence that the omnibus embedding offers competitive, if not preferable, performance on inference tasks in the heterogeneous network setting compared to other joint spectral embedding methods.
In an attempt to understand this behavior, we take a first step in characterizing the omnibus embedding under a heterogeneous network model. 
We begin this investigation by introducing the Eigen-Scaling Random Dot Product Graph.

\subsection{Eigen-Scaling Random Dot Product Graph}\label{Subsection: ESRDPG}

Under the Random Dot Product Graph~\parencite{RDPG}, each vertex in the graph is associated with a latent position in Euclidean space.
Conditional on these latent positions, edge connection probabilities are given by the inner product of the latent positions. 

\begin{definition}\label{RDPG}
Suppose that $\bvar{x}_1, \bvar{x}_2, \ldots, \bvar{x}_n\in\R^d$ have the property that $\bvar{x}_i^T\bvar{x}_j \in[0,1]$ for all $i,j\in [n]$.
Then we say that a random adjacency matrix $\bvar{A}\in\R^{n\times n}$ follows a \textit{Random Dot Product Graph} with latent positions $\{\bvar{x}_i\}_{i=1}^n$ if $\{\bvar{A}_{ij}\}_{i<j}$ are conditionally independent with $\bvar{A}_{ij}|\bvar{x}_i, \bvar{x}_j \sim \text{Bern}(\bvar{x}_i^T\bvar{x}_j)$ for $i<j$.
\end{definition}

We will assume that the latent positions are drawn i.i.d.\ from a distribution $F$ over an appropriate subset of $\R^d$.
In order to capture varying network structure, we extend the RDPG to the multiplex graph setting by applying graph-specific weights to the inner products between each vector in the support of $F$.  
The distribution $F$ induces a space of $d\times d$, diagonal matrices $\mathcal{C}_F$ that weight each inner product while remaining in the unit interval. 
We specify the requisite properties of $F$ and its weighting space $\mathcal{C}_F$ in Definition~\ref{Inner Product Distribution}.

\begin{definition}\label{Inner Product Distribution}
Let $F$ be a distribution over $\R^d$ with the property that for all $\bvar{x}, \bvar{y}\in\text{supp}(F)$ that $\bvar{x}^T\bvar{y}\in[0,1]$. 
If $F$ satisfies these properties we say $F$ is a \textit{d-dimensional inner production distribution}.
We say $F$ induces a \textit{diagonal weighting space}, $\mathcal{C}_F$, where $\mathcal{C}_F$ is given by 
    \begin{align*}
        \mathcal{C}_F = \{\bvar{C}\in \R_{\geq 0}^{d\times d}: \bvar{C}\text{ is diagonal}, \bvar{x}^T\bvar{Cy}\in[0,1],  \forall  \bvar{x},\bvar{y}\in\text{supp}(F)\}.
    \end{align*}
\end{definition}

This leads us to our definition of the ESRPDG model.

\begin{definition}\label{ESRDPG} Let $F$ be a $d$-dimensional inner product distribution such that for $\bvar{y}\sim F$ the second moment matrix $\Delta = \E[\bvar{yy}^T]\in\R^{d\times d}$ is diagonal and full rank.
Let $\bvar{X}_1, \ldots, \bvar{X}_n \overset{i.i.d.}{\sim}F$ and organize these vectors in the rows of the matrix $\bvar{X} = [\bvar{X}_1,\bvar{X}_2, \dotsc, \bvar{X}_n]^T$.
Let $\bvar{C}^{(1)},\ldots, \bvar{C}^{(m)}\in\mathcal{C}_F$ with the property that $\min_{i\in[d]}\max_{g\in[m]}\bvar{C}_{ii}^{(g)}>0$. 
We say that the vertex-aligned, random adjacency matrices $\{\bvar{A}^{(g)}\}_{g=1}^m $ are distributed according to the \textit{Eigen-Scaling Random Dot Product Graph} and write $(\{\bvar{A}^{(g)}\}_{g=1}^m, \bvar{X})\sim \text{ESRDPG}(F, n, \{\bvar{C}^{(g)}\}_{g=1}^m)$ if 
\begin{equation*}
    \prob[\bvar{A}^{(1)}, \bvar{A}^{(2)}, \ldots, \bvar{A}^{(m)}|\bvar{X}] = \prod_{g=1}^m\prod_{i<j}(\bvar{X}_i^T\bvar{C}^{(g)}\bvar{X}_j)^{\bvar{A}_{ij}^{(g)}}(1-\bvar{X}_i^T\bvar{C}^{(g)}\bvar{X}_j)^{1-\bvar{A}_{ij}^{(g)}}. 
\end{equation*}
Under this model, $\{\bvar{A}^{(g)}\}_{g=1}^m$ are conditionally independent given $\bvar{X}$ with $\bvar{A}_{ij}^{(g)}|\bvar{X}\sim \text{Bernoulli}(\bvar{X}_i^T\bvar{C}^{(g)}\bvar{X}_j)$.
\end{definition}

By associating each network with a different weighting matrix in $\mathcal{C}_F$, we can capture variation between networks within the RDPG framework.
Given $\bvar{X}$, the probability of observing an edge between vertex $i$ and vertex $j$ in graph $g$ is given by $\bvar{X}_i^T\bvar{C}^{(g)}\bvar{X}_j$. 
We denote the matrix containing these probabilities as $\bvar{P}^{(g)} = \bvar{X}\bvar{C}^{(g)}\bvar{X}^T$ for each graph $g\in[m]$, so that 
$\bvar{A}^{(g)}|\bvar{X}\sim\text{Bern}(\bvar{P}^{(g)})$.
See Figure \ref{fig:hierarchy_DSRDPG} for a visual illustration of this model.

\begin{figure}[h!]
    \begin{adjustwidth}{-.5in}{-.5in}
    \centering
    \begin{tikzpicture}
    \node[shape=circle,draw=black, fill = red, fill opacity=0.5, text opacity=1] (F) at (0,0) {$\bvar{X}\sim F$};
    
    
    \node[shape=rectangle, draw= yellow, draw opacity = .5, line width = .7mm, align = center] (L1) at (-7,-3.5) {
    $ \bvar{X}\sqrt{\bvar{C}^{(1)}}$\\
    \includegraphics[width = 1.3in, height = 1.3in]{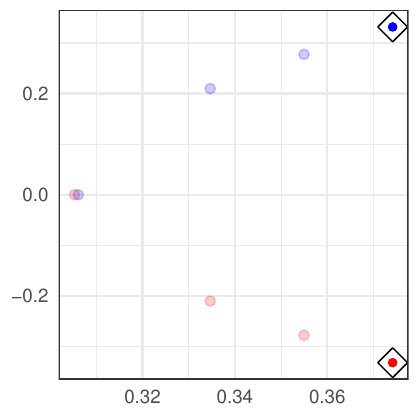}
    };
    \node[shape=rectangle, draw= yellow, draw opacity = .5, line width = .7mm, align = center] (L2) at (-3,-3.5) {
    $ \bvar{X}\sqrt{\bvar{C}^{(2)}}$\\ 
    \includegraphics[width = 1.3in, height = 1.3in]{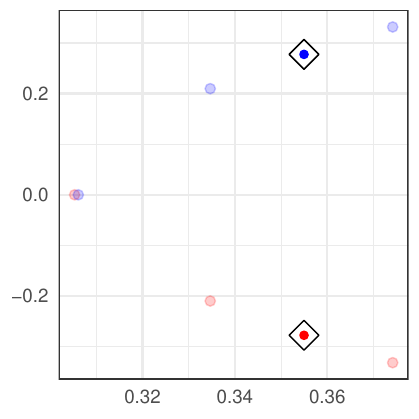}
    };
    \node[shape=rectangle, draw= yellow, draw opacity = .5, line width = .7mm, align = center] (Lm1) at (3,-3.5) {$ \bvar{X}\sqrt{\bvar{C}^{(m-1)}}$\\
    \includegraphics[width = 1.3in, height = 1.3in]{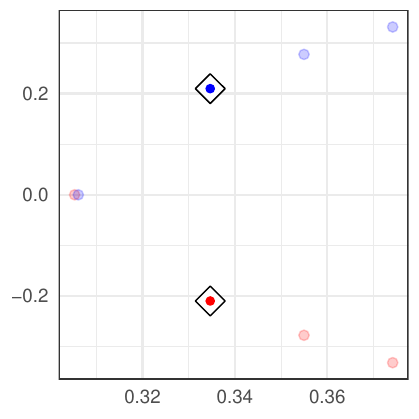}};
    \node[shape=rectangle, draw= yellow, draw opacity = .5, line width = .7mm, align = center] (Lm) at (7,-3.5) {
    $ \bvar{X}\sqrt{\bvar{C}^{(m)}}$\\
    \includegraphics[width = 1.3in, height = 1.3in]{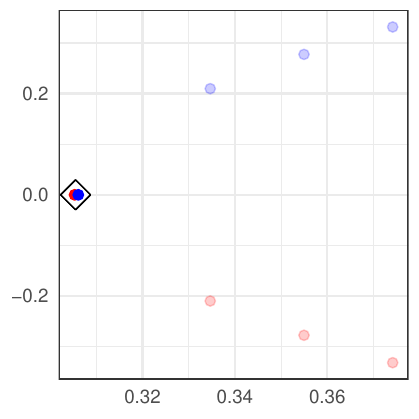}};
    
    \draw [loosely dotted] (L2) -- (Lm1);
    \draw [->] (F) ..controls+(left:10mm)and+(up:30mm)..(L1);
    \draw [->] (F) ..controls+(left:10mm)and+(up:25mm)..(L2);
    \draw [->] (F) ..controls+(right:10mm)and+(up:25mm)..(Lm1);
    \draw [->] (F) ..controls+(right:10mm)and+(up:30mm)..(Lm);
    
    \node[shape=rectangle, draw= green, draw opacity = .5, line width = .7mm, align = center] (A1) at (-7,-8) {\includegraphics[width = 1.3in, height = 1.5in]{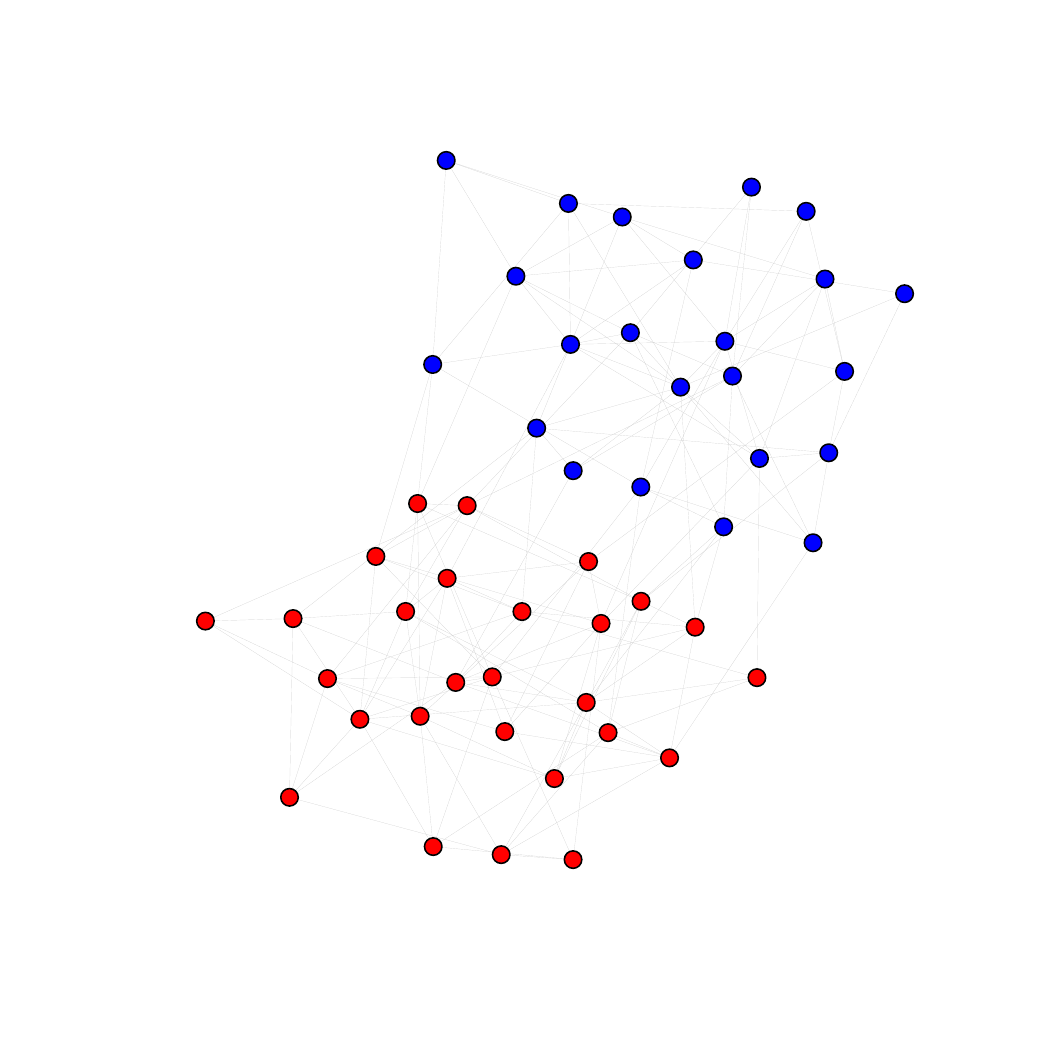}};
    \node[shape=rectangle, draw= green, draw opacity = .5, line width = .7mm, align = center] (A2) at (-3,-8) {\includegraphics[width = 1.3in, height = 1.5in]{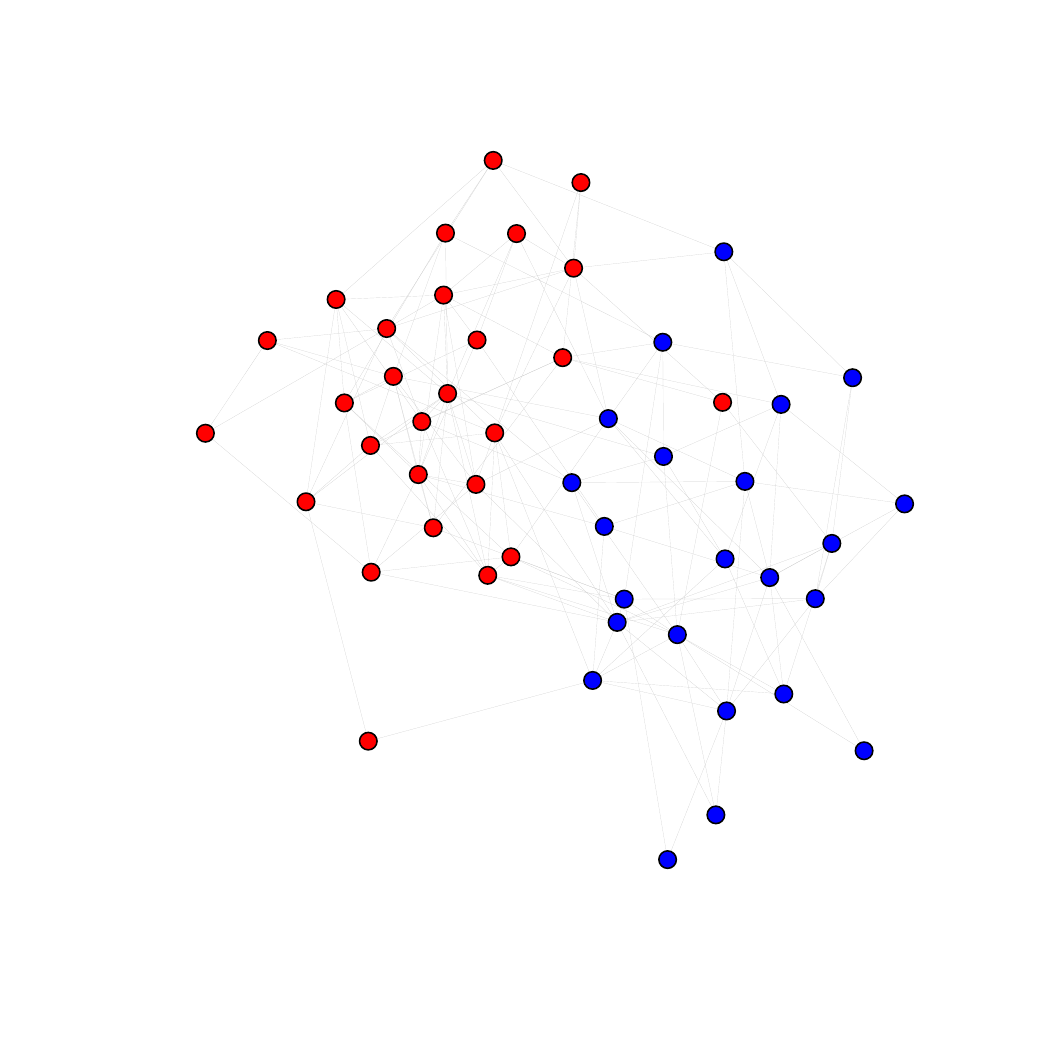}};
    \node[shape=rectangle, draw= green, draw opacity = .5, line width = .7mm, align = center] (Am1) at (3,-8) {\includegraphics[width = 1.4in, height = 1.5in]{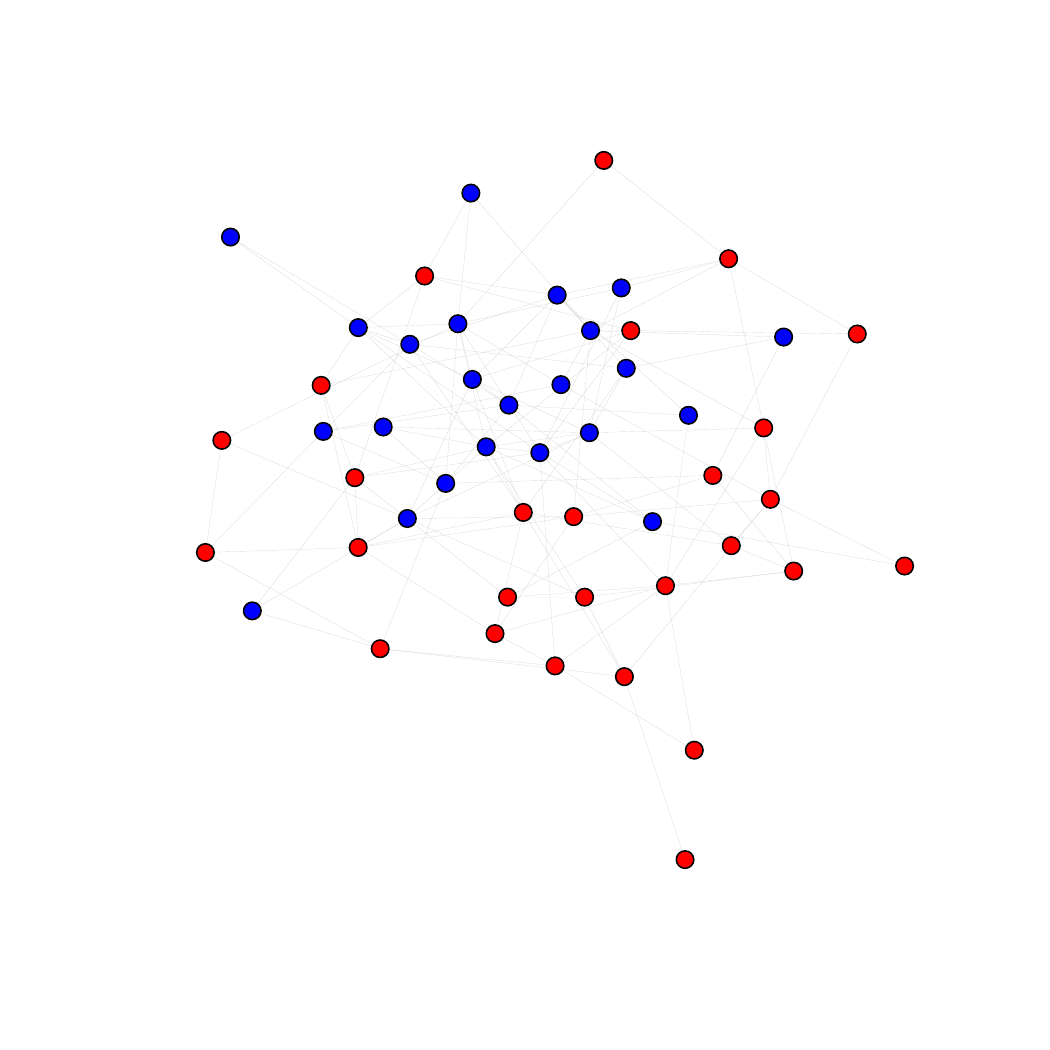}};
    \node[shape=rectangle, draw= green, draw opacity = .5, line width = .7mm, align = center] (Am) at (7,-8) {\includegraphics[width = 1.3in, height = 1.5in]{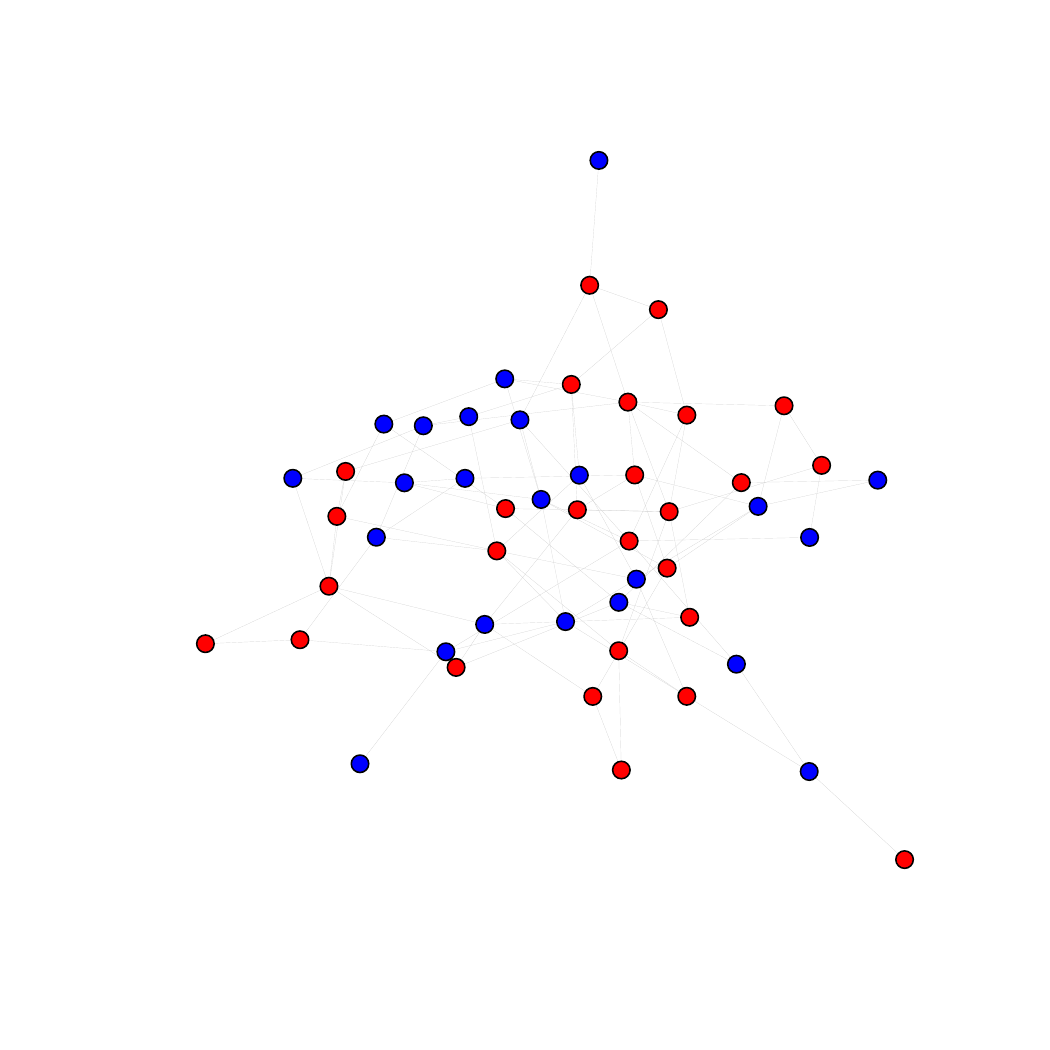}};
    
    \draw [loosely dotted] (A2) -- (Am1);
    \draw [->] (L1)--(A1);
    \draw [->] (L2)--(A2);
    \draw [->] (Lm1)--(Am1);
    \draw [->] (Lm)--(Am);
    
\end{tikzpicture}
    \caption{An illustration of the ESRPDG model's hierarchical structure. 
    Latent positions are drawn from a common distribution $F$. 
    From here each $\bvar{P}^{(g)}$ is defined as a function of the latent positions $\bvar{X}$ and weighting matrix $\bvar{C}^{(g)}$.  
    These scaled latent positions (yellow boxes) change the probability structure for each network.
    The random adjacency matrices (green boxes) are then drawn from $\bvar{A}^{(g)}\sim \text{Bern}(\bvar{P}^{(g)})$.}
    \label{fig:hierarchy_DSRDPG}
\end{adjustwidth}
\end{figure}
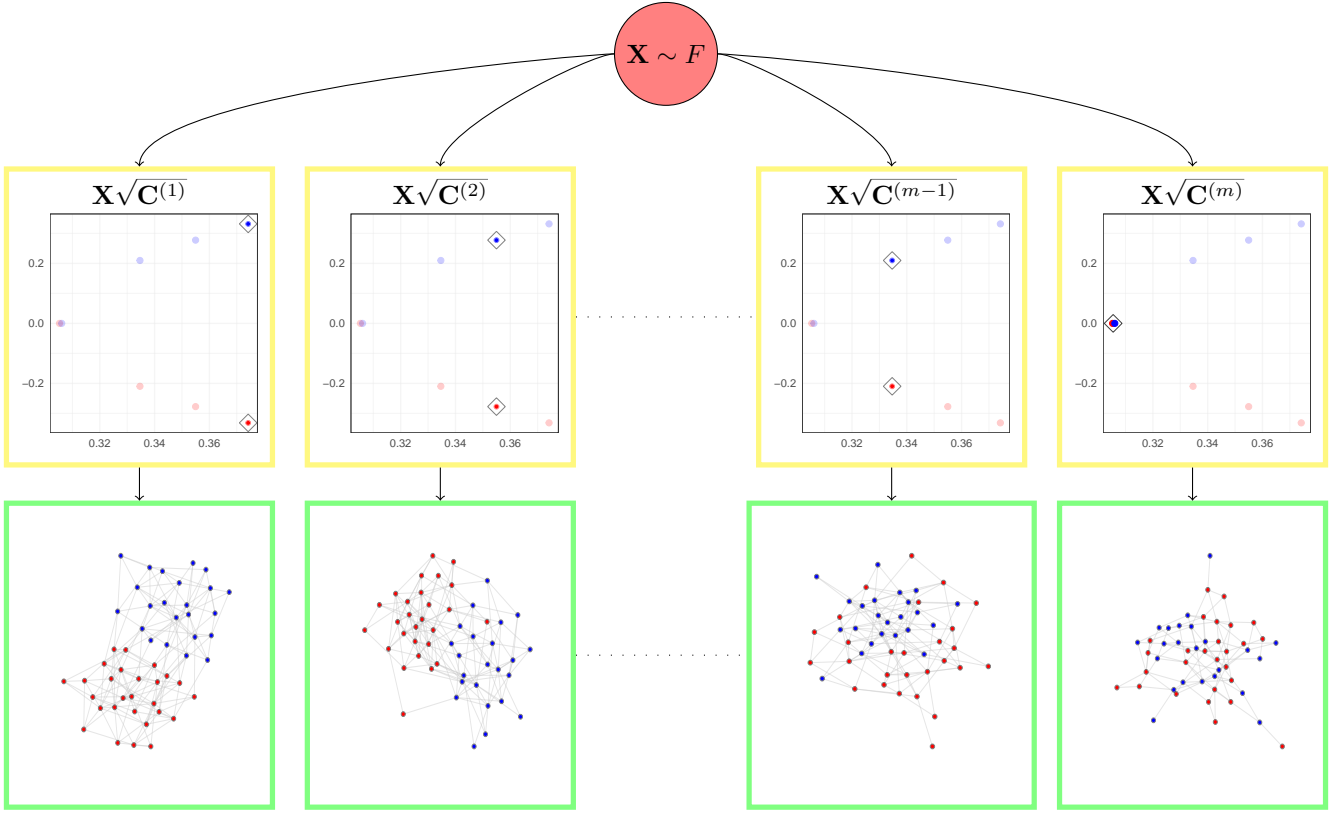

Heuristically, each dimension in the latent space can be interpreted as a principle direction capturing important features that govern the connectivity structure in the system. 
The weighting matrices can be seen as alternating the dot product that governs connectivity in each layer of the multiplex. 
Said another way, the matrix square root of the weighting matrices $\sqrt{\bvar{C}^{(g)}}$ can be interpreted as mapping the latent positions to new points in $\R^d$ where the standard dot product is applied. 

\begin{Remark}[Model Assumptions]\label{Remark:Model Asummptions}
In the definition of the ESRDPG, we impose assumptions on the distribution $F$ and weighting matrices $\{\bvar{C}^{(g)}\}_{g=1}^m$. 
Similar assumptions on $F$ are made in previous analysis \parencite{EigenCLT, OmniCLT} but have somewhat different implications here.
The assumptions $\mathrm{rank}(\Delta) = d$ and $\min_{i\in[d]}\max_{g\in[m]}\bvar{C}_{ii}^{(g)}>0$ fixes the dimension of the latent space, $d$.
Assuming $\Delta$ is diagonal enables analytic computation of the bias presented in Theorem \ref{Main Bias Theorem}.
Finally, assuming $\bvar{C}_{ii}^{(g)}\geq 0$ for all $i\in[d]$ and $g\in[m]$ is purely for notational convenience.
The results presented in Section~\ref{Section: Main Results} hold provided there does not exist $i\in[d]$ such that $\bvar{C}_{ii}^{(g)} = c \leq  0$ for all $g\in[m]$.
This condition ensures that the conditional expectation of the omnibus matrix $\tilde{\bvar{P}} = \E[\tilde{\bvar{A}}|\bvar{X}]$ has $d$ positive eigenvalues which allows us to focus on the positive definite part of $\tilde{\bvar{A}}$.
Relaxing the condition $\bvar{C}_{ii}^{(g)}\geq 0$ allows the ESRDPG to capture disassortative network structures and can be viewed as an extension of the Generalized Random Dot Product Graph (GRDPG) \parencite{rubindelanchy2017}, where considering both positive and negative eigenvalues is necessary.
We elect to exclude weighting matrices with negative values as it significantly complicates the presentation of the results without adding demonstrably new insights.
\qed
\end{Remark}

We note that the ESRPDG includes settings where $\bvar{C}_{ii}^{(g)} = 0$.
Therefore, if certain principle directions in graph $g$ are inconsequential in the connectivity structure of graph $k$, the ESRDPG can capture this relationship within a common probability model. 
By adjusting $d$ appropriately, the ESRDPG can capture a wide collection of conditionally independent RDPGs of differing dimension.
We explore some of these distributions in the experiments section of this work.

The ESRDPG is similar to other multiple random graph models that extend the RDPG to multiplex network data. 
Models that extend this paradigm include the Multiple Random Dot Product Graph (MRDPG) of \textcite{nielsen2018}, the Multiple Random Eigen Graphs (MREG) of \textcite{wang2017}, the Common Subspace Independent-Edge (COSIE) model of \textcite{arroyo2019}, and the Multilayer Random Dot Product Graph (Multilayer RDPG) of \textcite{jones2020multilayer}.
We summarize these models in Table~\ref{Submodel_Table}.

\begin{table}
    \centering
     \begin{tabular}{|c|c|c|}
    \hline 
    Model & Latent Positions $\bvar{X}^{(g)}$ & Source\\
    \hline\hline
    \rule{0pt}{2ex} MRDPG & $\bvar{U}\sqrt{\bvar{D}^{(g)}}$ & \textcite{nielsen2018} \\
    \rule{0pt}{2ex} MREG &  $\bvar{V}\sqrt{\bvar{D}^{(g)}}$ & \textcite{wang2017}\\
    \rule{0pt}{2ex} COSIE &  $\bvar{U}\sqrt{\bvar{R}^{(g)}}$ & \textcite{arroyo2019}\\
    \rule{0pt}{2ex} Multilayer RDPG & $\bvar{X}\sqrt{\bvar{R}^{(g)}}$ & \textcite{jones2020multilayer}\\  
    \hline
    \end{tabular}
    \caption{A summary of the multiplex network probability models that extend the RDPG. 
    Global structure is encoded in the matrices $\bvar{U}$, $\bvar{V}$, or $\bvar{X}$ where $\bvar{U}$ is matrix with orthogonal columns, $\bvar{V}$ has columns with unit Euclidean norm, and $\bvar{X}$ is a latent position matrix. 
    Layer variation is then encoded in the matrix square root of either the diagonal matrix $\bvar{D}^{(g)}$ or p.s.d.\ matrix $\bvar{R}^{(g)}$.}
    \label{Submodel_Table}
\end{table}

Notice that each method assumes the matrices $\{\bvar{P}^{(g)}\}_{g=1}^m$ share common structure through a vertex specific latent structure.
Each layer then has a different inner product weighted by the square root of a weighting matrix. 
The MRDPG and MREG models are similar to the ESRDPG as they adopt a latent position structure while layer-variation is captured by diagonal scaling of these positions. 
The Multilayer RDPG and the COSIE models adopt a similar latent position structure while layer-variation is instead captured by a linear transformation of the latent positions.
The MRDPG and MREG differ from the ESRDPG in the sampling procedure of the latent positions.
The ESRDPG and Multilayer RDPG both assume the rows of $\bvar{X}$ are drawn i.i.d.\ from a common distribution while the MRDPG, MREG, and COSIE models make no such assumption about the rows of $\bvar{U}$ or $\bvar{V}$.
These models instead, implicitly, make finite sample assumptions about the matrices $\{\bvar{P}^{(g)}\}_{g=1}^m$.
Of these models, the ESRDPG share traits with the MRDPG and MREG, as these methods capture layer-variation through diagonal scaling of the latent positions, and the Multilayer RDPG, as both methods assume the latent positions are exchangeable. 


\begin{Remark}[Identifiability]\label{Identifiability}
In most latent position models, the model parameters are typically identifiable only within a larger equivalence class. 
Of note, the latent positions of the RDPG are only identifiable up to an orthogonal transformation as the edge probabilities are invariant under such transformation.
Similar issues arise for the ESRDPG though the common structure among the graphs somewhat limits the non-identifiability.
Generally, the latent positions and weighting matrices are only identifiable within the equivalence class 
\begin{align*}
    (\mathcal{X}, \mathcal{C}) &= \{(\bvar{XM}, \{\bvar{M}^{-1}\bvar{C}^{(g)}\bvar{M}^{-T}\}_{g=1}^m): \bvar{M}\in\mathcal{M}, \bvar{X}\sim F\}.
\end{align*}
where $\mathcal{M}$ is the subset of $d\times d$ invertible matrices such that $\bvar{M}^T\Delta\bvar{M}$ and $\{\bvar{M}^{-1}\bvar{C}^{(g)}\bvar{M}^{-T}\}_{g=1}^m$ are diagonal and can be characterized more precisely by considering the co-multiplicities of the diagonal values of $\Delta$ and $\{\Delta\bvar{C}^{(g)}\}_{g=1}^m$.
We circumvent several of these identifiability issues by instead regarding the products $\bvar{X}\sqrt{\bvar{C}^{(g)}}$ for all $g\in[m]$ as the parameters of interest. 
By changing the estimand we introduce a new, but more familiar, form of indentifiability issue. 
Namely, for all $g\in[m]$ for any $\bvar{W}\in\mathcal{O}^{(d)}$
\begin{align*}
    \bvar{P}^{(g)} = (\bvar{X}\sqrt{\bvar{C}^{(g)}})(\bvar{X}\sqrt{\bvar{C}^{(g)}})^T = (\bvar{X}\sqrt{\bvar{C}^{(g)}}\bvar{W})(\bvar{X}\sqrt{\bvar{C}^{(g)}}\bvar{W})^T.
\end{align*}
Letting $\bvar{L} = [(\bvar{X}\sqrt{\bvar{C}^{(1)}})^T\dots (\bvar{X}\sqrt{\bvar{C}^{(m)}})^T]^T\in\R^{nm\times d}$ be the block matrix with $m$, $n\times d$ blocks given by $\bvar{X}\sqrt{\bvar{C}^{(g)}}$, the weighted latent positions are then only identifiable within the equivalence class $\mathcal{L}$ given by
\begin{align*}
    \mathcal{L} &= \{\bvar{LW}: \bvar{W}\in\mathcal{W}\}\\
    \mathcal{W} &= \{\bvar{W}\in\mathcal{O}^{(d)}: \bvar{W}^T\bvar{C}^{(g)}\Delta\bvar{W}\text{ is diagonal for all $g\in[m]$}\}.
\end{align*}
We note the size of $\mathcal{W}\subseteq\mathcal{O}^{(d)}$ is determined by the matrices $\bvar{C}^{(g)}\Delta$.
Assuming $\Delta = \delta\bvar{I}$ and $\bvar{C}^{(g)} = \bvar{I}$ for all $g\in[m]$, $\mathcal{W} = \mathcal{O}^{(d)}$.
However, if $\bvar{C}^{(g)}\Delta$ is full rank with unique elements, then $\mathcal{W} = \{\text{diag}(\bvar{w}): \bvar{w}\in\{\pm 1\}^d\}$ and hence finite. 
\qed
\end{Remark}

Given these identifiability considerations, we now focus on analyzing the omnibus embedding's ability to estimate the rows of $\bvar{L}$ up to orthogonal rotation.
We conclude this section with organizing the notation used throughout the course of this paper in Table~\ref{tab:notation_table}.

\begin{table}[h!]
    \footnotesize
    \centering
    \begin{tabular}{|l|p{.6\textwidth}|}
    \hline
     \textbf{Symbol} & \textbf{Definition}\\
     \hline 
      $(n, m, d)$ & Number of vertices, graphs, and embedding dimension, respectively\\
      $\bvar{M}_i = (\bvar{M}_{i\cdot})^T$ & The $i$-th row of $\bvar{M}$, transposed\\
      $\{E_n\}_{n=1}^{\infty}$  occurs w.h.p. & Events $\{E_n\}_{n=1}^{\infty}$ occurs with high probability if $\prob(E_n^C)\leq (nm)^{-2}$\\
      $\|\cdot\|$\quad$\|\cdot\|_F$ & Euclidean/spectral norm for vectors/matrices, the Frobenius norm\\
      $\mathcal{O}^{(d)}$ & The set of $d\times d$ orthogonal matrices\\
      $C>0$ & A constant that may vary from line to line\\
     $\lambda_i(\bvar{M})$ & The $i$-th largest eigenvalue of $\bvar{M}$\\
     $\bvar{X}\in\R^{n\times d}$, $\bvar{C}^{(g)}\in\R^{d\times d}$ & The latent position matrix and the weighting matrix for each $g\in[m]$\\
     $\bvar{A}^{(g)},\bvar{P}^{(g)}\in\R^{n\times n}$ & The adjacency and its (conditional) expectation $\bvar{P}^{(g)} = \bvar{XC}^{(g)}\bvar{X}^T$\\
     $\tilde{\bvar{A}}, \tilde{\bvar{P}}\in\R^{nm\times nm}$ & The omnibus matrix of $\{\bvar{A}^{(g)}\}_{g=1}^m$, $\{\bvar{P}^{(g)}\}_{g=1}^m$, respectively\\
     $\hat{\bvar{X}}^{(g)}_i\in\R^{n\times d}$ & Omnibus latent position estimate for vertex $i\in[n]$ in graph $g\in[m]$\\
     $\bar{\bvar{X}}_i = m^{-1}\sum_{g=1}^m\hat{\bvar{X}}_i^{(g)}\in\R^{n\times d}$ & The average omnibus latent position estimate for vertex $i\in[n]$ \\
     $\bvar{C}_m = m^{-1}\sum_{g=1}^m\bvar{C}^{(g)2}$ & The sums of squares of the weighting matrices $\{\bvar{C}^{(g)}\}_{g=1}^m$\\
     $\bvar{S}^{(g)} = 2^{-1}[\bvar{C}^{(g)}\bvar{C}_m^{-1/4} + \bvar{C}_m^{1/4}]$ & The scaling matrix for graph $g\in[m]$\\
     $\bvar{L}\in\R^{nm\times d}$ & The weighted latent position matrix with blocks $\bvar{L} = [\bvar{X}\sqrt{\bvar{C}^{(g)}}]_{g=1}^m$\\
     $\bvar{L}_S\in\R^{nm\times d}$ & The scaled latent position matrix with blocks $\bvar{L}_S = [\bvar{X}\bvar{S}^{(g)}]_{g=1}^m$\\
        \hline  
    \end{tabular}
    \caption{Notation used consistently throughout the presentation of the text. }
    \label{tab:notation_table}
\end{table}

\section{Theoretical Results}\label{Section: Main Results}
In this section we present the main theoretical analysis of the omnibus embedding under the ESRDPG.
We begin by stating Theorem~\ref{Main Bias Theorem} and Theorem~\ref{CLT} which establish the first and second moment properties of the omnibus embedding under the ESRDPG. 
We highlight these results in a simulation study that shows the eminent features of the analysis. 
We then aim to interpret the ramifications of these theorems by completing a bias-variance analysis of similar estimators and stating corollaries useful in statistical applications to follow.

\subsection{Main Results}\label{Subsec:Main_Results}

In this section we reveal a bias-variance tradeoff in the estimated latent positions produced by the omnibus embedding.
Let $\hat{\bvar{L}} = \text{Omni}(\{\bvar{A}^{(g)}\}_{g=1}^m, d)$ and $\bvar{L}_S$ and $\bvar{L}$ be as in Table~\ref{tab:notation_table}.
Then, for $\bvar{W}_1\in\mathcal{O}^{(d)}$, $\hat{\bvar{L}}$ emits the first order decomposition
\begin{equation}
    \hat{\bvar{L}}\bvar{W}_1 - \bvar{L} = (\bvar{L}_S- \bvar{L}) + \bvar{R}.
\end{equation}
In Theorem~\ref{Main Bias Theorem} we establish that the omnibus embedding provides biased latent position estimates and this bias is captured by $(\bvar{L}_S- \bvar{L})$.
Theorem~\ref{Main Bias Theorem} also provides a uniform bound on the residual term $\bvar{R}$, supporting this result. 

\begin{theorem}\label{Main Bias Theorem}
Suppose that $(\{\bvar{A}^{(g)}\}_{g=1}^m, \bvar{X})\sim \text{ESRDPG}(F,n,\{\bvar{C}^{(g)}\}_{g=1}^m)$ for some $d$-dimensional inner product distribution $F$. 
Let $\hat{\bvar{L}} = \text{Omni}\left(\{\bvar{A}^{(g)}\}_{g=1}^m, d\right)$ and $\bvar{L}$ be given as above. 
Let $h = n(g-1) + i$ for some $g\in[m]$ and $i\in[n]$.
Then there exists a sequence of orthogonal matrices $\{\tilde{\bvar{W}}_n\}_{n=1}^{\infty}$ depending on $\tilde{\bvar{A}}$ and $\tilde{\bvar{P}}$ such that
\begin{equation}
    (\hat{\bvar{L}}\tilde{\bvar{W}}_n - \bvar{L})_h = (\bvar{S}^{(g)} - \sqrt{\bvar{C}^{(g)}})\bvar{X}_i + \bvar{R}_h
\end{equation}
where $\bvar{S}^{(g)} = 2^{-1}[\bvar{C}^{(g)}\bvar{C}_m^{-1/4} +\bvar{C}_m^{1/4}]$
for $\bvar{C}_m = m^{-1}\sum_{g=1}^m\bvar{C}^{(g)2}$.
Moreover, $\bvar{R}_h$ is a residual term that with high probability satisfies
\begin{equation}
    \max_{h\in[nm]}\|\bvar{R}_h\|_2 \leq Cm^{3/2}\frac{\log nm}{\sqrt{n}}. 
\end{equation}
\end{theorem}

\begin{proof}
The proof can be found in Appendix~\ref{First Moment Appendix}. 
\end{proof}

The essence of this result is that the estimated latent positions $\hat{\bvar{L}}$ do not concentrate around $\bvar{L} = [\bvar{X}\sqrt{\bvar{C}^{(g)}}]_{g=1}^m$ but instead the scaled latent positions $\bvar{L}_S = [\bvar{X}\bvar{S}^{(g)}]_{g=1}^m$.
Therefore, the omnibus embedding introduces a row-wise asymptotic bias of $(\bvar{S}^{(g)} - \sqrt{\bvar{C}^{(g)}})\bvar{X}_i$.
The second portion of the theorem provides a uniform rate for this concentration.
The rate $O(m^{3/2}n^{-1/2}\log nm)$ is reminiscent of that given in \textcite{OmniCLT} with an additional factor of $m$ due to permitting $\bvar{C}_{ii}^{(g)} = 0$ for all but one $g\in[m]$.
If the number of nonzero $\{\bvar{C}_{ii}^{(g)}\}_{g=1}^m$ grows like $\Theta(m)$ then the bound is instead $O(m^{1/2}n^{-1/2}\log nm)$ consistent with the rate in \textcite{OmniCLT}.

The scaled latent positions $\bvar{X}\bvar{S}^{(g)} = 2^{-1}(\bvar{X}\bvar{C}^{(g)}\bvar{C}_m^{-1/4} + \bvar{X}\bvar{C}_{m}^{1/4})$
can be interpreted as an average between the term $\bvar{X}\bvar{C}^{(g)}\bvar{C}_m^{-1/4}$ which captures layer specific variation and the term $\bvar{XC}_m^{1/4}$ which is common for all $g\in[m]$. 
The term $\bvar{XC}_m^{1/4}$ can be interpreted as a regularization term which captures average graph behavior.
In particular, notice
\begin{align*}
    (\bvar{XC}_m^{1/4})(\bvar{XC}_m^{1/4})^T = \bvar{X}\sqrt{\frac{1}{m}\sum_{g=1}^m(\bvar{C}^{(g)})^2}\bvar{X}^T
\end{align*}
which implies $\bvar{XC}_m^{1/4}$ is the adjacency spectral embedding of a graph whose weighting matrix is given by $\bvar{C}_m^{1/2} = \sqrt{m^{-1}\sum_{g=1}^m(\bvar{C}^{(g)})^2}$ which is an $\ell_2$ average of the weighting matrices $\{\bvar{C}^{(g)}\}_{g=1}^m$.
Therefore, the omnibus embedding implicitly regularizes its latent position estimates and this regularization manifests in the omnibus latent position estimates reflecting both layer specific variation and structures shared across all graphs $g\in[m]$.

In order to more precisely characterize the residual term $\bvar{R}$, we look to establish its distributional properties. 
Heuristically, we anticipate the omnibus embedding of $\{\bvar{A}^{(g)}\}_{g=1}^m$ will concentrate around the omnibus embedding of $\{\bvar{P}^{(g)}\}_{g=1}^m$, conditional on $\bvar{X}$.
As each $\bvar{P}^{(g)}$ is a function of the random matrix $\bvar{X}$, the omnibus embedding of $\{\bvar{P}^{(g)}\}_{g=1}^m$ is itself random for finite $n$.
We anticipate the convergence of the omnibus embedding of $\{\bvar{P}^{(g)}\}_{g=1}^m$ to depend on the convergence of the sample second moment matrix $n^{-1}\bvar{X}^T\bvar{X}$ to its limit $\Delta = \E[\bvar{X}_1\bvar{X}_1^T]$.
These two sources of variability inspire the decomposition of $\bvar{R}$ in  equation~\eqref{eq:second_moment_decomp}.
Let $\bvar{Z} = \text{Omni}(\{\bvar{P}^{(g)}\}_{g=1}^m, d)$ and let  $\bvar{W}_2\in\mathcal{O}^{(d)}$.
Then we further expand $\bvar{R}$ as 
\begin{equation}\label{eq:second_moment_decomp}
    \bvar{R} = (\hat{\bvar{L}}\bvar{W}_1 - \bvar{Z}\bvar{W}_2) + (\bvar{Z}\bvar{W}_2 - \bvar{L}_S) =: \bvar{N} + \bvar{M}. 
\end{equation}
$\bvar{N}$ describes variation between the adjacencies and their conditional expectation $\{\bvar{A}^{(g)} - \bvar{P}^{(g)}\}_{g=1}^m$ while $\bvar{M}$ describes deviation between the sample second moment matrix and its expectation $n^{-1}\bvar{X}^T\bvar{X} - \Delta$. 
As a result $\bvar{N}$ and $\bvar{M}$ are conditionally independent given $\bvar{X}$. 
We first establish the distributional properties of $\bvar{N}$ in Lemma~\ref{Lemma:N-Dist} which shows that the scaled rows of $\sqrt{n}\bvar{N}$ converge in distribution to a mixture of mean zero normal random variables.

\begin{Lemma}\label{Lemma:N-Dist}
Suppose that $(\{\bvar{A}^{(g)}\}_{g=1}^m, \bvar{X})\sim \text{ESRDPG}(F,n,\{\bvar{C}^{(g)}\}_{g=1}^m)$ for some $d$-dimensional inner product distribution $F$.  
Let $\Delta_S = \sum_{g=1}^m\bvar{S}^{(g)2}\Delta$ where $\Delta = \E[\bvar{X}_1\bvar{X}_1^T]$ and $h = n(g-1) + i$ where $g\in[m]$ and $i\in[n]$.
Then in the context of Theorem~\ref{Main Bias Theorem}, we have
\begin{align*}
    \lim_{n\to\infty}\prob\left[\sqrt{n}\bvar{N}_h \leq \bvar{x}\right] &= \int_{\text{supp}(F)}\Phi(\bvar{x}; \bvar{0}, \Sigma_g^{(N)}(\bvar{y}))dF(\bvar{y})
\end{align*}
where $\Phi(\bvar{x}; \mu, \Sigma)$ is the multivariate normal cumulative distribution function with mean $\mu$ and covariance matrix $\Sigma$. 
Moreover, the covariance $\Sigma_g(\bvar{y})$ can be written explicitly as  
\begin{align*}
    \Sigma_g^{(N)}(\bvar{y}) &=\frac{1}{4} \Delta_S^{-1}\left[(\bvar{S}^{(g)} + m\bar{\bvar{S}})\tilde{\Sigma}_g(\bvar{y})(\bvar{S}^{(g)} + m\bar{\bvar{S}}) + \sum_{k\neq g}\bvar{S}^{(k)}\tilde{\Sigma}_k(\bvar{y})\bvar{S}^{(k)}\right]\Delta_S^{-1}
\end{align*}
where $\tilde{\Sigma}_{\ell}(\bvar{y}) = \E\left[(\bvar{y}^T\bvar{C}^{(\ell)}\bvar{X}_j - (\bvar{y}^T\bvar{C}^{(\ell)}\bvar{X}_j)^2)\bvar{X}_j\bvar{X}_j^T\right]$.
\end{Lemma}
\begin{proof}
The proof can be found in the Appendix \ref{Second Moment Appendix}. 
\end{proof}

Similar results were achieved in \textcite{EigenCLT} and \textcite{OmniCLT} using a combination of perturbation arguments. 
These arguments largely rely on studying differences in eigenstructure between $\tilde{\bvar{A}}$ and its expectation $\tilde{\bvar{P}} = \E[\tilde{\bvar{A}}|\bvar{X}]$.
In \textcite{OmniCLT}, under the assumption that the $m$ graphs have the same expectation, $\tilde{\bvar{P}}$ is p.s.d.\ and the rank of this matrix was equal to that of the latent space, i.e. $\text{rank}(\tilde{\bvar{P}}) = d$.
Under the ESRDPG, $\tilde{\bvar{P}}$ is indefinite with $q\in[0, d]$ negative eigenvalues which complicates the analysis. 
Utilizing analysis techniques introduced in \textcite{rubindelanchy2017} and extended in \textcite{agterberg2020nonparametric} and \textcite{chung2021valid}, we are able to relate the eigen-structure of these matrices to the model parameters which facilitate the analysis. 

Establishing the distributional properties of $\bvar{M}$ is a more complicated undertaking. 
Arguments presented Appendix~\ref{Second Moment Appendix} reveal that the rows of $\bvar{M}$ can be written as $\bvar{M}^{(g)}\bvar{X}_i$, the elements of $\sqrt{n}\bvar{M}^{(g)}$ are asymptotically normally distributed with mean zero, and the covariance of $\sqrt{n}\text{vec}(\bvar{M}^{(g)})$ is degenerate with rank $d(d+1)/2$.
As a result, the scaled rows $\sqrt{n}\bvar{M}^{(g)}\bvar{X}_i$ are asymptotically normally distributed with mean zero constrained to a subspace of $\R^d$.
When $d = 1$, $\bvar{M}$ vanishes, for $d = 2$, $\bvar{M}$ is constrained to a one dimensional subspace of $\R^d$, and for $d >2$, $\bvar{M}$ is full rank.
We explore this subspace restriction in Example~\ref{ex:two-group-running-ex}.

As the rows of $\bvar{M}$ can be written as the linear transformation $\bvar{M}^{(g)}\bvar{X}_i$, $\hat{\bvar{L}}\bvar{W}_1 - \bvar{L}$ can be decomposed as 
\begin{align*}
    (\hat{\bvar{L}}\bvar{W}_1 - \bvar{L})_h = (\bvar{S}^{(g)} + \bvar{M}^{(g)} - \sqrt{\bvar{C}^{(g)}})\bvar{X}_i + \bvar{N}_h.
\end{align*}
This decomposition of can be thought of in two parts.
First, $(\bvar{S}^{(g)} + \bvar{M}^{(g)})\bvar{X}_i$ can be thought of as a noisy bias term where $\bvar{M}^{(g)}$ introduces a mean zero translation of the biased scaling $\bvar{S}^{(g)}$. 
As this noise is a linear transformation in the direction of $\bvar{X}_i$ is it often immaterial for subsequent inference (see Example~\ref{ex:two-group-running-ex}).
Second, the rows of $\bvar{N}$ converge to a mixture of mean zero random variables. 
Combining these results leads to our distributional characterization of $\bvar{R}$ in Theorem~\ref{CLT}. 

\begin{theorem}\label{CLT}
Suppose that $(\{\bvar{A}^{(g)}\}_{g=1}^m, \bvar{X})\sim \text{ESRDPG}(F,n,\{\bvar{C}^{(g)}\}_{g=1}^m)$ for some $d$-dimensional inner product distribution $F$.  
Then in the context of Theorem~\ref{Main Bias Theorem}, for some $h = n(g-1) + i$ where $i\in[n]$ and $g\in[m]$, we have 
\begin{align*}
    \lim_{n\to\infty}\prob\left[\sqrt{n}\bvar{R}_h \leq \bvar{x}\right] &= \int_{\text{supp}(F)}\Phi(\bvar{x}; \bvar{0}, \Sigma_g(\bvar{y}))dF(\bvar{y})
\end{align*}
where $\Phi(\bvar{x}; \mu, \Sigma)$ is the multivariate normal cumulative distribution function with mean $\mu$ and covariance matrix $\Sigma$. 
Moreover, the covariance $\Sigma_g(\bvar{y})$ can be decomposed as 
\begin{align*}
 \Sigma_g(\bvar{y}) =  \Sigma_g^{(N)}(\bvar{y}) + \Sigma_g^{(M)}(\bvar{y}) + \Sigma_g^{(N,M)}(\bvar{y})
\end{align*}
where $\Sigma_g^{(N)}(\bvar{y})$ is given in Lemma~\ref{Lemma:N-Dist}, $\Sigma^{(M)}(\bvar{y})$ is attributable to the variance of $\bvar{M}_h$, and $\Sigma^{(N, M)}(\bvar{y})$ is attributable to the covariance of $(\bvar{M}_h, \bvar{N}_h)$.  
\end{theorem}
\begin{proof}
The proof can be found in Appendix \ref{Second Moment Appendix}. 
\end{proof}

Theorem~\ref{CLT} establishes that the rows of omnibus embedding when centered by the bias term given in Theorem~\ref{Main Bias Theorem} converges in distribution to a mixture of mean zero normal random variables. 
Moreover, the variance of this random variable, while complicated, is dominated by an explicit function of the weighting matrices $\{\bvar{C}^{(g)}\}_{g=1}^m$ and the second moment matrix $\Delta$.
While we do not explicitly derive the variance of the rows of $\bvar{M}$, as we explore in Example~\ref{ex:two-group-running-ex}, the variance $\Sigma_g(\bvar{y})$ is dominated by the $\Sigma_g^{(N)}(\bvar{y})$ which we explicitly state in Lemma~\ref{Lemma:N-Dist}.

In statistical applications it will often be necessary to consider linear combinations of the rows of $\hat{\bvar{L}}$.
We now provide two corollaries that characterize the asymptotic joint distribution of the rows of the omnibus embedding.
Corollary~\ref{Cor:Asym_Joint_Distribution_Same_Vertex} specifies the asymptotic joint distribution of two rows $\hat{\bvar{L}}$ corresponding to the same vertex in different graphs and can be easily extended to any finite collection of rows corresponding to the same vertex.

\begin{Corollary}\label{Cor:Asym_Joint_Distribution_Same_Vertex}
Let $r_g = i + n(g-1)$ and $r_k = i + n(k - 1)$ for $i\in[n]$ and $g,k\in[m]$, $g\neq k$ denote rows corresponding to the same node.
Define the vector
$\bvar{V} = ((\bvar{N} + \bvar{M})_{r_g}^T\hspace{.5em}(\bvar{N} + \bvar{M})_{r_k}^T \in \R^{2d}$ and let $\bvar{v}\in \R^{2d}$.
Then in the context of Theorem \ref{CLT} we have
\begin{align*}
    \lim_{n\to\infty}\prob\left[\sqrt{n}\bvar{V} \leq \bvar{v}\right]  
    = \int_{\text{supp}(F)}\Phi(\bvar{v}; \bvar{0}, \Omega_{gk}(\bvar{y}))dF(\bvar{y})
\end{align*}
where $\Omega_{gk}(\bvar{y})\in \R^{2d\times 2d}$ can be decomposed as $\Omega_{gk}(\bvar{y}) = \Omega_{gk}^{(N)}(\bvar{y}) + \Omega_{gk}^{(M)}(\bvar{y}) + \Omega_{gk}^{(N, M)}(\bvar{y})$ and $\Omega_{gk}^{(N)}(\bvar{y})$ can be written explicitly as  
\begin{align*}
    \Omega_{gk}^{(N)}(\bvar{y}) = \begin{bmatrix}
    \Sigma_g^{(N)}(\bvar{y}) &  \Sigma_{gk}^{(N)}(\bvar{y})\\
    \Sigma_{kg}^{(N)}(\bvar{y}) & \Sigma_k^{(N)}(\bvar{y})
    \end{bmatrix} 
\end{align*}
where $\Sigma_g^{(N)}(\bvar{y})$ 
is given in Theorem~\ref{CLT} and $\Sigma^{(N)}_{gk}(\bvar{y}) = \Sigma^{(N)}_{kg}(\bvar{y})^T$, which can be written explicitly as 
\begin{align*}
    \Sigma_{gk}^{(N)}(\bvar{y}) = \frac{1}{4}\Delta_S^{-1}\big[&(\bvar{S}^{(g)} + m\bar{\bvar{S}})\tilde{\Sigma}_g(\bvar{y})\bvar{S}^{(g)} + \bvar{S}^{(k)}\tilde{\Sigma}_k(\bvar{y})(\bvar{S}^{(k)} + m\bar{\bvar{S}})\\ 
    &+ \sum_{\ell\neq g,k}\bvar{S}^{(\ell)}\tilde{\Sigma}_{\ell}(\bvar{y})\bvar{S}^{(\ell)}\Big]\Delta_S^{-1}.
\end{align*}
\end{Corollary}

\begin{proof}
The proof can be found in Appendix~\ref{Second Moment Appendix}. 
\end{proof}

This result establishes the asymptotic covariances of the rows of $\hat{\bvar{L}}$ corresponding to a common vertex and provides an explicit expression for the leading covariance term.
The term $\Omega_{gk}^{(M)}(\bvar{y})$ describes covariance between $\bvar{M}_{r_g}$ and $\bvar{M}_{r_k}$ and $\Omega_{gk}^{(N,M)}(\bvar{y})$ describes covariance between pairs $(\bvar{N}_{r_g}, \bvar{M}_{r_g})$ and $(\bvar{N}_{r_k}, \bvar{M}_{r_k})$.
Corollary~\ref{Cor:Asym_Joint_Distribution_Different_Vertex} establishes this leading covariance term is block diagonal for rows of $\hat{\bvar{L}}$ corresponding to different vertices.

\begin{Corollary}\label{Cor:Asym_Joint_Distribution_Different_Vertex}
Let $r_i = i + n(g-1)$ and $r_k = j + n(k - 1)$ for $i,j\in[n]$ and $g,k\in[m]$ denote rows corresponding to different nodes $i\neq j\ \in[n]$.
Define the vector
$\bvar{V} = ((\bvar{N} + \bvar{M})_{r_i}^T\hspace{.5em}(\bvar{N} + \bvar{M})_{r_j}^T)^T \in \R^{2d}$ and let $\bvar{v}\in \R^{2d}$.
Then in the context of Theorem \ref{CLT} we have
\begin{align*}
    \lim_{n\to\infty}\prob\left[\sqrt{n}\bvar{V} \leq \bvar{v}\right]  
    = \int_{\text{supp}(F)}\int_{\text{supp}(F)}\Phi(\bvar{v}; \bvar{0}, \Psi_{gk}(\bvar{y}_1, \bvar{y}_2))dF(\bvar{y}_1)dF(\bvar{y}_2)
\end{align*}
where $\Psi_{gk}(\bvar{y}_1, \bvar{y}_2)\in \R^{2d\times 2d}$ can be decomposed as $\Psi_{gk}(\bvar{y}_1, \bvar{y}_2) = \Psi_{gk}^{(N)}(\bvar{y}_1, \bvar{y}_2) + \Psi_{gk}^{(M)}(\bvar{y}_1, \bvar{y}_2) + \Psi_{gk}^{(N, M)}(\bvar{y}_1, \bvar{y}_2)$ where 
\begin{align*}
    \Psi_{gk}^{(N)}(\bvar{y}_1, \bvar{y}_2) = \begin{bmatrix}
    \Sigma_g^{(N)}(\bvar{y}_1) &  \bvar{0}\\
    \bvar{0} & \Sigma_k^{(N)}(\bvar{y}_2)
    \end{bmatrix}
\end{align*}
where $\Sigma_g^{(N)}(\bvar{y})$ is given in Theorem~\ref{CLT}.
\end{Corollary}
\begin{proof}
The proof can be found in Appendix~\ref{Second Moment Appendix}. 
\end{proof}

This result establishes that rows of $\hat{\bvar{L}}$ corresponding to different vertices are approximately independent, asymptotically.
Any covariance that exists is due to covariance between
$\bvar{M}_{r_i}$ and $\bvar{M}_{r_j}$ and the pairs $(\bvar{N}_{r_i}, \bvar{M}_{r_i})$ and $(\bvar{N}_{r_j}, \bvar{M}_{r_j})$ which we anticipate to be small in practice.

\subsection{Theoretical Ramifications}

To support these theoretical findings, we consider a simulation of a two-layer multiplex network and show that (i) the estimates produced by the omnibus embedding are biased, (ii) they concentrate around the unique rows of $\bvar{L}_S$ at the rate given in Theorem~\ref{Main Bias Theorem}, and (iii) each unique latent position has a unique, graph-specific variance.

\begin{ex}\label{ex:two-group-running-ex}
Consider the two-group stochastic block model parameterized by class probabilities $(1/2, 1/2)$ and block probability matrix $\bvar{B}\in\R^{2\times 2}$.  
Let $F$ be a discrete distribution over $\ell_1 = (0.39, 0.32)^T$ and $\ell_2 = (0.39, -0.32)^T$ with mass $(1/2, 1/2)$. 
Consider $(\{\bvar{A}^{(g)}\}_{g=1}^2, \bvar{X})\sim\text{ESRDPG}(F, n, \{\bvar{C}^{(g)}\}_{g=1}^2)$ where $\bvar{B}$ and $\{\bvar{C}^{(g)}\}_{g=1}^2$ are given by 
\begin{align}\label{eq:three-layer}
    \bvar{B} = \begin{bmatrix} 0.25 & 0.05\\ 0.05 & 0.25\end{bmatrix} \hspace{2em}
    \bvar{C}^{(1)} = \begin{bmatrix} 1 & 0\\ 0 & 1\end{bmatrix} \hspace{2em}
    \bvar{C}^{(2)} = \begin{bmatrix} 1 & 0\\ 0 & 0\end{bmatrix} . 
\end{align}
Marginally, $\bvar{A}^{(1)}$ is a two-group SBM and $\bvar{A}^{(2)}$ is an \ErdosRenyi network model with parameter $p = 0.15$.
The latent positions for these networks can be found in Figure~\ref{fig:hierarchy_DSRDPG}.

We sample each $\bvar{A}^{(1)}, \bvar{A}^{(2)}$ from this ESRDPG for $n\in\{250, 500, 1000\}$.
For each sample, we construct the omnibus matrix $\tilde{\bvar{A}}$ and calculate the omnibus embedding $\hat{\bvar{L}}$ in $d=2$ dimensions. 
We look to compare $\hat{\bvar{L}}$ to the weighted latent positions $\bvar{L}$ as well as the scaled latent positions $\bvar{L}_S$.
We compare these three quantities in the left panel of Figure~\ref{fig:bias_plots}. 
The \textsf{x}s represent the points $\bvar{S}^{(g)}\bvar{x}_i$ and the \texttt{+}s represent the points $\sqrt{\bvar{C}^{(g)}}\bvar{x}_i$ for each $i=1,2$ and $g=1,2$.
The colored points are the estimated latent positions. 
The confidence ellipses are calculated \textit{a priori} from known model parameters and the expression of $\Sigma_g^{(N)}(\bvar{y})$ given in Theorem~\ref{CLT}.
While little bias is observed in the first graph the second network incurs non-trivial bias.
Theorem~\ref{Main Bias Theorem} provides a uniform bound on the rows of $\hat{\bvar{L}}\tilde{\bvar{W}}_n - \bvar{L}_S$.
The right panel of Figure~\ref{fig:bias_plots}, compares the $\log(nm)/\sqrt{n}$ rate of this bound to the observed residuals from the simulation study.

\begin{figure}
    \centering
    \includegraphics[width = .48\linewidth]{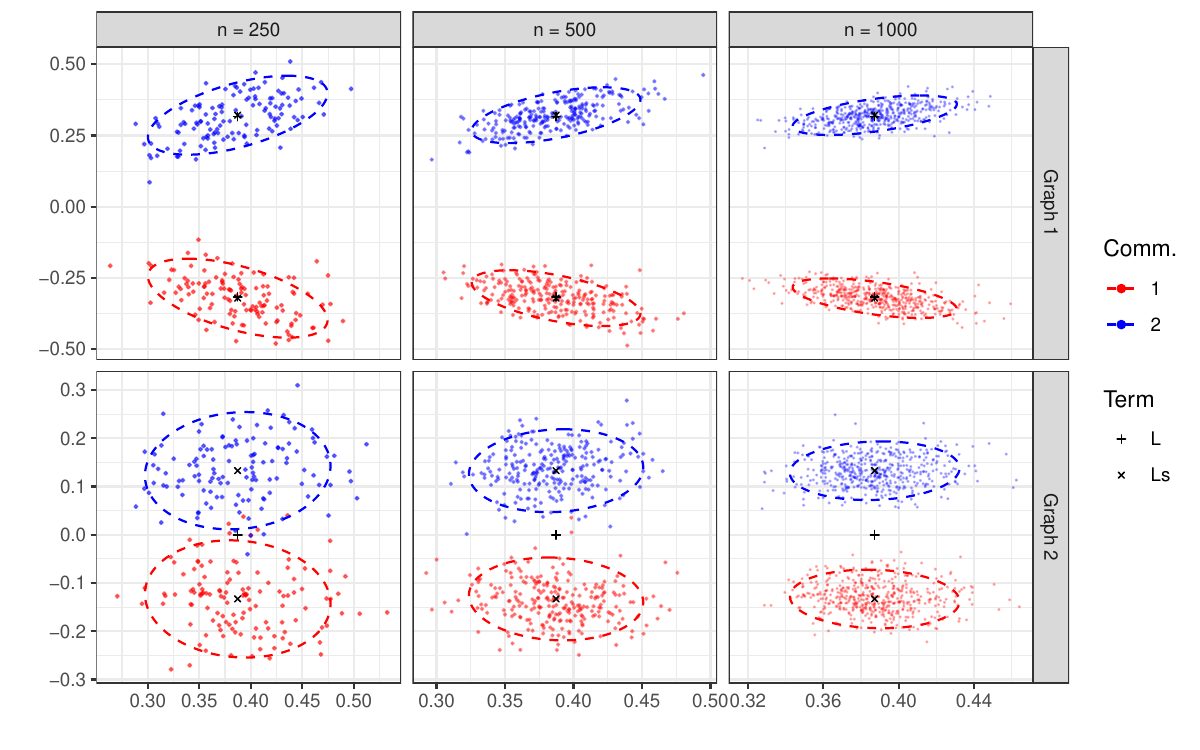}
    \includegraphics[width = .48\linewidth]{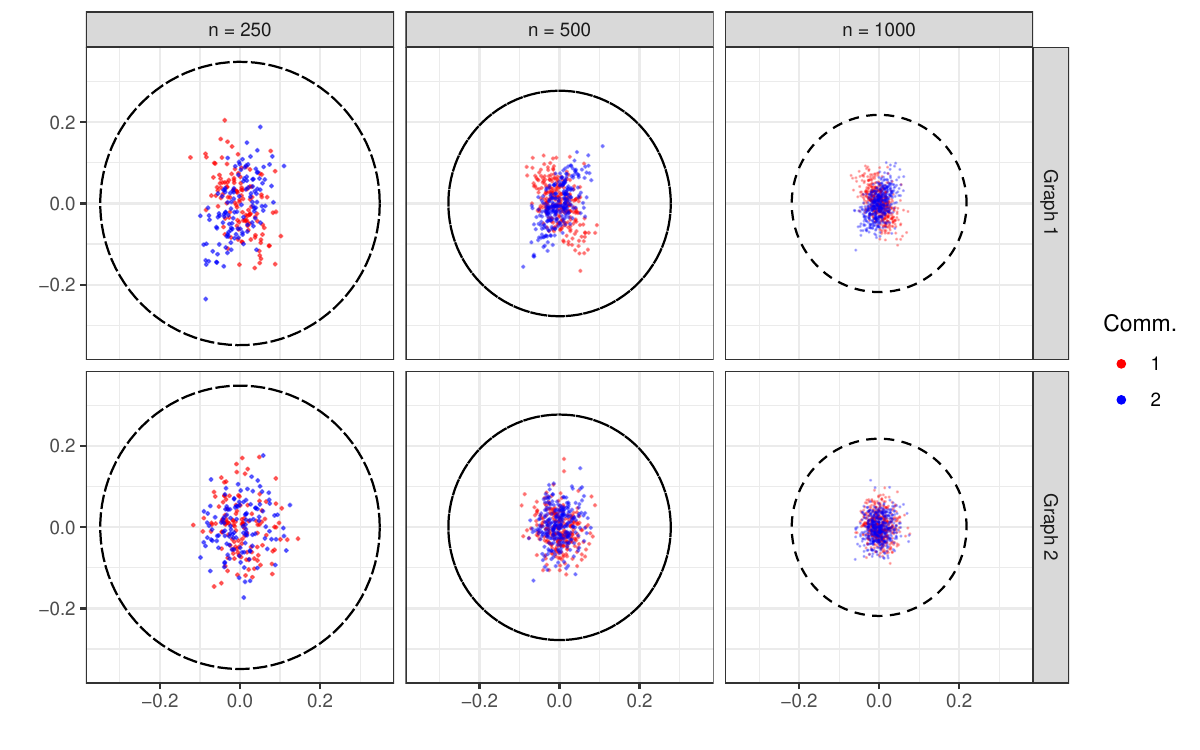}
    \caption{Left panel: Estimated latent positions produced by the omnibus embedding for the ESRDPG in Eq.~\eqref{eq:three-layer}.
    Rows correspond to each network $\bvar{A}^{(1)}$, $\bvar{A}^{(2)}$ while columns represent network size. 
    Points correspond to node embeddings and are colored by community labels. The \textsf{x}s correspond to the estimands $\bvar{L}_S$ and the \texttt{+}s correspond the weighted latent positions $\bvar{L}$.
    Right panel: Centered latent positions $\hat{\bvar{L}}\tilde{\bvar{W}}_n - \bvar{L}_S$.
    The dashed balls are of radius $n^{-1/2}\log 2n$ corresponding to the bound provided in Theorem~\ref{Main Bias Theorem}.}
    \label{fig:bias_plots}
\end{figure}
\qed
\end{ex}

The confidence ellipses in Figure~\ref{fig:bias_plots} are based only on $\Sigma_g^{(N)}(\bvar{y})$ and not the full covariance $\Sigma_g^{(N)}(\bvar{y}) + \Sigma_g^{(M)}(\bvar{y}) + \Sigma_g^{(N,M)}(\bvar{y})$.
This is in primarily due to the fact that the covariance is dominated by $\Sigma_g^{(N)}(\bvar{y})$. 
As we show in Appendix~\ref{Second Moment Appendix},  $\sqrt{n}\bvar{M}_h$ converges to matrix-vector product between a matrix with normally distributed entries with degenerate covariance and the vector $\bvar{X}_i$. 
For finite $n$, we interpret $\bvar{M}_h$ as a random matrix-vector product that translates the the mean $\bvar{L}_S$ in the direction of $\bvar{X}_i$ so the rows of $\hat{\bvar{L}}$ are approximately normally distributed around the rows of $\bvar{L}_S + \bvar{M}$.
As we demonstrate in Figure~\ref{fig:residual_expansion}, this shift is negligible in comparison to the covariance $\Sigma_g^{(N)}(\bvar{y})$ and vanishes in finite sample when $n^{-1}\bvar{X}^T\bvar{X} - \Delta = 0$ which is guaranteed asymptotically as $\max_{h\in[nm]}\|\bvar{M}_h\| = O(m^{3/2}n^{-1/2}\log nm)$. 
For a full discussion, the Appendix~\ref{Second Moment Appendix}.

\begin{excont}
Extending Example~\ref{ex:two-group-running-ex}, we now highlight the role of $\bvar{M}_h$ in finite sample networks. 
Under the two group SBM, $n^{-1}\bvar{X}^T\bvar{X} - \Delta = (p_1 - 1/2)\ell_1\ell_1^T +(1/2 - p_1)\ell_2\ell_2^T$ where $p_1 = n_1/n$ is the proportion of vertices in the first community. 
Therefore, the translation $\bvar{M}_h$ is determined entirely by the deviation $p_1 - 1/2$.
Under the ESRDPG, $np_1\sim\text{Binom}(n, 1/2)$ so we vary $p_1 = 1/2 + c/\sqrt{4n}$ in our simulation and view the effect of the translation term $\bvar{M}_h$ in Figure~\ref{fig:residual_expansion}.

\begin{figure}
    \centering
    \includegraphics[width = 0.75\linewidth]{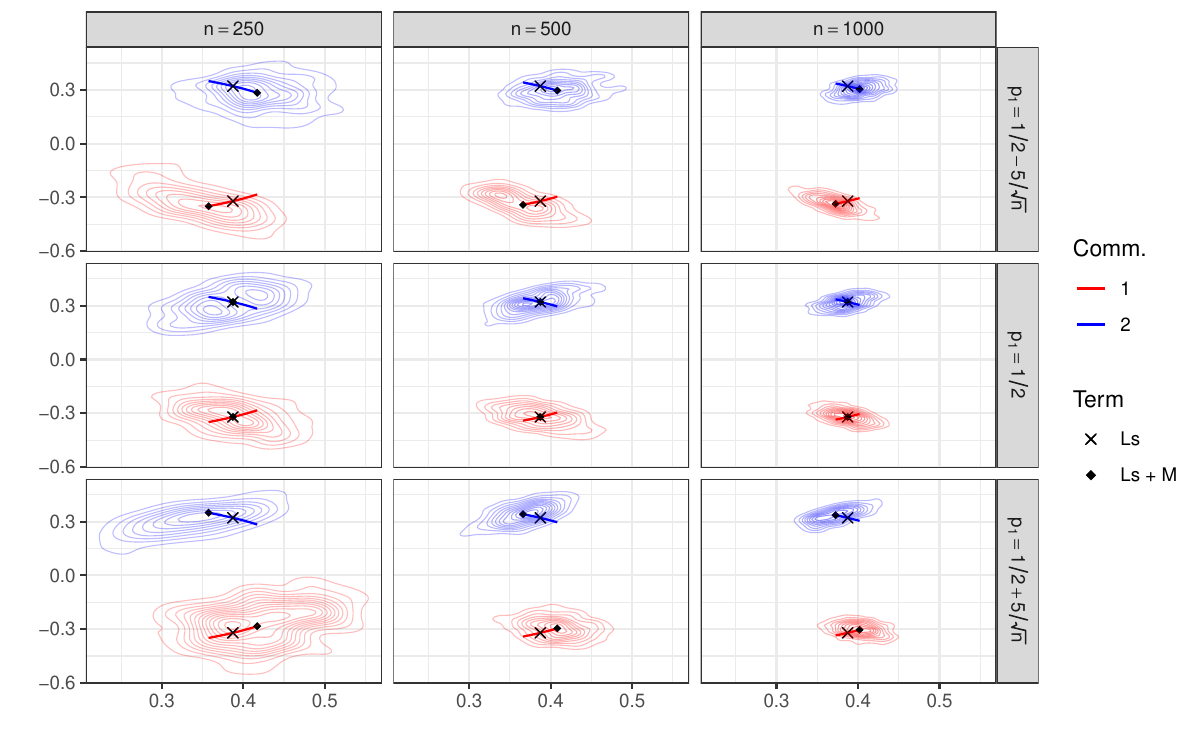}
    \caption{Contour plots of the omnibus estimated latent positions $\hat{\bvar{L}}$ corresponding to $\bvar{A}^{(1)}$, colored by community assignment, for networks of size $n\in\{250, 500, 1000\}$. 
    The \textsf{x}s correspond to the estimands $\bvar{L}_S$ and the $\diamond$s correspond the scaled latent positions $\bvar{L}_S$ translated by $\bvar{M}_h$ corresponding to $p_1 = 1/2 + c/\sqrt{4n}$ for $c = -10, 0, 10$.
    The red/blue lines denote the unique rows of $\bvar{L}_s + \bvar{M}$ for values of $c\in[-10,10]$.
    } 
    \label{fig:residual_expansion}
\end{figure}

In Figure~\ref{fig:residual_expansion}, we see that the contour plots representing $\hat{\bvar{L}}$ do not center around $\bvar{L}_S$ but instead $\bvar{L}_S + \bvar{M}$ in finite samples. 
This translation term itself is normally distributed with mean zero and will vary along the red/blue lines which pass through, and eventually concentrate around, $\bvar{L}_S$.
Note that the values $p_1 = 1/2 \pm 5/\sqrt{n}$ occur with near 0 probability.
We choose to include these extreme values as the differences between $\bvar{L}_S + \bvar{M}$ ($\diamond$s) and $\bvar{L}_S$ (\texttt{x}s) are difficult to decipher for values of $p_1 \in 1/2 \pm 3/\sqrt{4n}$ when $n =250$ and nearly indistinguishable when $n = 1000$.
Therefore, the covariance presented in Theorem~\ref{CLT} is almost entirely attributable to $\Sigma_g^{(N)}(\bvar{y})$.
\qed
\end{excont}

Having demonstrated our theoretical contributions, we now use these results to establish the properties of similar network embedding techniques and carry out a bias-variance analysis. 
As we saw in Example~\ref{Ex:omni_under_heterogenous_comm_det}, community detection can be carried out by clustering the rows of the so called \textit{Omnibar} matrix
\begin{align}\label{Omnibar Def}
    \bar{\bvar{X}} &= \frac{1}{m}\sum_{g=1}^m\hat{\bvar{X}}^{(g)}
\end{align}
where $\hat{\bvar{X}}^{(g)}$ is the $g$-th, $n\times d$ block matrix of $\hat{\bvar{L}}$.
Corollary~\ref{Cor:Asym_Joint_Distribution_Same_Vertex} allows us to directly establish the asymptotic distribution of $\bar{\bvar{X}}$.
This result, formalized in Corollary~\ref{Omnibar Distribution}, will be useful in the analysis of clustering algorithms applied to $\bar{\bvar{X}}$ for multiplex community detection discussed in Section~\ref{Section: Statistical Consequences} as well as facilitate our mean squared error comparisons.

\begin{Corollary}\label{Omnibar Distribution}
Let $\bar{\bvar{X}} = m^{-1}\sum_{g=1}^m\hat{\bvar{X}}^{(g)}$. 
Let $\{\tilde{\bvar{W}}_n\}_{n=1}^{\infty}$ be as in Theorem~\ref{CLT} and let $\bar{\bvar{S}} = m^{-1}\sum_{g=1}^m\bvar{S}^{(g)}$. 
Then the $i$-th row of this matrix satisfies 
\begin{align*}
    \lim_{n\to\infty}\prob\left[\sqrt{n}(\bar{\bvar{X}}\bvar{W}_n -\bvar{X\bar{S}})_i \leq \bvar{x}\right] = \int_{\text{supp}(F)}\Phi(\bvar{x}; \bvar{0}, \Sigma_{\text{OB}}(\bvar{y}))dF(\bvar{y})
\end{align*}
where the variance can be decomposed $\Sigma_{\text{OB}}(\bvar{y}) = \Sigma_{\text{OB}}^{(N)}(\bvar{y}) + \Sigma_{\text{OB}}^{(M)}(\bvar{y}) + \Sigma_{\text{OB}}^{(N,M)}(\bvar{y})$ and $\Sigma_{\text{OB}}^{(N)}(\bvar{y})$ can be written explicitly
\begin{align*}
    \Sigma_{\text{OB}}^{(N)}(\bvar{x}_i) =\frac{1}{4} \Delta_S^{-1}\sum_{g=1}^m\left(\bar{\bvar{S}} + \bvar{S}^{(g)}\right)\tilde{\Sigma}_g(\bvar{x}_i)\left(\bar{\bvar{S}} +\bvar{S}^{(g)}\right)\Delta_S^{-1}
\end{align*}
\end{Corollary}
\begin{proof}
The result follows from an application of Corollary~\ref{Cor:Asym_Joint_Distribution_Same_Vertex}.
\end{proof}

The results in Section~\ref{Subsec:Main_Results} and work from \textcite{EigenCLT} and \textcite{Connectome-Smooth} allow us to analytically compare the MSE of the omnibus embedding to similar network embedding techniques.
We consider the four embedding techniques; the adjacency spectral embedding (ASE) of \textcite{ASE}, the Abar estimator of \textcite{Connectome-Smooth}, and the omnibus and omnibar estimators of \textcite{OmniCLT}.
Equipped with asymptotic distributions for all four estimates, we can now compare their the mean squared error for estimating the scaled latent positions. 
The bias and asymptotic variance of each estimator can be found in Table~\ref{MSE table}. 
\begin{table}[h]
    \centering
     \begin{tabular}{|c|c|c|}
    \hline 
    Method & Bias & Variance\\
    \hline
    ASE     & $\bvar{0}$ & $(\sqrt{\bvar{C}^{(g)}}\Delta)^{-1}\tilde{\Sigma}_g(\bvar{x}_i)(\Delta\sqrt{\bvar{C}^{(g)}})^{-1}/n$ \\
    Abar     & $(\sqrt{\bar{\bvar{C}}} - \sqrt{\bvar{C}^{(g)}})\bvar{x}_i$ & $(\sqrt{\bar{\bvar{C}}}\Delta)^{-1}\sum_{g=1}^m\tilde{\Sigma}_g(\bvar{x}_i)(\Delta\sqrt{\bar{\bvar{C}}})^{-1}/m^2n$ \\
    Omni     & $(\bvar{S}^{(g)} - \sqrt{\bvar{C}^{(g)}})\bvar{x}_i$ &  $\Sigma_g(\bvar{x}_i)/n$ \\
    Omnibar     & $(\bar{\bvar{S}} - \sqrt{\bvar{C}^{(g)}})\bvar{x}_i$  & $\Sigma_{OB}(\bvar{x}_i)/n$ \\
    \hline
    \end{tabular}
    \caption{The asympoptic bias and variance of the ASE, Abar, omnibus, and omnibar embedding techniques under the ESRDPG.}
    \label{MSE table}
\end{table}

The ASE is the only asymptotically unbiased estimator of the scaled latent positions.
This method, however, ignores common structure among the networks and as a consequence suffers higher variance than the other methods.
The remaining methods all incur bias in the direction of the true latent position $\bvar{x}_i$, but are unbiased when $\bvar{C}^{(g)} = \bvar{C}$ for all $g\in[m]$.

The comparison of variances under the ESRDPG is a more complicated undertaking. 
The variance introduced in Theorem~\ref{CLT} and Corollary~\ref{Omnibar Distribution} can be interpreted as linear combinations of the individual network variances $\tilde{\Sigma}_g(\bvar{x}_i)$ with weights given by the scaling matrices $\{\bvar{S}^{(g)}\}_{g=1}^m$.
This interpretation becomes more clear for the omnibar estimator where we see each $\bvar{S}^{(g)}$ is included in the scaling of each $\tilde{\Sigma}_g(\bvar{x}_i)$. 
The variance of ASE and Abar can be seen as a normalization of the graph variance $\tilde{\Sigma}_g(\bvar{x}_i)$ by pre and post multiplying either $(\sqrt{\bvar{C}^{(g)}})^{-1}$ or $(\sqrt{\bar{\bvar{C}}})^{-1}$.
In the i.i.d.\ setting, $\tilde{\Sigma}_g(\bvar{x})_i = \tilde{\Sigma}(\bvar{x}_i)$ for all $g\in [m]$ and the Abar variance reduces to $\Delta^{-1}\tilde{\Sigma}(\bvar{x}_i)\Delta^{-1}/mn$.
This expression was presented in Theorem C.1 of \textcite{Connectome-Smooth} and highlights the variance reduction enjoyed by the Abar embedding.

Due to its asymptotic bias, this bias-variance analysis affirms that the omnibus embedding is not suited for latent position estimation for large graphs.
While the embedding may provide more robust estimation performance due to its variance reduction for small $n$, for moderate and large $n$ the adjacency spectral embedding is likely the preferable estimator.
However, a practitioner's final goal is rarely on latent position estimation and instead on the accuracy of algorithms that use these estimates.
In the following section, we demonstrate algorithms applied to omnibus node embeddings are provably accurate despite the biased latent position estimation and often better perform due to the variance reduction and the effect the bias plays in the regularizing the node embeddings.

\section{Statistical Consequences}\label{Section: Statistical Consequences}
Having characterized the large-graph properties of the rows of the omnibus embedding under the ESRDPG, we turn to analyzing the viability of utilizing these node embeddings for accurate multi-graph infernece.
Facilitated by Theorem~\ref{Main Bias Theorem} and Theorem~\ref{CLT}, we first investigate the ability of various clustering algorithms, applied to omnibus node embeddings, to detect community structure.
We then develop a two graph hypothesis test which exploits the distributional results found in Theorem~\ref{CLT} and Corollary~\ref{Cor:Asym_Joint_Distribution_Same_Vertex}.

\subsection{Community Detection}

Example~\ref{Ex:omni_under_heterogenous_comm_det} demonstrates that accurate community detection can be achieved by applying clustering techniques to the average node embeddings $\bar{\bvar{X}}_i = m^{-1}\sum_{g=1}^m\hat{\bvar{X}}_i^{(g)}$.
Theorem~\ref{Main Bias Theorem} and Corollary~\ref{Omnibar Distribution} establishes $\bar{\bvar{X}}_i$ concentrates around, $\bar{\bvar{S}}\bvar{X}_i$ at a rate of $O(m^{3/2}n^{-1/2}\log nm)$, up to Gaussian error. 
These theoretical results allow us to rigorously analyze clustering algorithms applied to the points $\{\bar{\bvar{X}}_i\}_{i=1}^n\subset\R^d$.
Here, we consider both the $k$-means clustering algorithm and Gaussian Mixture Models (GMM) as clustering algorithms.

Suppose that $F$ is a discrete distribution over $\{\bvar{x}_k\}_{k=1}^K\subset \R^d$ with probabilities $\{\pi_k\}_{k=1}^K\subset (0,1)$ with $\sum_{k=1}^{K}\pi_k = 1$.
Within an ESRDPG framework, sampling $\bvar{X}_1, \bvar{X}_2, \ldots, \bvar{X}_n \overset{i.i.d.}{\sim }F$ and $\bvar{A}^{(g)}\sim\text{Bern}(\bvar{X}\bvar{C}^{(g)}\bvar{X}^T)$ corresponds to sampling $\{\bvar{A}^{(g)}\}_{g=1}^m$ from a Multilayer Stochastic Block Model (MSBM) \parencite{arroyo2019}.
Under the MSBM, a node's community assignment is fixed across networks while the block connection probabilities are allowed to vary. 
The community detection task is then to recover these community assignments upon observing $\{\bvar{A}^{(g)}\}_{g=1}^m$.

\subsubsection{Community Detection with $k$-means}

As the distance between community centroids increases, the clustering task becomes easier. 
Embedding techniques that separate these centroids the furthest should achieve lower error for sufficiently large $n$.
Therefore, we first characterize the distance between centroids for both the Omnibar embedding $\{\bar{\bvar{S}}\bvar{x}_k\}_{k=1}^m$ and the Abar embedding $\{\sqrt{\bar{\bvar{C}}}\bvar{x}_k\}_{k=1}^m$.
We on focus on these embeddings as they achieve the lowest error in Example~\ref{Ex:omni_under_heterogenous_comm_det}.
Employing expressions introduced in Theorem~\ref{Main Bias Theorem}, Corollary~\ref{Cor:Centroid_Distance} reveals that the distance between $\{\bar{\bvar{S}}\bvar{x}_k\}_{k=1}^m$ are always at least as far as the distance between $\{\sqrt{\bar{\bvar{C}}}\bvar{x}_k\}_{k=1}^m$.

\begin{Corollary}\label{Cor:Centroid_Distance}
Suppose that $F$ is a discrete distribution over $\{\bvar{x}_k\}_{k=1}^K$.
Recall $\|\sqrt{\bar{\bvar{C}}}(\bvar{x}_k - \bvar{x}_{\ell})\|$ and $\|\bar{\bvar{S}}(\bvar{x}_k - \bvar{x}_{\ell})\|$ give the distance between the $k$ and $\ell$ centroid for the Abar and Omnibar embeddings, respectively.
Then for any $k, \ell\in[K]$  with $k\neq \ell$
\begin{align*}
1 \leq   \frac{\|\bar{\bvar{S}}(\bvar{x}_k - \bvar{x}_{\ell})\|}{\|\sqrt{\bar{\bvar{C}}}(\bvar{x}_k - \bvar{x}_{\ell})\|} \leq \frac{m^{-1/4} + m^{1/4}}{2} 
\end{align*}
\end{Corollary}
\begin{proof}
The proof can be found in Appendix~\ref{Corollaries and Statistical Consequences Appendix}.
\end{proof}

The lower bound is achieved in the i.i.d.\ setting where $\bar{\bvar{S}} = \sqrt{\bar{\bvar{C}}} = \sqrt{\bvar{C}}$ for some fixed matrix $\bvar{C}$.
The upper bound is achieved, for example, when $\bvar{X}\in\R^{n\times (m+1)}$ and $\bvar{C}^{(g)} = \bvar{e}_1\bvar{e}_1^T + \bvar{e}_{g+1}\bvar{e}_{g+1}^T$. 
In this setting, only graph $g$ is active in the $g+1$ dimension of the latent space and the relative differences in centroid centroids satisfies 
\begin{align*}
    \frac{\|\bar{\bvar{S}}(\bvar{x}_k - \bvar{x}_{\ell})\|}{\|\sqrt{\bar{\bvar{C}}}(\bvar{x}_k - \bvar{x}_{\ell})\|} &= \sqrt{\frac{\bar{\bvar{C}}_{11}(\bvar{x}_k - \bvar{x}_{\ell})_{11}^2}{\sum_{i=1}^m\bar{\bvar{C}}_{ii}(\bvar{x}_k - \bvar{x}_{\ell})_{ii}^2} +\left(\frac{m^{-1/4} + m^{1/4}}{2}\right)\frac{\sum_{i=2}^m\bar{\bvar{C}}_{ii}(\bvar{x}_k - \bvar{x}_{\ell})_{ii}^2}{\sum_{i=1}^m\bar{\bvar{C}}_{ii}(\bvar{x}_k - \bvar{x}_{\ell})_{ii}^2}}\\ &\geq O(m^{-1/4} + m^{1/4}).
\end{align*}
In this example the distances between $\{\bar{\bvar{S}}\bvar{x}_k\}_{k=1}^m$ are $O(m^{1/4})$ further apart than the $\{\sqrt{\bar{\bvar{C}}}\bvar{x}_k\}_{k=1}^m$.
Therefore, for large $m$, the omnibus embedding separates the community centroids considerably further than the Abar embedding.
While Corollary~\ref{Cor:Centroid_Distance} does not address the rate of convergence to these centroids, the result highlights an example where the biased estimates provided by the omnibus embedding make the downstream inference task easier to address.

We now consider the effect of the rate of convergence to these centroids by stating sufficient conditions for the exact recovery of the community labels.
Denote $\{\bar{\bvar{X}}_i\}_{i=1}^n$ and $\{\bar{\bvar{Y}}_i\}_{i=1}^n$ as the Omnibar and Abar node embeddings, respectively. 
The node embeddings, $\{\bar{\bvar{X}}_i\}_{i=1}^n$ and $\{\bar{\bvar{Y}}_i\}_{i=1}^n$, concentrate around the centroids $\{\bar{\bvar{S}}\bvar{x}_k\}_{k=1}^K$ and $\{\sqrt{\bar{\bvar{C}}}\bvar{x}_k\}_{k=1}^K$, respectively.
Suppose that vertex $i\in[n]$ is assigned to community $k\in[K]$.
Then Theorem~\ref{Main Bias Theorem} establishes $\bar{\bvar{X}}_i = \bar{\bvar{S}}\bvar{x}_k  + O(m^{3/2}n^{-1/2}\log nm)$ and extending Theorem 8 of~\textcite{RDPGSurvey} gives $\bar{\bvar{Y}}_i = \sqrt{\bar{\bvar{C}}}\bvar{x}_i + O(n^{-1/2}\log nm)$.
Therefore, to achieve exact recovery of the community labels in finite sample networks, these centroids must be sufficiently separated.
We formalize these conditions in Corollary~\ref{Cor:exact_recovery_conditions}.

\begin{Corollary}\label{Cor:exact_recovery_conditions}
Suppose that $(\{\bvar{A}^{(g)}\}_{g=1}^m, \bvar{X})\sim \text{ESRPDG}(F, n, \{\bvar{C}^{(g)}\}_{g=1}^m)$ where $F$ is a discrete distribution over $\{\bvar{x}_k\}_{k=1}^K$.
Let $\{\bar{\bvar{X}}_i\}_{i=1}^n$ and $\{\bar{\bvar{Y}}_i\}_{i=1}^n$ denote the Omnibar and Abar embeddigns, respectively. 
Then the community labels are recovered exactly with high probability provided
\begin{align*}
    \min_{i,j\in[K]}\|\sqrt{\bar{\bvar{C}}}(\bvar{x}_i - \bvar{x}_j)\| > \alpha(n,m) \quad \quad \min_{i,j\in[K]}\|\bar{\bvar{S}}(\bvar{x}_i - \bvar{x}_j)\| > \beta(n,m).
\end{align*}
Here, $\alpha(n,m), \beta(n,m)$ are constants depending on model parameters $(n, m)$ such that $\alpha(n,m) = \Theta(n^{-1/2}\log(nm))$ and $\beta(n,m) = \Theta(m^{3/2}n^{-1/2}\log(nm))$.
\end{Corollary}

An analogous analysis of k-means applied to the rows of adjacency spectral embedding for single network data was completed in \textcite{PerfectClustering}.
As these arguments follows \textit{mutatis mutandis} from Theorem 2.6 of \textcite{PerfectClustering} with the application of Theorem~\ref{Main Bias Theorem} and the extension of Theorem 8 of \textcite{RDPGSurvey} to the Abar embedding, we refer the reader to this analysis and state the sufficient conditions without proof.

While Corollary~\ref{Cor:Centroid_Distance} establishes that the Omnibar embedding separates centroid centers more than the Abar embedding, Corollary~\ref{Cor:exact_recovery_conditions} establishes that less separation is needed for the Abar embedding to achieve exact recovery. 
This tradeoff appears to be driven by the number of networks, $m$.
For small $m$, the concentration rates of the two methods are comparable, but the upper bound in Corollary~\ref{Cor:Centroid_Distance} ensure that the distances between the Omnibar centroids $\{\bar{\bvar{S}}\bvar{x}_k\}_{k=1}^K$ are comparable to the distances between the Abar centroids $\{\sqrt{\bar{\bvar{C}}}\bvar{x}_k\}_{k=1}^K$.
For large $m$, the Omnibar concentration could potentially be slower than the Abar concentration, yet the method could separate the centroids $\{\bar{\bvar{S}}\bvar{x}_k\}_{k=1}^K$ significantly further than the Abar centroids $\{\sqrt{\bar{\bvar{C}}}\bvar{x}_k\}_{k=1}^K$.
This insight could explain why the Omnibar embedding the Abar embedding achieve near identical clustering performance in Example~\ref{Ex:omni_under_heterogenous_comm_det}.

\subsubsection{Community Detection with Gaussian Mixture Models}
Under the RDPG parameterization of the MSBM, conditioning on a vertex's community assignment is equivalent to conditioning on its latent position $\{\bvar{X}_i = \bvar{x}_i\}$.
This observation gives rise to Corollary~\ref{SBM Conditional}.
\begin{Corollary}\label{SBM Conditional}
In the context of Theorem~\ref{CLT}, suppose for $\bvar{x}_k\in\text{supp}(F)$ that $\prob(\bvar{X}_i = \bvar{x}_k)> 0$.
Then conditional on $\{\bvar{X}_i = \bvar{x}_k\}$, the Omnibar embedding satisfies
\begin{equation}
    \lim_{n\to\infty}\prob\left[\sqrt{n}(\bar{\bvar{X}}\tilde{\bvar{W}}_n - \bvar{X}\bar{\bvar{S}})_i\leq \bvar{z} | \bvar{X}_i = \bvar{x}_k\right] = \Phi(\bvar{z}; \bvar{0}, \Sigma_{OB}(\bvar{x}_k))
\end{equation}
\end{Corollary}
\begin{proof}
The proof can be found in Appendix~\ref{Corollaries and Statistical Consequences Appendix}.
\end{proof}

The community specific covariance given in Corollary~\ref{SBM Conditional} suggests performing clustering with GMM fit using the EM algorithm as this approach can flexibly incorporate differing variance structures between communities.
While Corollary~\ref{SBM Conditional} establishes an asymptotic distribution for $\bar{\bvar{X}}_i|\{\bvar{X}_i = \bvar{x}_k\}$ and Corollary~\ref{Cor:Asym_Joint_Distribution_Different_Vertex} establishes approximate asymptotic independence between $\bar{\bvar{X}}_i$ and $\bar{\bvar{X}}_j$ for $i\neq j$, these results are only approximate for finite $n$. 
Nonetheless, these results motivate the use of a pseudo-likelihood method that assume $\bar{\bvar{X}}_i$ are independent, normally distributed vectors. 
Let $\bvar{Z}\in[K]^n$ be a vector in which $\bvar{Z}_i$ is the community assignment of vertex $i$.
Then we utilize the pseudo-likelihood of $(\bar{\bvar{X}}, \bvar{Z})$ which takes the form of a GMM
\begin{align*}
    \mathcal{PL}(\bar{\bvar{X}}, \bvar{Z}) =  \prod_{i=1}^n\prod_{k=1}^K\pi_k^{I(\bvar{Z}_i = k)}\phi(\bar{\bvar{S}}\bvar{x}_k, n^{-1}\Sigma_{OB}(\bvar{x}_k))^{I(\bvar{Z}_i = k)}
\end{align*}
where $\phi(\mu, \Sigma)$ is the density function of a normal random variable with mean $\mu$ and covariance $\Sigma$.
The maximization step is slightly more complex than that of a standard GMM as the variance is mean dependent and hence represent a curved sub-family of the multivariate normal model.
To avoid this difficultly, we implement the traditional GMM algorithm to derive estimates for the model parameters $\left(\pi, \bar{\bvar{S}}\bvar{x}_k, n^{-1}\Sigma_{OB}(\bvar{x}_k)\right)$ and community assignment vector $\bvar{Z}$.

With this choice of algorithm, we now look to study the Mahalanobis distance between centroids for each embedding method considered.
As the GMM estimates both the mean and variance of each community on each iteration, optimal embedding methods will both separate the centroids sufficiently and reduce dispersion around these centroids. 
By analyzing the Mahalanobis distance between centroids, we hope to quantify the general difficulty of the clustering task induced by each embedding technique. 
In Example~\ref{Ex:mahalanobis_comm_det_example}, we compare an assortment of joint embedding techniques in a community detection task where we record both the classification accuracy as well as the estimated Mahalanobis distances between community centroids.

\begin{ex}\label{Ex:mahalanobis_comm_det_example}
Suppose that $(\{\bvar{A}^{(g)}\}_{g=1}^4, \bvar{X})\sim\mathrm{ESRDPG}(F, n = 100, \{\bvar{C}^{(g)}\}_{g=1}^4)$ where $F$ corresponds to a SBM with $K = 3$ groups. 
Let $\bvar{B}$ be the block probability matrix given by 
\begin{align*}
    \bvar{B} = \begin{bmatrix} 0.3 & 0.1 & 0.1\\
    0.1 & 0.25 & 0.15\\
    0.1 & 0.15 & 0.25
    \end{bmatrix}.
\end{align*}
corresponding to latent positions $\ell_1 = (0.41,-0.37,0)^T$, $\ell_2 = (0.41, 0.18, -0.23)^T$, and $\ell_3 = (0.41,0.18, 0.23)^T$. 
Moreover, suppose the four weighting matrices are given by 
\begin{align*}
    \bvar{C}^{(1)} = \bvar{I}\hspace{1em}\bvar{C}^{(2)}(t) = \begin{bmatrix}1 & 0 & 0\\ 0 & 1 & 0 \\ 0 & 0 & 1-t\end{bmatrix}
    \hspace{1em}\bvar{C}^{(3)}(t) = \begin{bmatrix}1 & 0 & 0\\ 0 & 1-t & 0 \\ 0 & 0 & 1\end{bmatrix}
    \hspace{1em}\bvar{C}^{(4)}(t) = \begin{bmatrix}1 & 0 & 0\\ 0 & 1-t & 0 \\ 0 & 0 & 1-t\end{bmatrix}
\end{align*}
for $t\in [0,1]$. 
$t$ parameterizes the distance from homogeneity.
At $t = 1$, $\bvar{A}^{(2)}|\bvar{X}$ corresponds to a $K = 2$ group SBM, $\bvar{A}^{(3)}|\bvar{X}$ corresponds a $K = 3$ group SBM with different connectivity structure, and $\bvar{A}^{(4)}|\bvar{X}$ corresponds to an \ErdosRenyi graph. 

We sample networks of size $n = 100$ and attain estimated community labels from the GMM applied to node embedding produced by embedding techniques Abar, ASE1, JE, MASE, MRDPG, Omnibar.
ASE1, Abar, MASE, and Omnibar were introduce in Example~\ref{Ex:omni_under_heterogenous_comm_det} and JE and MRDPG correspond to the Joint Embedding of Graphs of \textcite{wang2017} and the Multiple Random Dot Product Graph (MRDPG) of \textcite{nielsen2018}, respectively.
In addition, using the estimated community labels, we estimated the Mahalanobis distance between community centroids by $(\bar{\bvar{x}}_k - \bar{\bvar{x}}_{\ell})^T\hat{\Sigma}_{k,\ell}^{-1}(\bar{\bvar{x}}_k - \bar{\bvar{x}}_{\ell})$ where $\bar{\bvar{x}}_k$ is the average node embedding for community $k\in\{1, 2, 3\}$ and $\hat{\Sigma}_{k,\ell}$ is the standard pooled variance estimate.
We replicate this process 500 times.
The results of the simulation can be found in Figure~\ref{fig:multiple network clustering comp 2}. 

\begin{figure}
    \centering
    \includegraphics[width=.7\linewidth]{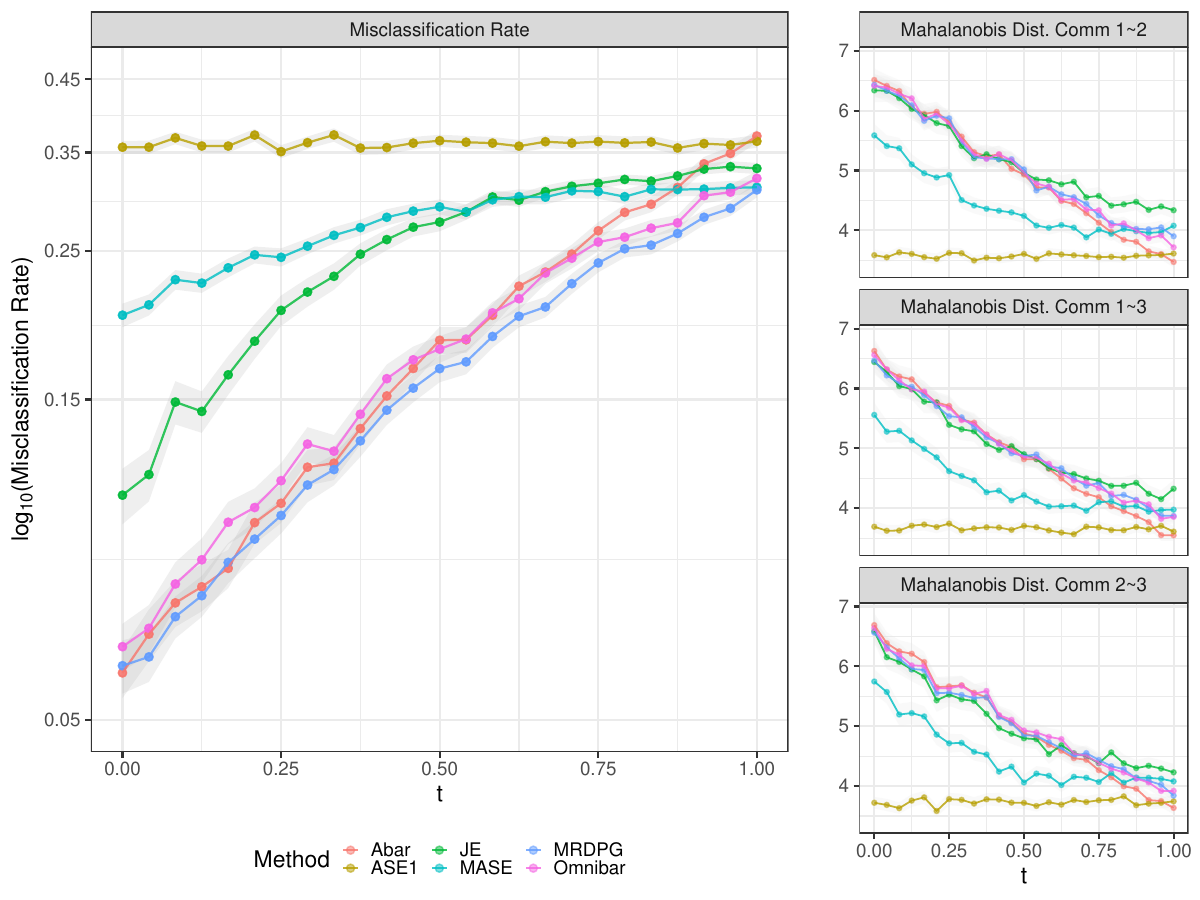}
    \caption{
    \textit{Left panel}: The misclassification rate for GMM clustering applied to node emebddings provided by several joint embedding techniques. 
    The vertical axis is the log-transformed missclassification rate, the horizontal axis corresponds to $t\in[0,1]$, and each technique is shaded a different color.
    \textit{Right panel}: The estimated Mahalanobis distance between community centroids for each pair $(k, \ell)$ for $k,\ell\in\{1,2,3\}$.}
    \label{fig:multiple network clustering comp 2}
\end{figure}

From Figure~\ref{fig:multiple network clustering comp 2} it is clear as the networks become more heterogeneous the performance of each joint embedding method declines.
As the first network is independent of $t$, the performance of ASE1 is constant with respect to $t$ and is included to compare joint embedding techniques to individual network embeddings.
Each joint embedding method outperforms the individual embedding approach, ASE1, for every value of $t\in[0,1]$.
The MRDPG, Omnibar, and Abar methods offer the best performance and are comparable for all values of $t\in[0,1]$. 
Of the joint embedding techniques, the JE and MASE techniques suffer the worst misclassification rate for homogeneous networks (i.e. $t = 0$) but appear to be relatively stable with respect to $t$. 

The performance of these methods is reflected in the Mahalanobis distance between centroids. 
ASE1 appears to separate the centroids the least for all $t\in[0,1]$ and hence suffers the worst classification accuracy.
The MRDPG, Abar, JE, and Omnibar embeddings all separate the centroids similarly and achieve comparable classification accuracy. 
Finally, MASE separates the centroids the least among joint embedding techniques and hence has the worse classification performance.
\qed
\end{ex}
Corollary~\ref{Cor:Centroid_Distance}, Corollary~\ref{Cor:exact_recovery_conditions}, and Example~\ref{Ex:mahalanobis_comm_det_example} establish that accurate community detection can be achieved by applying clustering algorithms to the Omnibar node embeddings in the heterogeneous network setting. 
These insights rely chiefly on the contents of Theorem~\ref{Main Bias Theorem} and Theorem~\ref{CLT} and suggest clustering algorithms that utilize the Omnibar node embeddings are competitive, if not preferable, when compared to other embedding techniques.

\subsection{Hypothesis Testing}
Under the ESRDPG, we parameterize network differences through the graph specific weighting matrices $\{\bvar{C}^{(g)}\}_{g=1}^m$. 
In this section, we consider the task of testing the hypothesis that two networks drawn from the ESRDPG share the same weighting matrix:
\begin{align*}
    H_0: \bvar{C}^{(g)} = \bvar{C}^{(k)} \quad \quad H_A: \bvar{C}^{(g)} \neq \bvar{C}^{(k)}. 
\end{align*}
For ease of notation, we let $g = 1$ and $k = 2$. 
Recall, as $\bvar{X}$ is full rank, $\bvar{C}^{(1)} = \bvar{C}^{(2)}$ if and only if $\bvar{X}\bvar{S}^{(1)} = \bvar{X}\bvar{S}^{(2)}$.
Therefore, an equivalent hypothesis test is given by 
\begin{align*}
    H_0: \bvar{XS}^{(1)} = \bvar{XS}^{(2)} \quad \quad H_A: \bvar{XS}^{(1)} \neq \bvar{XS}^{(2)}. 
\end{align*}
Theorem \ref{Main Bias Theorem} ensures that the matrix $\hat{\bvar{D}} = \hat{\bvar{X}}^{(1)} - \hat{\bvar{X}}^{(2)}$ will reflect differences in $\bvar{C}^{(1)}$ and $\bvar{C}^{(2)}$.
Thus, to asses $H_0$ a natural first step is to construct test statistics from the rows of $\hat{\bvar{D}}$.
We derive the asymptotic distribution of the rows of $\hat{\bvar{D}}$ in Corollary~\ref{Cor:Vertex_Differences_Test_Stat}.

\begin{Corollary}\label{Cor:Vertex_Differences_Test_Stat}
Let $i\in[n]$ be some row of $\hat{\bvar{D}}$ and let $\{\tilde{\bvar{W}}_n\}_{n=1}^{\infty}$ be as in Theorem \ref{CLT}.
Then, under $H_0$, we have the convergence
\begin{align*}
    \lim_{n\to\infty}\prob\left[\sqrt{n}(\hat{\bvar{D}}\tilde{\bvar{W}}_n - \bvar{X}(\bvar{S}^{(1)} - \bvar{S}^{(2)}))_i \leq \bvar{x}\right] = \int_{\text{supp}(F)}\Phi(\bvar{x}; \bvar{0}, \Sigma_{\text{D}}(\bvar{y}))dF(\bvar{y})
\end{align*}
The variance can be decomposed as $ \Sigma_{\text{D}}(\bvar{y}) =  \Sigma_{\text{D}}^{(N)}(\bvar{y}) +  \Sigma_{\text{D}}^{(M)}(\bvar{y})+  \Sigma_{\text{D}}^{(N,M)}(\bvar{y})$ and $\Sigma_{\text{D}}^{(N)}(\bvar{y})$ can be written explicitly as
\begin{align*}
    \Sigma_D^{(N)}(\bvar{y}) = \frac{m^2}{4}\Delta_S^{-1}\bar{\bvar{S}}(\tilde{\Sigma}_1(\bvar{y}) + \tilde{\Sigma}_2(\bvar{y}))\bar{\bvar{S}}\Delta_S^{-1}. 
\end{align*}
\end{Corollary}
\begin{proof}
The result follows from an application of Corollary~\ref{Cor:Asym_Joint_Distribution_Same_Vertex}.
\end{proof}

\textcite{OmniCLT} compare the test statistic $T = \|\hat{\bvar{D}}\|_F^2$ to a reference distribution constructed through Monte Carlo iterations under the null hypothesis.
In simulation settings, this test statistic demonstrates higher empirical power than a Procrustes based test introduced by \textcite{Semi-Par} that utilizes individual network embeddings. 
We stress that this statistic does not correct for row-wise correlation in $\hat{\bvar{D}}$ and relies on a reference distribution that is constructed with prior knowledge of the latent positions $\bvar{X}$.

In an attempt to remedy these issues, we propose a test statistic constructed from Wald statistics for each row of $\hat{\bvar{D}}$.
These statistics are estimated directly for the data and utilize the covariance expression presented in Corollary~\ref{Cor:Vertex_Differences_Test_Stat} to correct for row-wise variability.
We derive the asymptotic distribution of these Wald Statistics in Theorem \ref{Test Statistic W}.

\begin{theorem}\label{Test Statistic W}
Let $\bvar{D}(\bvar{x}) = (\bvar{S}^{(1)} - \bvar{S}^{(2)})\bvar{x}$ and let $F_{\chi^2_d}(x)$ be the cumulative distribution function of a $\chi^2$ random variable with $d$ degrees of freedom.
The asymptotic distribution of the statistic $W_i = \hat{\bvar{D}}_i^T\Sigma_D^{-1}(\bvar{X}_i)\hat{\bvar{D}}_i$ 
under both hypotheses is given by
\begin{align*}
    H_0:& \lim_{n\to\infty}\prob\left[nW_i \leq x\right] =  F_{\chi^2_d}(x)\\
    H_A:& \lim_{n\to\infty}\prob\left[\sqrt{n}[W_i  - \bvar{D}(\bvar{X}_i)^T\Sigma_D^{-1}(\bvar{X}_i)\bvar{D}(\bvar{X}_i)]\leq x\right]\\  
    &\hspace{10em}= \int_{\text{supp} F}\Phi(\bvar{x}; \bvar{0},4\bvar{D}(\bvar{y})^T\Sigma_D^{-1}(\bvar{y})\bvar{D}(\bvar{y}))dF(\bvar{y})
\end{align*}
\end{theorem}
\begin{proof}
The proof can be found Appendix~\ref{Corollaries and Statistical Consequences Appendix}. 
\end{proof}

These statistics are constructed for each vertex $i\in[n]$ but suggest a test statistic for the full network hypothesis $H_0:\bvar{XS}^{(1)} = \bvar{XS}^{(2)}$.
We propose using the test statistic $W = \sum_{i=1}^n W_i$ for evaluating $H_0$.
If each $W_i$ were independent for finite $n$, under the null $nW$ would follow a $\chi^2_{nd}$ distribution. 
However, results presented in Section~\ref{Section: Main Results} establish that $W_i$ and $W_j$ have a small covariance structure.
Nevertheless, we treat $nW$ as approximately distributed as $\chi^2_{nd}$ as a principled approach to testing $H_0: \bvar{C}^{(1)} = \bvar{C}^{(2)}$. 
This assumption effects the power of our test for small networks but this effect diminishes for moderate network sizes as demonstrated in Example~\ref{Ex:Hypothesis_Testing}.

To this point, the test statistic $W$ still relies on unknown model parameters, $\{\Sigma_D(\bvar{X}_i)\}_{i=1}^n$, which will need to be estimated in practice.
Following the argument presented in Section~\ref{Section: Main Results}, $\Sigma_D(\bvar{y})$ is dominated by $\Sigma_{D}^{(N)}(\bvar{y})$ presented in Corollary~\ref{Cor:Vertex_Differences_Test_Stat}.
Therefore, we propose a combination of method of moments estimators to estimate $\Sigma_D^{(N)}(\bvar{X}_i)$ under $H_0$ and use this as a plugin estimator for $\Sigma_D(\bvar{X}_i)$. 
Under the null hypothesis, $\Sigma_D^{(N)}(\bvar{X}_i)$ takes the form
\begin{align*}
    \Sigma_D^{(N)}(\bvar{X}_i) = \frac{\Delta^{-1}\tilde{\Sigma}(\bvar{X}_i)\Delta^{-1}}{2}
\end{align*}
where $\tilde{\Sigma}(\bvar{y}) = \E[(\bvar{y}^T\bvar{X}_j - (\bvar{y}^T\bvar{X}_j)^2)\bvar{X}_j\bvar{X}_j^T]$.
Due to Theorem 1, under the null hypothesis, $n^{-1}\tilde{\bvar{W}}_n\hat{\bvar{X}}^{(g)^T}\hat{\bvar{X}}^{(g)}\tilde{\bvar{W}}_n \overset{a.s.}{\longrightarrow} \Delta$ for $g = 1, 2$.
Therefore our estimator for $\Delta$ can be written as 
\begin{align*}
    \hat{\Delta} = \frac{1}{2n}\sum_{g=1}^2\hat{\bvar{X}}^{(g)T}\hat{\bvar{X}}^{(g)}.
\end{align*}
We estimate $\tilde{\Sigma}(\bvar{X}_i)$ using
\begin{align*}
    \hat{\tilde{\Sigma}}(\bvar{X}_i) = \frac{1}{2n}\sum_{g=1}^2\sum_{j=1}^n(\bar{\bvar{X}}_i^T\hat{\bvar{X}}_j^{(g)} - (\bar{\bvar{X}}_i^T\hat{\bvar{X}}_j^{(g)})^2) \hat{\bvar{X}}_j^{(g)}\hat{\bvar{X}}_j^{(g)T}
\end{align*}
Under a MANOVA null hypothesis, $H_0:\bvar{C}^{(1)} = \bvar{C}^{(2)} = \dots = \bvar{C}^{(m)}$, these estimators should include latent position from all $m$ networks, not just the two networks being compared.  
Combining these estimates, estimates for the precision $\hat{\Sigma}_D^{-1}(\bvar{X}_i)$ is written as
\begin{align*}
    \hat{\Sigma}_D(\bvar{X}_i)^{-1} = 2\hat{\Delta}\left(\hat{\tilde{\Sigma}}(\bvar{X}_i)\right)^{-1}\hat{\Delta}.
\end{align*}
Given our estimates for $\Sigma_D^{-1}(\bvar{X}_i)$, we define our estimates of $W_i$ and $W$ as 
\begin{align*}
\hat{W}_i= \hat{\bvar{D}}_i^T\hat{\Sigma}_D^{-1}(\bvar{X}_i)\hat{\bvar{D}}_i\hspace{4em} \hat{W} = \sum_{i=1}^n\hat{W}_i    
\end{align*}
Using this test statistic, we reject $H_0$ when $n\hat{W} > F^{-1}_{\chi^2_{nd}}(1-\alpha)$ where $F^{-1}_{\chi^2_{nd}}$ is the quantile function of a $\chi^2_{nd}$ random variable and $\alpha$ is the predetermined significance level. 
We note that $\hat{W}$ is purely a function of the data and can be estimated after having computed the omnibus embedding of $\{\bvar{A}^{(g)}\}_{g=1}^m$.
As the test statistic $T$ is computed based on Euclidean rather than Mahalanobis distances, we expect our test statistic will offer improvements in practice over that of $T$.
We also consider a level-corrected version of the $\hat{W}$ statistic, $\tilde{W}$.
These corrections are completed by choosing $c_n\in\N_0$ such that the critical value $\chi^2_{nd + c_n}$ achieves an $\alpha$-level rejection under $H_0$ for each value of $n$.
We compare the empirical power of $T$, $W$, $\hat{W}$, and $\tilde{W}$ in the following simulation setting. 

\begin{ex}\label{Ex:Hypothesis_Testing}

Suppose that $(\{\bvar{A}^{(g)}\}_{g=1}^2, \bvar{X})\sim\mathrm{ESRDPG}(F, n,\{\bvar{C}^{(g)}\}_{g=1}^2)$ where $F$ corresponds to a SBM with $K = 2$ groups. 
Let $\bvar{B}$ be the block probability matrix corresponding to latent position $\ell_1 = (0.39,-0.32)^T$ and $\ell_2 = (0.39,0.32)^T$ and suppose $\bvar{B}$ and the weighting matrices $\{\bvar{C}^{(g)}\}_{g=1}^m$ are given by 
\begin{align*}
    \bvar{B} = \begin{bmatrix} 0.25 & 0.05\\
    0.05 & 0.25
    \end{bmatrix}
    \quad 
    \bvar{C}^{(1)} = \bvar{I}\hspace{2em}\bvar{C}^{(2)}(t) = \begin{bmatrix}1+t & 0\\ 0 & 1-t\end{bmatrix}
\end{align*}
for $t\in [0,1]$. 
$t$ parameterizes the distance from homogeneity with where $t = 0$ corresponds to $H_0$ and $t\in (0,1]$ corresponds to $H_A$.
We sample networks of size $n \in \{50, 100, 200\}$ and test $H_0$ using $T$, $W$, $\hat{W}$, and $\tilde{W}$. 
We complete $1000$ Monte Carlo replicates and calculate the empirical power of each testing procedure.
The results of this simulations study can be found in Figure \ref{fig:Empirical Power}. 
\begin{figure}
    \centering
    \includegraphics[width = .9\linewidth, height = .45\linewidth]{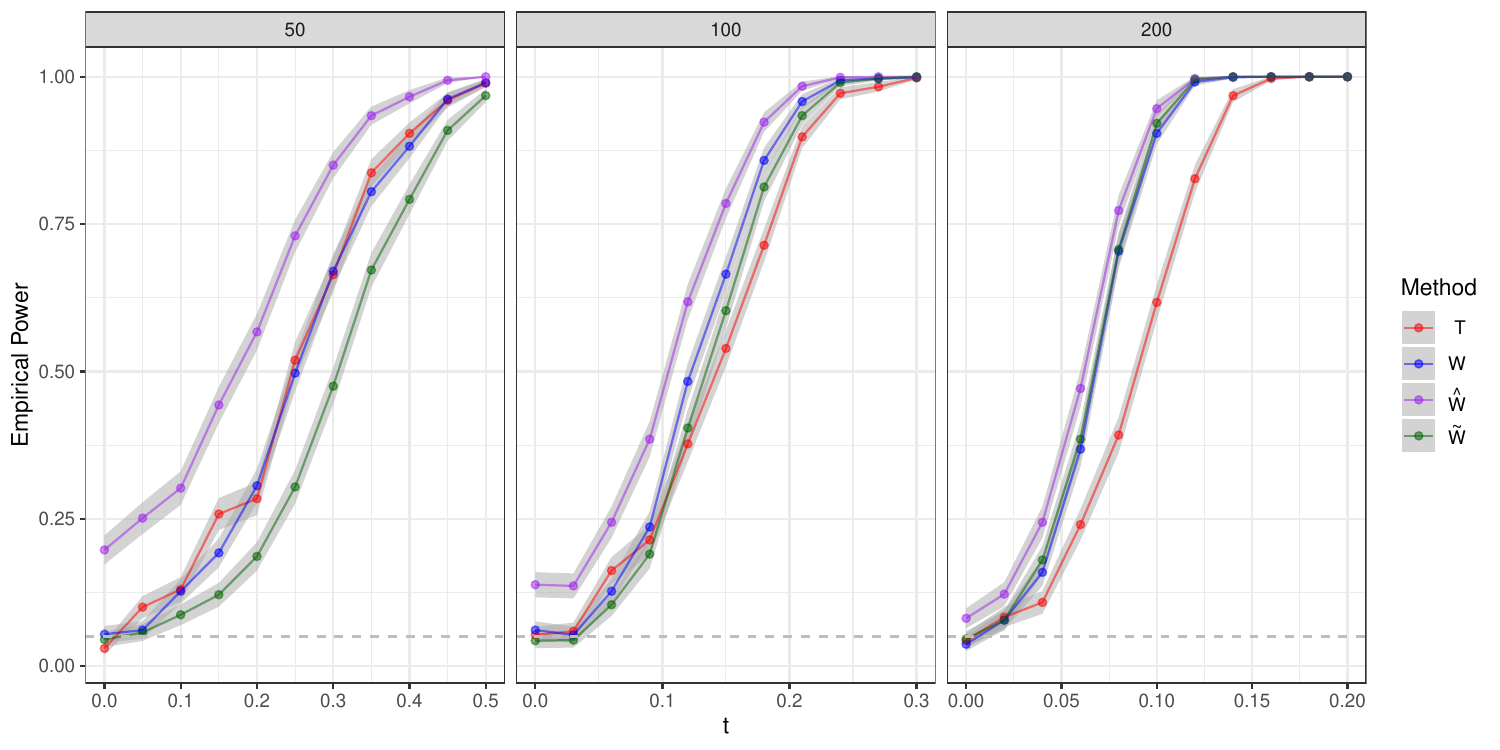}
    \caption{The empirical power of $T$, $W$, $\hat{W}$, and $\tilde{W}$ for testing $H_0: \bvar{I} = \bvar{C}(t)$ for differing values of $t$.}
    \label{fig:Empirical Power}
\end{figure}

First, as $t$ increases the power of each method achieves perfect power. 
Moreover, as $n$ increases, each method achieves perfect power for smaller values of $t$.
For networks of size $n = 50$, there appears to be no difference between the $T$ and $W$ test statistics but as the networks increase to moderate size, ($n = 100, 200$), our proposed test statistic $W$ outperforms the $T$ statistic.
Indeed, $W$ offers an average relative improvement in empirical power over $T$ of $\{1\%, 7\%, 21\%\}$ for networks of size $\{50, 100, 200$\}, respectively.
Both $T$ and $W$ achieve the correct level for each value of $n$ but the test statistic $\hat{W}$ is overpowered for small network sizes. 
This is due to the under estimation of the variance matrix $\Sigma_D(\bvar{x}_i)$.
As $n$ increases, our proposed estimator $\hat{\Sigma}_D(\bvar{x}_i)$ improves and $\hat{W}$ begins to achieve level. 
The corrected version of the $\hat{W}$ statistic, $\tilde{W}$, achieves similar level of power as that of $T$ for networks of size $n = 100$ and outperforms this statistic for networks of size $n = 200$.
The degree of freedom correction for $\tilde{W}$ was $(c_{50}, c_{100}, c_{200}) = (19, 13, 10)$. 
\qed
\end{ex}

The $\hat{W}$ test statistic is a fully data-dependent approach to testing the hypothesis $H_0: \bvar{C}^{(1)} = \bvar{C}^{(2)}$ yet does not achieve level for moderate network sizes.
The level corrected test statistic, $\tilde{W}$, offers comparable empirical power to semi-parametric testing approaches for moderate network sizes but relies on unknown model parameters. 
An estimation scheme for the degrees of freedom of the critical value for the level-correction of $\hat{W}$ will offer an improved fully data-dependent, parametric testing framework that achieves comparable empirical power to semi-parametric approaches.

\section{Discussion}\label{Section: Extensions and Discussion}

In this work we study the omnibus embedding under a heterogenous network model and establish its viability for multiple graph inference beyond the homogeneous network setting.  
We establish an explicit bias-variance tradeoff for latent position estimates provided by the omnibus embedding. 
We reveal an analytic bias expression, derive a uniform concentration bound on the residual term at a rate of $O(m^{3/2}n^{-1/2}\log  nm)$, and prove a central limit theorem which characterizes the distributional properties of the estimator.
These explicit bias and variance expressions enable us to state sufficient conditions for exact recovery in community detection tasks, determine appropriate clustering algorithms for community detection, and develop a test statistic to determine whether two graphs drawn from the ESRDPG share the same weighting matrices.
This analysis offers a first step in theoretically justifying the favorable performance of the omnibus embedding in the presence of heterogeneous network data. 

In what follows we provide remarks on possible extensions of Theorem~\ref{Main Bias Theorem} and Theorem~\ref{CLT} beyond the ESRDPG.
Specifically, we consider the possibility of reducing our assumptions on the weighting matrices $\mathbf{C}^{(g)}$ and the implications of asymptotics in the number of graphs $m$.

In Definition \ref{ESRDPG}, it is required that the $\{\bvar{C}^{(g)}\}_{g=1}^m$ are diagonal and nonnegative.
As discussed in Remark~\ref{Remark:Model Asummptions}, the results presented here can be extend to settings where the ESRDPG allows $\bvar{C}^{(g)}$ to have negative values.
This extension allows the ESRDPG to capture both assortative and disassortative community structures as in the Generalized Random Dot Product Graph of \parencite{rubindelanchy2017}.
By enriching the model class, however, $\tilde{\bvar{P}}$ isn't guaranteed to have $d$ positive eigenvalues which obfuscates the proper embedding approach and dimension.
However, under the assumption that there does not exist $i\in[d]$ such that $\bvar{C}_{ii}^{(g)} = c \leq  0$ for all $g\in[m]$, $\tilde{\bvar{P}}$ has $d$ positive eigenvalues and the results presented in Section~\ref{Section: Main Results} apply, after adjusting the requisite notation.
Furthermore, many of the results can be extended to embeddings which include negative eigenvalues of $\tilde{\bvar{A}}$, though this further complicates interpretations.

Further extending the ESRDPG to include non-diagonal $\bvar{C}^{(g)}$ complicates the interpretation and analysis of the scaling matrices $\bvar{S}^{(g)}$.
For example, if the $\bvar{C}^{(g)}$ are symmetric matrices, the scaling matrices $\bvar{S}^{(g)}$ are intricate functions of the eigenvectors and eigenvalues from the positive and negative definite part of the omnibus matrix of $\{\bvar{C}^{(g)}\}_{g=1}^m$.
While characterizing these matrices is possible, these matrices $\bvar{S}^{(g)}$ are not interpretable in terms of the ESRDPG model parameters.
Moreover, for general symmetric weighting matrices, it may be more appropriate to consider embedding methods that utilize both the negative and positive eigenvalues of $\tilde{\bvar{A}}$.
While we anticipate similar results as those presented here will extend to these embeddings, these approaches may produce node embeddings not in $\R^d$ and restrict our ability to analyze spectral embeddings as latent position estimators. 


Throughout, we assumed that the number of networks $m$ was of fixed size but we can readily extend the results to an asymptotic analysis in the number of networks, $m$. 
Considering the convergence rate presented in Theorem \ref{Main Bias Theorem}, $O(m^{3/2}n^{-1/2}\log nm)$ by letting $n = \omega(m^{3/2 + \xi})$ for $\xi > 0$ we still achieve asymptotic concentration in $m$.
If the number of nonzero $\{\bvar{C}_{ii}^{(g)}\}_{g=1}^m$ grows as $\Theta(m)$, for instance if each weighting matrix has strictly positive entries or if $\{\bvar{C}^{(g)}\}_{g=1}^m$ are sampled i.i.d.\ from a distribution over $\mathcal{C}_F$, then concentration presented in Theorem \ref{Main Bias Theorem} occurs at the rate consistent with \textcite{OmniCLT}, $O(m^{1/2}n^{-1/2}\log nm)$.
Hence, provided $n = \Theta(m^{1/2 + \xi})$ for $\xi>0$ concentration will occur asymptotically in $m$
\begin{align*}
    \max_{h\in[nm]}\|\bvar{R}_h\|_{2}\leq C\frac{\log m^{5/2 + \xi}}{m^{\xi}}.
\end{align*}
This result is of particular interest as we establish that the number of networks can dominate the number of vertices while still achieving concentration of the rows of the omnibus embedding.
This result suggests that the omnibus embedding may be useful in dynamical network applications where the weighting matrices are a discrete time stochastic process $\{\bvar{C}^{(t)}\}_{t=1}^T$.
This stochastic process could impose a dependence structure among edges across layers or among the scaling matrices $\{\bvar{S}^{(t)}\}_{t=1}^T$.
Characterizing this dependency structure for different stochastic processes models will result in a wide array of new theoretical questions as well as potential methodological developments for dynamical network models.  
A first step in this analysis has been presented in \textcite{pantazis2021importance}.

Finally, a full power analysis of the test statistic introduced in Section \ref{Section: Statistical Consequences} will provide further insight into our proposed testing paradigm.
Deriving guarantees on the covariance estimator $\hat{\Sigma}_D(\bvar{x}_i)$ will help in establishing the asymptotic distribution of $\hat{W}_i$ and by extension $\hat{W}$.
These asymptotic distributions could lead to a power analysis for this test statistic and offer insights into a data dependent choice of the degrees of freedom for test statistic $\tilde{W}$.
Moreover, as our analysis allows for testing the hypothesis $H_0: \bvar{C}^{(g)} = \bvar{C}^{(k)}$ for any $g\neq k$, this test provides groundwork for developing a full MANOVA framework for heterogeneous network data.
Finally, understanding the test's power with the alternative hypothesis falling outside the ESRDPG will further enrich the hypothesis testing framework supported by these theoretical findings.

\paragraph{Acknowledgements}

This material is based on research sponsored by the Air Force Research Laboratory and DARPA under agreement number FA8750-20-2-1001. The U.S. Government is authorized to reproduce and distribute reprints for Governmental purposes notwithstanding any copyright notation thereon. The views and conclusions contained herein are those of the authors and should not be interpreted as necessarily representing the official policies or endorsements, either expressed or implied, of the Air Force Research Laboratory and DARPA or the U.S. Government.

\printbibliography


\appendix

\section{Analysis Layout} \label{Layout Appendix}

Our main focus is on the rows of the matrix $\hat{\bvar{L}}\tilde{\bvar{W}}_n- \bvar{L}$ for a sequence of orthogonal matrices $\{\tilde{\bvar{W}}_n\}_{n=1}^{\infty}$. 
We propose that $\tilde{\bvar{W}}_n = \tilde{\bvar{V}}^T\tilde{\bvar{W}}_n^{*T}$ where $\tilde{\bvar{V}}$ and $\tilde{\bvar{W}}_n^{*}$ are rotation matrices to be introduced.
Define $\bvar{L}_S$ to be the $nm\times d$ block matrix whose $g$-th, $n\times d$ block is $\bvar{XS}^{(g)}$. 
Then by adding and subtracting this term, we arrive at our first moment expansion
\begin{align}\label{first_decomp}
    \hat{\bvar{L}}\tilde{\bvar{W}}_n - \bvar{L} = (\hat{\bvar{L}}\tilde{\bvar{W}}_n - \bvar{L}_S) + (\bvar{L}_S - \bvar{L}) =: \bvar{R} + (\bvar{L}_S - \bvar{L})
    \end{align}
From here we will prove the following. 
\begin{enumerate}
    \item The second term in \eqref{first_decomp}, $(\bvar{L}_S - \bvar{L})$, captures the asymptotic bias of the omnibus embedding and is a known matrix that is a function of the weighting matrices $\{\bvar{C}^{(g)}\}_{g=1}^m$ and the latent positions $\bvar{X}$. 
    \item Defining $\bvar{R} = \hat{\bvar{L}}\tilde{\bvar{W}}_n - \bvar{L}_S$ we intend to show 
    \begin{align*}
        \|\bvar{R}\|_{2,\infty}\leq O\left(m^{3/2}n^{-1/2}\log nm\right).
    \end{align*}
\end{enumerate}
We establish the first result directly by defining the $\{\bvar{S}^{(g)}\}_{g=1}^m$ in terms of the $\{\bvar{C}^{(g)}\}_{g=1}^m$.
We prove the second result through a series of perturbation arguments. 
Moving to the second moment, let $\bvar{Z} = \mathrm{ASE}(\tilde{\bvar{P}}, d)$ and consider the expansion 
\begin{align*}
    (\hat{\bvar{L}}\tilde{\bvar{W}}_n - \bvar{L}_S) &= (\hat{\bvar{L}} - \bvar{Z}\tilde{\bvar{V}})\tilde{\bvar{W}}_n + (\bvar{Z}\bvar{W}_n^{*T} - \bvar{L}_S) := \bvar{N} + \bvar{M}
\end{align*}
Heuristically, $\bvar{N}$ describes the variation between the eigenvectors of $\tilde{\bvar{A}}$ and $\tilde{\bvar{P}}$ while $\bvar{M}$ is due to variation between the eigenvectors of $n^{-1}\bvar{X}^T\bvar{X}$ and $\Delta$.
We address this first term, $\bvar{N}$, using a similar expansion given \textcite{OmniCLT}.
In particular, we consider the expansion 
\begin{equation}\label{eq:N_expansion}
    \bvar{N} = (\tilde{\bvar{A}} - \tilde{\bvar{P}})\bvar{U}_{\tilde{\bvar{P}}}\bvar{S}_{\tilde{\bvar{P}}}^{-1/2}\tilde{\bvar{V}}\tilde{\bvar{W}}_n + \bvar{R}^{(2)}\tilde{\bvar{W}}_n
\end{equation}
where $\bvar{R}^{(2)}$ is a further residual term that will converge in probability to zero after scaled by $\sqrt{n}$.
Following arguments from \textcite{EigenCLT} and \textcite{OmniCLT}, we will show the $\sqrt{n}$-scaled rows of the first term in \eqref{eq:N_expansion} will converge in distribution to a mixture of normal random variables with explicit covariance.
Finally, we will show $\sqrt{n}\bvar{M}$ converges to a matrix-vector product where the matrix contains normally distributed entries with degenerate covariances and the vector is a row of $\bvar{X}$.
For concreteness we consider the expansion 
\begin{align*}
    \hat{\bvar{L}}\tilde{\bvar{W}}_n - \bvar{L}_S = \bvar{M} + \bvar{N}
\end{align*}
from which we intend to prove the following
\begin{enumerate}
    \item 
    We will show $\sqrt{n}\bvar{M}|\{\bvar{X}_i = \bvar{x}_i\} \overset{D}{\longrightarrow}N(0, \Sigma_g^{(M)}(\bvar{x}_i))$
    where $\Sigma_g^{(M)}(\bvar{x}_i)$ is rank deficient.
    \item Following the analysis of \textcite{OmniCLT}, we intend to show $
    \sqrt{n}\bvar{N}|\{\bvar{X}_i = \bvar{x}_i\} \overset{D}{\longrightarrow}N(0, \Sigma_g^{(N)}(\bvar{x}_i))$.
\end{enumerate}
Having established the asymptotic bias and variance of the omnibus embedding estimates, we can use these results to prove corollaries and useful in subsequent statistical procedures. 
Before we proceed we provide a table of notation that we utilize in the Appendix in addition to Table~\ref{tab:notation_table2}.

\begin{table}[h!]
    \small
    \centering
    \begin{tabular}{|l|p{.6\textwidth}|}
    \hline
     \textbf{Symbol} & \textbf{Definition}\\
     \hline 
      $\tilde{\bvar{C}}\in\R^{dm\times dm}$ & The omnibus matrix of $\{\bvar{C}^{(g)}\}_{g=1}^m$\\
      $q\in[d]$ & The number of negative eigenvalues of $\tilde{\bvar{C}}$\\
      $\bvar{I}_{d,q}$ & Diagonal matrix with $d$, $1$s and $q$, $-1$s.\\
      $\mathbb{O}(d,q)$ & Set of indefinite orthogonal matrices with signature $(d, q)$\\
      $\bvar{S}_{\perp}^{(g)} = \kappa(2^{-1}[\bvar{C}^{(g)}\bvar{C}_m^{-1/4} - \bvar{C}_m^{1/4}])$ & The negative definition scaling matrix for graph $g\in[m]$ where $\kappa(\cdot)$ removes zero columns $\bvar{C}_m = m^{-1}\sum_{g=1}^m\bvar{C}^{(g)2}$\\
      $\bvar{S}\in\R^{md\times d}$, $\bvar{S}_{\perp}\in\R^{md\times q}$ & Block matrices of $[\bvar{S}^{(g)}]_{g=1}^m$ and $[\bvar{S}_{\perp}^{(g)}]_{g=1}^m$, respectively\\
      $\bvar{K} = [\bvar{S}|\bvar{S}_{\perp}]\in\R^{md\times (d+q)}$ & The concatenation of $\bvar{S}$ and $\bvar{S}_{\perp}$\\
      $\bvar{Z} = \bvar{U}_{\tilde{\bvar{P}}}|\bvar{S}_{\tilde{\bvar{P}}}|^{1/2}\in\R^{nm\times d}$ & The ASE of $\tilde{\bvar{P}}$\\
      $\bvar{Z}_{\perp} = \bvar{U}_{\tilde{\bvar{P}}}^{\perp}|\bvar{S}_{\tilde{\bvar{P}}}^{\perp}|^{1/2}\in\R^{nm\times q}$ & The matrix square root of the negative definite part of $\tilde{\bvar{P}}$\\
      \hline  
    \end{tabular}
    \caption{Notation used consistently throughout the Appendices. }
    \label{tab:notation_table2}
\end{table}

\section{First Moment}\label{First Moment Appendix}

To begin, let $\tilde{\bvar{P}}$ have eigendecomposition $\tilde{\bvar{P}} = [\bvar{U}_{\tilde{\bvar{P}}}|\bvar{U}_{\tilde{\bvar{P}}}^{\perp}](\bvar{S}_{\tilde{\bvar{P}}}\otimes \bvar{S}_{\tilde{\bvar{P}}}^{\perp})[\bvar{U}_{\tilde{\bvar{P}}}|\bvar{U}_{\tilde{\bvar{P}}}^{\perp}]^T$ where $\bvar{S}_{\tilde{\bvar{P}}}$ and $\bvar{S}_{\tilde{\bvar{P}}}^{\perp}$ are the diagonal matrices containing the positive and negative eigenvalues of $\tilde{\bvar{P}}$ in non-increasing order, respectively.
We anticipate that the rows of $\hat{\bvar{L}}$ will concentrate around those of $\bvar{U}_{\tilde{\bvar{P}}}\bvar{S}_{\tilde{\bvar{P}}}^{1/2}$, so our first goal is to relate this matrix to the latent positions $\bvar{X}$.
In doing so, we can analyze the difference between the rows of $\bvar{U}_{\tilde{\bvar{P}}}\bvar{S}_{\tilde{\bvar{P}}}^{1/2}$, properly rotated, and those of the weighted latent positions $\bvar{L}$.

First notice we may write $\tilde{\bvar{P}}$ as follows.
\begin{align*}
    \tilde{\bvar{P}} = \begin{bmatrix}
    \bvar{X} & \dots & \bvar{0}\\
    \vdots & \ddots & \vdots\\
    \bvar{0} & \dots & \bvar{X}
    \end{bmatrix}
    \begin{bmatrix}
    \bvar{C}^{(1)} & \dots & 2^{-1}[\bvar{C}^{(1)} + \bvar{C}^{(m)}]\\
    \vdots & \ddots & \vdots\\
    2^{-1}[\bvar{C}^{(1)} + \bvar{C}^{(m)}] & \dots &\bvar{C}^{(m)}
    \end{bmatrix}
    \begin{bmatrix}
    \bvar{X}^T & \dots & \bvar{0}\\
    \vdots & \ddots & \vdots\\
    \bvar{0} & \dots & \bvar{X}^T
    \end{bmatrix}.
\end{align*}
Therefore, $\tilde{\bvar{P}} = (\bvar{I}\otimes\bvar{X})\tilde{\bvar{C}}(\bvar{I}\otimes\bvar{X})^T$ where $\tilde{\bvar{C}}$ is the omnibus matrix of $\{\bvar{C}^{(g)}\}_{g=1}^m$.
Thus, a first step in characterizing the spectral structure of $\tilde{\bvar{P}}$ is to establish the spectral properties of $\tilde{\bvar{C}}$.
This is the focus of Lemma~\ref{Lemma:Ctil_properties} and Lemma~\ref{Lemma:S_definition}.

\begin{Lemma}\label{Lemma:Ctil_properties}
For some $d$-dimensional inner product distribution $F$, let $\{\bvar{C}^{(g)}\}_{g=1}^m\subset \mathcal{C}_F$.
Then, assuming $\min_{i\in[d]}\max_{g\in[m]}\bvar{C}^{(g)}_{ii} > 0$, $\tilde{\bvar{C}}$ has signature $(d, q)$ where $q\in[d]$. 
\end{Lemma}

\begin{proof}
Let $\lambda(\bvar{M})$ denote the multi-set of non-zero eigenvalues of $\bvar{M}$ and recall $\lambda(\bvar{AB}) = \lambda(\bvar{BA})$.
Let $\bvar{C} = [\bvar{C}^{(1)},\bvar{C}^{(2)},\dots, \bvar{C}^{(m)}]^T\in\R^{md\times d}$ and $\tilde{\bvar{I}} = [\bvar{I}, \bvar{I}, \dots, \bvar{I}]^T\in\R^{md\times d}$.
Notice we can write $\tilde{\bvar{C}}$ as 
\begin{align*}
    \tilde{\bvar{C}} = \frac{1}{2}\begin{bmatrix}
    \bvar{C} & \tilde{\bvar{I}}
    \end{bmatrix}
    \begin{bmatrix}
    \bvar{0} & \bvar{I}\\
    \bvar{I} & \bvar{0}
    \end{bmatrix}
    \begin{bmatrix}
    \bvar{C} & \tilde{\bvar{I}}
    \end{bmatrix}^T.
\end{align*}
Therefore, $\lambda(\tilde{\bvar{C}})$ can be computed by considering the following
\begin{align*}
    \lambda(\tilde{\bvar{C}}) &= \lambda\left(\frac{1}{2}\begin{bmatrix}
    \bvar{C} & \tilde{\bvar{I}}
    \end{bmatrix}^T\begin{bmatrix}
    \bvar{C} & \tilde{\bvar{I}}
    \end{bmatrix}\begin{bmatrix}
    \bvar{0} & \bvar{I}\\
    \bvar{I} & \bvar{0}
    \end{bmatrix}\right) = 
    \lambda\left(\frac{1}{2}\begin{bmatrix}
    \bvar{C}^T\tilde{\bvar{I}} & \bvar{C}^T\bvar{C}\\
    \tilde{\bvar{I}}^T\tilde{\bvar{I}} & \tilde{\bvar{I}}^T\bvar{C}
    \end{bmatrix}\right).
\end{align*}
Each block of this matrix is diagonal so writing the matrix as a sum of the $(i,i)$-th element of each block, we have 
\begin{align*}
    \lambda(\tilde{\bvar{C}}) = \lambda\left(\frac{1}{2}\sum_{i=1}^d\begin{bmatrix}
    \bvar{v}_i^T\bvar{1}_m & \bvar{v}_i^T\bvar{v}_i\\
    \bvar{1}_m^T\bvar{1}_m & \bvar{1}_m^T\bvar{v}_i\\
    \end{bmatrix} \otimes \bvar{e}_i\bvar{e}_i^T\right).
\end{align*}
As this is an orthogonal decomposition (e.g. $\bvar{e}_i\bvar{e}_i^T\bvar{e}_j\bvar{e}_j^T = 0$ for all $i\neq j$), the non-zero eigenvalues of $\tilde{\bvar{C}}$ are the union of the eigenvalues of the summands.
That is, 
\begin{align*}
    \lambda(\tilde{\bvar{C}}) = \bigcup_{i=1}^d\lambda\left(\frac{1}{2}\begin{bmatrix}
    \bvar{v}_i^T\bvar{1}_m & \bvar{v}_i^T\bvar{v}_i\\
    \bvar{1}_m^T\bvar{1}_m & \bvar{1}_m^T\bvar{v}_i\\
    \end{bmatrix} \otimes \bvar{e}_i\bvar{e}_i^T\right)
\end{align*}
Finally, let $\bvar{H}(\bvar{x}) = 2^{-1}(\bvar{x}\bvar{1}_m^T + \bvar{1}_m\bvar{x}^T)$.
By direct calculation, the eigenvalues of $\bvar{H}(\bvar{x})$ are given by $\lambda(\bvar{H}(\bvar{x})) = 2^{-1}(\bvar{x}^T\bvar{1}_m \pm \sqrt{m}\|\bvar{x}\|_2)$.
Moreover, as the eigenvalues of a Kronecker product are the product of the eigenvalues, we can write 
\begin{align*}
    \lambda(\tilde{\bvar{C}}) = \bigcup_{i=1}^d\lambda\left(\bvar{H}(\bvar{v}_i) \otimes \bvar{e}_i\bvar{e}_i^T\right) = \bigcup_{i=1}^d\left\{\frac{1}{2}\left(\bvar{v}_i^T\bvar{1}_m \pm \sqrt{m}\|\bvar{v}_i\|_2\right)\right\}.
\end{align*}
Under the ESRDPG, $\min_{i\in[d]}\max_{g\in[m]}\bvar{C}^{(g)}_{ii} > 0$ ensuring $\bvar{v}_i^T\bvar{1}_m> 0$ and $\|\bvar{v}_2\|_2> 0$ for all $i\in[d]$.
Moreover, as each $\bvar{C}_{ii}^{(g)}\geq 0$, we can also write $\bvar{v}_i^T\bvar{1}_m = \|\bvar{v}_i\|_1$.
Hence, as $\|\bvar{v}_i\|_1 \leq \sqrt{m}\|\bvar{v}_i\|_2$, for all $i\in[d]$
\begin{align*}
    \frac{1}{2}(\|\bvar{v}_i\|_1 - \sqrt{m}\|\bvar{v}_i\|_2) \leq 0 < \frac{1}{2}(\|\bvar{v}_i\|_1 + \sqrt{m}\|\bvar{v}_i\|_2).
\end{align*}
Therefore, $\tilde{\bvar{C}}$ has signature $(d, q)$ where $q\in[d]$.
Moreover, as $\|\bvar{v}_i\|_1 - \sqrt{m}\|\bvar{v}_i\|_2 = 0$ occurs only when $\bvar{v}_i = c\bvar{1}_m$ for some $c>0$, $q$ is the number of dimensions where the set $\{\bvar{C}^{(g)}_{ii}\}_{i=1}^d$ do not equal the same number. 
\end{proof}

As $\tilde{\bvar{C}}$ has signature $(d, q)$, it will be useful to introduce the set of indefinite orthogonal matrices $\mathbb{O}(d, q)$ given by
\begin{align*}
    \mathbb{O}(d, q) &= \{\bvar{Q}\in\R^{(d+q)\times (d+q)}:\bvar{Q}\bvar{I}_{d,q}\bvar{Q}^T = \bvar{I}_{d, q}\}\\
    \bvar{I}_{d,q} &= \text{diag}(\underbrace{1, \ldots, 1}_{\text{$d$, $1$s}}, \underbrace{-1, \ldots, -1}_{\text{$q$,  $-1$s}})
\end{align*}
As $\tilde{\bvar{C}}$ is indefinite, we will frequently encounter matrices of the form $\bvar{M}\bvar{I}_{d,q}\bvar{M}^T$.
Identification of the matrix $\bvar{M}$ can only be completed up to an indefinite orthogonal rotation as $\bvar{M}\bvar{I}_{d,q}\bvar{M}^T = (\bvar{MQ})\bvar{I}_{d,q}(\bvar{MQ})^T$. 
When $q = 0$, $\mathbb{O}(d, 0) = \mathcal{O}^{(d)}$ and this non-identifiability is constrained to an orthogonal matrix $\bvar{W}\in\mathcal{O}^{(d)}$.
With this observations, we're ready to relate $\mathrm{ASE}(\tilde{\bvar{C}}, d)$ to the weighting matrices $\{\bvar{C}^{(g)}\}_{g=1}^m$ which enable us to analytically express the bias in Theorem~\ref{Main Bias Theorem}.

\begin{Lemma}\label{Lemma:S_definition}
Let $\tilde{\bvar{C}}$ have eigendecomposition $[\bvar{U}_{\tilde{\bvar{C}}}|\bvar{U}_{\tilde{\bvar{C}}}^{\perp}][\Lambda_{\tilde{\bvar{C}}}\oplus\Lambda_{\tilde{\bvar{C}}}^{\perp}][\bvar{U}_{\tilde{\bvar{C}}}|\bvar{U}_{\tilde{\bvar{C}}}^{\perp}]^T$ where $\Lambda_{\tilde{\bvar{C}}}\in\R^{d\times d}$ and $\Lambda_{\tilde{\bvar{C}}}^{\perp}\in\R^{q\times q}$ are the positive and negative eigenvalues, respectively, ordered in non-increasing order.
Define $\bvar{H} = \bvar{U}_{\tilde{\bvar{C}}}|\Lambda_{\tilde{\bvar{C}}}|^{1/2}$ and $\bvar{H}_{\perp} = \bvar{U}_{\tilde{\bvar{C}}}^{\perp}|\Lambda_{\tilde{\bvar{C}}}^{\perp}|^{1/2}$.
Denote the $g$-th $d\times d$ block of $\bvar{H}$ as $\bvar{H}^{(g)}$ and the $g$-th $d\times q$ block of $\bvar{H}_{\perp}$ as $\bvar{H}_{\perp}^{(g)}$.
For each $g\in[m]$ define 
\begin{align*}
    \bvar{S}^{(g)}= \frac{\bvar{C}^{(g)}\bvar{C}_m^{-1/4} + \bvar{C}_m^{1/4}}{2} \quad \quad \bvar{S}_{\perp}^{(g)} =\kappa\left( \frac{\bvar{C}^{(g)}\bvar{C}_m^{-1/4} - \bvar{C}_m^{1/4}}{2}\right)
\end{align*}
where $\kappa:\R^{d\times d}\to\R^{d\times q}$ is the function that removes zero columns.
Then, there exists orthogonal matrices $\bvar{W}_H\in\mathcal{O}^{(d)}$ and $\bvar{W}^{\perp}_H\in\mathcal{O}^{(q)}$ such that $\bvar{H}^{(g)}\bvar{W}_H = \bvar{S}^{(g)}$ and $\bvar{H}_{\perp}^{(g)}\bvar{W}_H^{\perp} = \bvar{S}_{\perp}^{(g)}$.
\end{Lemma}

\begin{proof}
First notice, under the assumption $\min_{i\in[d]}\max_{g\in[m]} \bvar{C}^{(g)}_{ii} > 0$, that $(\bvar{C}_m)_{ii} > 1$ for all $i\in[d]$.
Hence $\bvar{C}_m^{-1/4}$ is well defined.
Define $\bvar{S} = [\bvar{S}^{(1)T},\ldots, \bvar{S}^{(m)T}]^T\in\R^{md\times d}$ and $\bvar{S}_{\perp} = [\bvar{S}_{\perp}^{(1)T},\ldots, \bvar{S}_{\perp}^{(m)T}]^T\in\R^{md\times q}$.
Notice the $(g, k)$-th, $d\times d$ block of $[\bvar{S}|\bvar{S}_{\perp}]\bvar{I}_{d,q}[\bvar{S}|\bvar{S}_{\perp}]^T$ 
\begin{align*}
 \left([\bvar{S}|\bvar{S}_{\perp}]\bvar{I}_{d,q}[\bvar{S}|\bvar{S}_{\perp}]^T\right)_{gk} &= \frac{1}{4}[\bvar{C}^{(g)}\bvar{C}^{(k)}\bvar{C}_m^{-1/2} + \bvar{C}^{(g)} + \bvar{C}^{(k)} + \bvar{C}_m^{1/2}]\\
 &- \frac{1}{4}[\bvar{C}^{(g)}\bvar{C}^{(k)}\bvar{C}_m^{-1/2} - \bvar{C}^{(g)} - \bvar{C}^{(k)} - \bvar{C}_m^{1/2}]\\
 &= \frac{\bvar{C}^{(g)} + \bvar{C}^{(k)}}{2}.
\end{align*}
As this equality holds for all $g,k\in[m]$ we have
\begin{align*}
    \tilde{\bvar{C}}= [\bvar{H}|\bvar{H}_{\perp}]\bvar{I}_{d,q}[\bvar{H}|\bvar{H}_{\perp}]^T = [\bvar{S}|\bvar{S}_{\perp}]\bvar{I}_{d,q}[\bvar{S}|\bvar{S}_{\perp}]^T.
\end{align*}
Next, consider the products $\bvar{S}^T\bvar{S}_{\perp}$, $\bvar{S}^T\bvar{S}$, and $\bvar{S}_{\perp}^T\bvar{S}_{\perp}$
\begin{align*}
    \bvar{S}^T\bvar{S}_{\perp} &= \frac{1}{4}\sum_{g=1}^m\left((\bvar{C}^{(g)})^2\bvar{C}_m^{-1/2} - \bvar{C}^{(g)} + \bvar{C}^{(g)} - \bvar{C}_m^{1/2}\right)= \left(\frac{m}{4}\bvar{C}_m\bvar{C}_m^{-1/2} - \frac{m}{4}\bvar{C}_m^{1/2}\right) = 0\\
    \bvar{S}^T\bvar{S} &= \frac{1}{4}\sum_{g=1}^m\left((\bvar{C}^{(g)})^2\bvar{C}_m^{-1/2} + 2\bvar{C}^{(g)} + \bvar{C}_m^{1/2}\right) =  \frac{m}{2}(\bvar{C}_m^{1/2} +  \bar{\bvar{C}})\\ 
    \bvar{S}_{\perp}^T\bvar{S}_{\perp} &= \kappa\left(\frac{1}{4}\sum_{g=1}^m\left((\bvar{C}^{(g)})^2\bvar{C}_m^{-1/2} - 2\bvar{C}^{(g)} + \bvar{C}_m^{1/2}\right)\right)
        = \kappa\left(\frac{m}{2}(\bvar{C}_m^{1/2} - \bar{\bvar{C}})\right).
\end{align*}
As $\bvar{C}_m$ is full rank, we see that $\text{rank}(\bvar{S}) = d$ and $\text{rank}(\bvar{S}_{\perp}) = q$.
This observation with the fact $\bvar{S}_{\perp}^T\bvar{S} = \bvar{0}$ implies that $[\bvar{S}|\bvar{S}_{\perp}]$ is full rank. 
Next, using notation from Lemma~\ref{Lemma:Ctil_properties}, notice $\bvar{C}_m^{1/2} = m^{-1/2}\text{diag}(\|\bvar{v}_i\|_2)$ and $\bar{\bvar{C}} = m^{-1}\text{diag}(\bvar{v}_i^T\bvar{1}_m)$ and therefore
\begin{align*}
\bvar{S}^T\bvar{S} &= \frac{m}{2}\left(\bar{\bvar{C}} + \bvar{C}_m^{1/2}\right) = \text{diag}\left(\frac{\bvar{v}_i^T\bvar{1}_m + \sqrt{m}\|\bvar{v}_i\|_2}{2}\right)\\
-\bvar{S}_{\perp}^T\bvar{S}_{\perp} &= \kappa\left[\frac{m}{2}\left(\bar{\bvar{C}} - \bvar{C}_m^{1/2}\right)\right] = \kappa\left[\text{diag}\left(\frac{\bvar{v}_i^T\bvar{1}_m - \sqrt{m}\|\bvar{v}_i\|_2}{2}\right)\right]
\end{align*}
which are exactly the positive and negative eigenvalues of $\tilde{\bvar{C}}$.
Therefore, there exists permutation matrices $\bvar{Q}_1\in\mathcal{O}^{(d)}$ and $\bvar{Q}_2\in\mathcal{O}^{(q)}$ that satisfy
\begin{align*}
\bvar{Q}_1\bvar{S}^T\bvar{S}\bvar{Q}_1^T &= \Lambda_{\tilde{\bvar{C}}}\\ \bvar{Q}_2\bvar{S}_{\perp}^T\bvar{S}_{\perp}\bvar{Q}_2^T &= |\Lambda_{\tilde{\bvar{C}}}^{\perp}|    
\end{align*}
Define $\bvar{Q} = \bvar{Q}_1\oplus\bvar{Q}_2$, and notice $\bvar{Q}\in\mathcal{O}^{(d+q)}\cap\mathbb{O}(d, q)$.
With these observations, we can can write 
\begin{align*}
    \bvar{Q}[\bvar{S}|\bvar{S}_{\perp}]^T[\bvar{S}|\bvar{S}_{\perp}]\bvar{Q}^T =
    \begin{bmatrix} 
    \Lambda_{\tilde{\bvar{C}}} & \bvar{0}\\
    \bvar{0} & |\Lambda_{\tilde{\bvar{C}}}^{\perp}|
    \end{bmatrix} 
\end{align*}
Set $\bvar{U} = [\bvar{S}|\bvar{S}_{\perp}]([\bvar{S}|\bvar{S}_{\perp}]^T[\bvar{S}|\bvar{S}_{\perp}])^{-1/2}\bvar{Q}^T$ and notice 
\begin{align*}
    \tilde{\bvar{C}}\bvar{U} &= [\bvar{S}|\bvar{S}_{\perp}]\bvar{I}_{d,q}[\bvar{S}|\bvar{S}_{\perp}]^T[\bvar{S}|\bvar{S}_{\perp}]([\bvar{S}|\bvar{S}_{\perp}]^T[\bvar{S}|\bvar{S}_{\perp}])^{-1/2}\bvar{Q}^T\\
    &= [\bvar{S}|\bvar{S}_{\perp}]\bvar{I}_{d,q}([\bvar{S}|\bvar{S}_{\perp}]^T[\bvar{S}|\bvar{S}_{\perp}])^{1/2}\bvar{Q}^T\\
    &= [\bvar{S}|\bvar{S}_{\perp}]\bvar{Q}^T\bvar{Q}\bvar{I}_{d,q}\bvar{Q}^T\bvar{Q}([\bvar{S}|\bvar{S}_{\perp}]^T[\bvar{S}|\bvar{S}_{\perp}])^{1/2}\bvar{Q}^T\\
    &= [\bvar{S}|\bvar{S}_{\perp}]\bvar{Q}^T\bvar{I}_{d,q}(\Lambda_{\tilde{\bvar{C}}}^{1/2}\oplus |\Lambda_{\tilde{\bvar{C}}}^{\perp}|^{1/2})\\
    &= [\bvar{S}|\bvar{S}_{\perp}]\bvar{Q}^T(\Lambda_{\tilde{\bvar{C}}}\oplus |\Lambda_{\tilde{\bvar{C}}}^{\perp}|)^{-1/2}\bvar{Q}\bvar{Q}^T(\Lambda_{\tilde{\bvar{C}}}\oplus \Lambda_{\tilde{\bvar{C}}}^{\perp})\\
    &= [\bvar{S}|\bvar{S}_{\perp}]([\bvar{S}|\bvar{S}_{\perp}]^T[\bvar{S}|\bvar{S}_{\perp}])^{-1/2}\bvar{Q}^T(\Lambda_{\tilde{\bvar{C}}}\oplus \Lambda_{\tilde{\bvar{C}}}^{\perp})\\
    &= \bvar{U}(\Lambda_{\tilde{\bvar{C}}}\oplus \Lambda_{\tilde{\bvar{C}}}^{\perp})\\
\end{align*}
Therefore for some $\bvar{W}\in\mathcal{O}^{(d+q)}\cap\mathbb{O}(d, q)$ that commutes with $\Lambda_{\tilde{\bvar{C}}}\oplus \Lambda_{\tilde{\bvar{C}}}^{\perp}$ we see that $[\bvar{U}_{\tilde{\bvar{C}}}|\bvar{U}_{\tilde{\bvar{C}}}^{\perp}]\bvar{W} = [\bvar{S}|\bvar{S}_{\perp}]([\bvar{S}|\bvar{S}_{\perp}]^T[\bvar{S}|\bvar{S}_{\perp}])^{-1/2}\bvar{Q}^T$.
As $\Lambda_{\tilde{\bvar{C}}}$ and $\Lambda_{\tilde{\bvar{C}}}^{\perp}$ contains the positive and negative eigenvalues of $\tilde{\bvar{C}}$, respectively, we can partition $\bvar{W} = \bvar{W}_1\oplus \bvar{W}_2$ where $\bvar{W}_1\in\mathcal{O}^{(d)}$ and $\bvar{W}_2\in\mathcal{O}^{(q)}$.
Finally, we have 
\begin{align*}
    [\bvar{H}|\bvar{H}_{\perp}]\bvar{W}\bvar{Q} &= [\bvar{U}_{\tilde{\bvar{C}}}|\bvar{U}_{\tilde{\bvar{C}}}^{\perp}](\Lambda_{\tilde{\bvar{C}}}^{1/2}\oplus |\Lambda_{\tilde{\bvar{C}}}^{\perp}|^{1/2})\bvar{WQ}\\
    &=[\bvar{U}_{\tilde{\bvar{C}}}\bvar{W}_1|\bvar{U}_{\tilde{\bvar{C}}}^{\perp}\bvar{W}_2](\Lambda_{\tilde{\bvar{C}}}^{1/2}\oplus |\Lambda_{\tilde{\bvar{C}}}^{\perp}|^{1/2})\bvar{Q}\\
    &= [\bvar{S}|\bvar{S}_{\perp}]([\bvar{S}|\bvar{S}_{\perp}]^T[\bvar{S}|\bvar{S}_{\perp}])^{-1/2}\bvar{Q}^T(\Lambda_{\tilde{\bvar{C}}}^{1/2}\oplus |\Lambda_{\tilde{\bvar{C}}}^{\perp}|^{1/2})\bvar{Q}\\
    &= [\bvar{S}|\bvar{S}_{\perp}]([\bvar{S}|\bvar{S}_{\perp}]^T[\bvar{S}|\bvar{S}_{\perp}])^{-1/2}([\bvar{S}|\bvar{S}_{\perp}]^T[\bvar{S}|\bvar{S}_{\perp}])^{1/2}\\
    &= [\bvar{S}|\bvar{S}_{\perp}].
\end{align*}
Defining $\bvar{W}_H = \bvar{W}_1\bvar{Q}_1\in\mathcal{O}^{(d)}$ and $\bvar{W}_H^{\perp} = \bvar{W}_2\bvar{Q}_2\in\mathcal{O}^{(q)}$ concludes the proof. 
\end{proof}

Having characterized the spectral structure of $\tilde{\bvar{C}}$, we now look to extend these findings to study the spectral structure of $\tilde{\bvar{P}}$.
Notice as $\text{rank}(\bvar{X}) = d$ that $\bvar{X}^T\bvar{X}$ is invertible.
Therefore, we see that 
\begin{align*}
    \lambda(\tilde{\bvar{P}}) = \lambda[(\bvar{I}\otimes \bvar{X})\tilde{\bvar{C}}(\bvar{I}\otimes \bvar{X})^T] = \lambda[(\bvar{I}\otimes (\bvar{X}^T\bvar{X})^{1/2})\tilde{\bvar{C}}(\bvar{I}\otimes (\bvar{X}^T\bvar{X})^{1/2})]
\end{align*}
and by Sylveter's law of inertia, $\tilde{\bvar{P}}$ and $\tilde{\bvar{C}}$ have the same signature $(d,q)$.
That is, under the ESRDPG, $\tilde{\bvar{P}}$ can be indefinite with $q\in[0, d]$ negative eigenvalues characterized by Lemma~\ref{Lemma:Ctil_properties}.
Moreover, as we establish in Lemma~\ref{Lemma:Ctil_properties} $\tilde{\bvar{C}} = [\bvar{S}|\bvar{S}_{\perp}]\bvar{I}_{d,q}[\bvar{S}|\bvar{S}_{\perp}]^T$ which implies
\begin{align*}
    \tilde{\bvar{P}} = [(\bvar{I}\otimes \bvar{X})\bvar{S}|(\bvar{I}\otimes \bvar{X})\bvar{S}_{\perp}]\bvar{I}_{d,q}[(\bvar{I}\otimes \bvar{X})\bvar{S}|(\bvar{I}\otimes \bvar{X})\bvar{S}_{\perp}]^T
\end{align*}
Therefore, we look to relate the eigenvectors corresponding to the positive eigenvalues of $\tilde{\bvar{P}}$ to $(\bvar{I}\otimes\bvar{X})\bvar{S}$.
We specify this relation in Lemma~\ref{Lemma:indefinite_omni}.

\begin{Lemma}\label{Lemma:indefinite_omni}
Let $\tilde{\bvar{P}} = [\bvar{U}_{\tilde{\bvar{P}}}|\bvar{U}_{\tilde{\bvar{P}}}^{\perp}](\bvar{S}_{\tilde{\bvar{P}}}\otimes \bvar{S}_{\tilde{\bvar{P}}}^{\perp})[\bvar{U}_{\tilde{\bvar{P}}}|\bvar{U}_{\tilde{\bvar{P}}}^{\perp}]^T$ and define
$\bvar{Z} = \bvar{U}_{\tilde{\bvar{P}}}\bvar{S}_{\tilde{\bvar{P}}}^{1/2}$ and $\bvar{Z}_{\perp} = \bvar{U}_{\tilde{\bvar{P}}}^{\perp}|\bvar{S}_{\tilde{\bvar{P}}}^{\perp}|^{1/2}$.
Then there exists an indefinite orthogonal matrix $\bvar{Q}\in\mathbb{O}(p,q)$ such that $[\bvar{Z}|\bvar{Z}_{\perp}]\bvar{Q} = (\bvar{I}\otimes\bvar{X})[\bvar{S}|\bvar{S}_{\perp}]$.
\end{Lemma}

\begin{proof}
Recall by Lemma~\ref{Lemma:Ctil_properties}, $\tilde{\bvar{C}} = [\bvar{S}|\bvar{S}_{\perp}]\bvar{I}_{d,q}[\bvar{S}|\bvar{S}_{\perp}]^T$.
For ease of notation, let $\bvar{K} = [\bvar{S}|\bvar{S}_{\perp}]$.
Then we have 
\begin{align*}
    \tilde{\bvar{P}} = (\bvar{I}\otimes \bvar{X})\tilde{\bvar{C}}(\bvar{I}\otimes \bvar{X})^T = [(\bvar{I}\otimes \bvar{X})\bvar{K}]\bvar{I}_{d,q}[(\bvar{I}\otimes \bvar{X})\bvar{K}]^T = \bvar{L}_K\bvar{I}_{d,q}\bvar{L}_K^T
\end{align*}
where $\bvar{L}_K = (\bvar{I}\otimes \bvar{X})\bvar{K}$.
Recall from Lemma~\ref{Lemma:Ctil_properties} that $\text{rank}(\bvar{K}) = d+q$ so $\bvar{L}_K^T\bvar{L}_K$ is full rank. 
With this observation, notice that $\lambda(\bvar{\tilde{\bvar{P}}}) = \lambda(\bvar{L}_K\bvar{I}_{d,q}\bvar{L}_K^T) = \lambda((\bvar{L}_K^T\bvar{L}_K)^{1/2}\bvar{I}_{d,q}(\bvar{L}_K^T\bvar{L}_K)^{1/2})= \lambda(\bvar{L}_K^T\bvar{L}_K\bvar{I}_{d,q})$
and define
\begin{align*}
    (\bvar{L}_K^T\bvar{L}_K)^{1/2}\bvar{I}_{d,q}(\bvar{L}_K^T\bvar{L}_K)^{1/2} & = \bvar{V}_n\bvar{S}_{\tilde{\bvar{P}}}\bvar{V}_n^T.
\end{align*}
Let $\bvar{U} = \bvar{L}_K(\bvar{L}_K^T\bvar{L}_K)^{-1/2}\bvar{V}_n$ and notice 
\begin{align*}
    \tilde{\bvar{P}}\bvar{U} &= \bvar{L}_K\bvar{I}_{d,q}\bvar{L}_K^T\bvar{L}_K(\bvar{L}_K^T\bvar{L}_K)^{-1/2}\bvar{V}_n\\
    &= \bvar{L}_K(\bvar{L}_K^T\bvar{L}_K)^{-1/2}(\bvar{L}_K^T\bvar{L}_K)^{1/2}\bvar{I}_{d,q}(\bvar{L}_K^T\bvar{L}_K)^{1/2}\bvar{V}_n\\
    &= \bvar{L}_K(\bvar{L}_K^T\bvar{L}_K)^{-1/2}\bvar{V}_n\bvar{S}_{\tilde{\bvar{P}}}\\
    &= \bvar{US}_{\tilde{\bvar{P}}}
\end{align*}
which shows that $\bvar{U} = \bvar{U}_{\tilde{\bvar{P}}}\bvar{W}$ for some $\bvar{W}\in\mathbb{O}(d,q)\cap\mathcal{O}^{(d+q)\times(d+q)}$.
Finally, define $\bvar{Q} = |\bvar{S}_{\tilde{\bvar{P}}}|^{-1/2}\bvar{V}_n^T(\bvar{L}_K^T\bvar{L}_K)^{1/2}$ and notice that $\bvar{Q}\in\mathcal{O}(d,q)$ as 
\begin{align*}
    \bvar{Q}\bvar{I}_{d,q}\bvar{Q}^T &= |\bvar{S}_{\tilde{\bvar{P}}}|^{-1/2}\bvar{V}_n^T(\bvar{L}_K^T\bvar{L}_K)^{1/2}\bvar{I}_{d,q}(\bvar{L}_K^T\bvar{L}_K)^{1/2}\bvar{V}_n|\bvar{S}_{\tilde{\bvar{P}}}|^{-1/2}\\
    &= |\bvar{S}_{\tilde{\bvar{P}}}|^{-1/2}\bvar{V}_n^T\bvar{V}_n\bvar{S}_{\tilde{\bvar{P}}}\bvar{V}_n^T\bvar{V}_n|\bvar{S}_{\tilde{\bvar{P}}}|^{-1/2}\\
    &= |\bvar{S}_{\tilde{\bvar{P}}}|^{-1/2}\bvar{S}_{\tilde{\bvar{P}}}|\bvar{S}_{\tilde{\bvar{P}}}|^{-1/2}\\
    &= \bvar{I}_{d,q}. 
\end{align*}
Therefore, we can write $[\bvar{Z}|\bvar{Z}_{\perp}]\bvar{I}_{d,q}[\bvar{Z}|\bvar{Z}_{\perp}]^T = \bvar{L}_K\bvar{I}_{d,q}\bvar{L}_K^T = (\bvar{L}_K\bvar{Q})\bvar{I}_{d,q}(\bvar{L}_K\bvar{Q})^T$ and by post multiplying by $\bvar{U}_{\tilde{\bvar{P}}}|\bvar{S}_{\tilde{\bvar{P}}}|^{-1/2}\bvar{I}_{d,q}$ we can express $[\bvar{Z}|\bvar{Z}_{\perp}]$ as follows
\begin{align*}
    [\bvar{Z}|\bvar{Z}_{\perp}] &= (\bvar{L}_K\bvar{Q})\bvar{I}_{d,q}(\bvar{L}_K\bvar{Q})^T\bvar{U}_{\tilde{\bvar{P}}}|\bvar{S}_{\tilde{\bvar{P}}}|^{-1/2}\bvar{I}_{d,q}\\
    &= \bvar{L}_K|\bvar{S}_{\tilde{\bvar{P}}}|^{-1/2}\bvar{V}_n^T(\bvar{L}_K^T\bvar{L}_K)^{1/2}\bvar{I}_{d,q}(\bvar{L}_K^T\bvar{L}_K)^{1/2}\bvar{V}_n|\bvar{S}_{\tilde{\bvar{P}}}|^{-1/2}\bvar{L}_K^T\bvar{U}_{\tilde{\bvar{P}}}|\bvar{S}_{\tilde{\bvar{P}}}|^{-1/2}\bvar{I}_{d,q}\\
    &= \bvar{L}_K\bvar{I}_{d,q}\bvar{L}_K^T\bvar{U}_{\tilde{\bvar{P}}}|\bvar{S}_{\tilde{\bvar{P}}}|^{-1/2}\bvar{I}_{d,q}\\
    &= \bvar{L}_K\bvar{I}_{d,q}\bvar{L}_K^T\bvar{L}_K(\bvar{L}_K^T\bvar{L}_K)^{-1/2}\bvar{V}_n|\bvar{S}_{\tilde{\bvar{P}}}|^{-1/2}\bvar{I}_{d,q}\\
    &= \bvar{L}_K\bvar{I}_{d,q}(\bvar{L}_K^T\bvar{L}_K)^{1/2}\bvar{V}_n|\bvar{S}_{\tilde{\bvar{P}}}|^{-1/2}\bvar{I}_{d,q}\\
    &= \bvar{L}_K\bvar{I}_{d,q}\bvar{Q}^T\bvar{I}_{d,q}
\end{align*}
Therefore, by post multiplying by $\bvar{Q}$ we see that $[\bvar{Z}|\bvar{Z}_{\perp}]\bvar{Q} = \bvar{L}_K = (\bvar{I}\otimes \bvar{X})[\bvar{S}|\bvar{S}_{\perp}]$ concluding the proof.
\end{proof}

We stress that $\bvar{Q}$ is strictly a function of $\bvar{X}$.
Under the assumption that $F$ has a diagonal second moment, we will show that $\bvar{Q}$ concentrates, after rotation, to the identity $\bvar{I}$. 
Therefore for large $n$ we expect $\bvar{ZW}\approx (\bvar{I}\otimes \bvar{X})\bvar{S}$ for $\bvar{W}\in\mathcal{O}^{(d)}$.
We establish this concentration in Lemma~\ref{Lemma:Q_characterization}.

\begin{Lemma}\parencite[Adapted from][Lemma 2]{agterberg2020nonparametric}\label{Lemma:Q_characterization}
Suppose $\bvar{Q}$ is given as in Lemma~\ref{Lemma:indefinite_omni}.
Then there exists a sequence $\bvar{W}_n\in\mathbb{O}(d,q)\cap\mathcal{O}^{(d+q)}$ such that with high probability 
\begin{align}
    \|\bvar{W}_n\bvar{Q} - \bvar{I}\| &\leq O\left(m\sqrt{\frac{\log nm}{n}}\right)
\end{align}
\end{Lemma}

\begin{proof}
First define $\bvar{Q} = |\bvar{S}_{\tilde{\bvar{P}}}|^{-1/2}\bvar{V}_n^T(\bvar{L}_K^T\bvar{L}_K)^{1/2}$ as in Lemma~\ref{Lemma:indefinite_omni}.
Let $\Delta_K = \bvar{K}^T(\bvar{I}\otimes \Delta)\bvar{K}$ and let $\Delta_K^{1/2}\bvar{I}_{d,q}\Delta_K^{1/2}$ have eigendecomposition $\Delta_K^{1/2}\bvar{I}_{d,q}\Delta_K^{1/2} = \bvar{V}\Lambda\bvar{V}^T$.
Define $\tilde{\bvar{Q}} = |\Lambda|^{-1/2}\bvar{V}^T(\Delta_K)^{1/2}$ for some distinct ordering of the eigenvalues $\Lambda$ and eigenvectors $\bvar{V}$. 
We first will show that $\bvar{Q}$ converges to $\tilde{\bvar{Q}}$ and then show that $\tilde{\bvar{Q}}$ is orthogonal. 

Following directly from the argument of \textcite{agterberg2020nonparametric} we consider the expansion
\begin{align*}
    \|\bvar{Q} - \breve{\bvar{W}}_n^T\tilde{\bvar{Q}}\| &= \left\|\left[\left(\frac{|\bvar{S}_{\tilde{\bvar{P}}}|}{n}\right)^{-1/2} - |\Lambda|^{-1/2}\right]\bvar{V}_n^T\left(\frac{\bvar{L}_K^T\bvar{L}_K}{n}\right)^{1/2}\right\|\\ 
    &+ \left\||\Lambda|^{-1/2}\left[\bvar{V}_n -   \bvar{V}\breve{\bvar{W}}_n\right]^T\left(\frac{\bvar{L}_K^T\bvar{L}_K}{n}\right)^{1/2}\right\|\\ 
    &+\left\||\Lambda|^{-1/2}\breve{\bvar{W}}_n^T\bvar{V}^T\left[\left(\frac{\bvar{L}_K^T\bvar{L}_K}{n}\right)^{1/2} - \Delta_K^{1/2}\right]\right\|. 
\end{align*}
where $\breve{\bvar{W}}_n\in\mathbb{O}(d,q)\cap\mathcal{O}^{(d+q)}$ so $\breve{\bvar{W}}_n$ commutes with $\bvar{I}_{d,q}$, $\Lambda$, and, for sufficiently large $n$, $\bvar{S}_{\tilde{\bvar{P}}}$. 
Repeated use the Davis-Kahan theorem establishes the inequalities 
\begin{align*}
    \left\|(n^{-1}\bvar{L}_K^T\bvar{L}_K)^{1/2} - \Delta_K^{1/2}\right\| &\leq O\left(m\sqrt{\frac{\log nm}{n}}\right)\\
    \left\|(n^{-1}|\bvar{S}_{\tilde{\bvar{P}}}|)^{-1/2} - |\Lambda|^{-1/2}\right\| &\leq O\left(m\sqrt{\frac{\log nm}{n}}\right)\\
    \left\|\bvar{V}_n - \bvar{V}\breve{\bvar{W}}_n\right\| &\leq O\left(m\sqrt{\frac{\log nm}{n}}\right).
\end{align*}

Next notice, under the assumption of that $\Delta$ is diagonal and use of Lemma~\ref{Lemma:Ctil_properties} we see that $\bvar{S}^T(\bvar{I}\otimes \Delta)\bvar{S} = \Delta\bvar{S}^T\bvar{S}$ and $\bvar{S}_{\perp}^T(\bvar{I}\otimes \Delta)\bvar{S}_{\perp} = \Delta_{1:q}\bvar{S}_{\perp}^T\bvar{S}_{\perp}$ where $\Delta_{1:q} = \text{diag}(\Delta_{11}, \ldots,\Delta_{qq})$. 
Therefore, 
\begin{align*}
    \Delta_K &= \begin{bmatrix}
    \bvar{S}^T\Delta\bvar{S} & \bvar{S}^T\Delta\bvar{S}_{\perp}\\
    \bvar{S}_{\perp}^T\Delta\bvar{S} & \bvar{S}_{\perp}^T\Delta\bvar{S}_{\perp}
    \end{bmatrix} = \frac{m}{2}\begin{bmatrix}
    \Delta(\bvar{C}_m^{1/2} + \bar{\bvar{C}}) & \bvar{0}\\
    \bvar{0} & \Delta_{1:q}(\bvar{C}_m^{1/2} - \bar{\bvar{C}})
    \end{bmatrix}
\end{align*}
Therefore, $\Delta_K$ is diagonal.
The eigenvectors of $\Delta_K$, $\bvar{V}$, are specified only up to an reordering of the columns $\bvar{V}$ corresponding to repeated values of $\Delta_K$ and the usual sign flips.
Let $\{\lambda_i\}_{i=1}^k$ be the set of unique values of $\bvar{I}_{d,q}\Delta_K$ each with multiplicity $\{n_i\}_{i=1}^k$ such that $\sum_{i=1}^kn_i = d+q$.
Then for some $\breve{\bvar{O}}_n = \oplus_{i=1}^k\bvar{O}_i$ where $\bvar{O}_i\in\mathcal{O}^{n_i \times n_i}$ we have $\bvar{V}\breve{\bvar{O}}_n = \bvar{I}$.

Finally, as $\Delta_K = \bvar{V}\Lambda\bvar{V}^T$ then $\Delta_K^{1/2} = \bvar{V}|\Lambda|^{1/2}\bvar{V}^T$ which implies 
\begin{align*}
    \tilde{\bvar{Q}} = |\Lambda|^{-1/2}\breve{\bvar{O}}_{n}\breve{\bvar{O}}_{n}^T\bvar{V}^T\bvar{V}\breve{\bvar{O}}_n|\Lambda|^{1/2}\breve{\bvar{O}}_n^T\bvar{V}^T  = \bvar{V}^T = \breve{\bvar{O}}_n
\end{align*}
Setting $\bvar{W}_n = \breve{\bvar{O}}_n^T\breve{\bvar{W}}_n\in\mathbb{O}(d,q)\cap\mathcal{O}^{(d)}$ we see that 
\begin{align*}
    \left\|\bvar{W}_n\bvar{Q} - \bvar{I}\right\| \leq C\|\bvar{Q} - \breve{\bvar{W}}_n^T\tilde{\bvar{Q}}\| \leq O\left(m\sqrt{\frac{\log nm}{n}}\right).
\end{align*}

\end{proof}

As $\bvar{W}_n\in\mathbb{O}(d,q)\cap\mathcal{O}^{(d)}$ we can partition $\bvar{W}_n = \bvar{W}_n^{*}\oplus\bvar{W}_n^{\perp}$.
This result establishes that $\bvar{ZW}_n^{*T}= \bvar{L}_S + O(mn^{-1/2}\log^{1/2}nm)$.
Heuristically, we expect the omnibus embedding $\hat{\bvar{L}}$ to concentrate around $\bvar{Z}$, which in turn concentrates around $\bvar{L}_S$.
By simultaneously controlling these convergences, we can bound the error between $\hat{\bvar{L}}$ and $\bvar{L}_S$. 

We are now ready to present the proof of Theorem~\ref{Main Bias Theorem}.
We closely follow the approach introduced by \textcite{OmniCLT} and, for this reason, only include proofs in which the argument was fundamentally changed by the ESRDPG model. 
Other results we state without proof and we refer the reader to \textcite{OmniCLT}.

\begin{proof}[Proof of Theorem~\ref{Main Bias Theorem}]
Our focus is on the study of the rows $\hat{\bvar{L}}\tilde{\bvar{W}}_n - \bvar{L}$.
Here, $\tilde{\bvar{W}_n} = \tilde{\bvar{V}}^T\bvar{W}_n^{*T}$, where $\tilde{\bvar{V}}$ is introduced in Lemma~\ref{Vtil} and $\bvar{W}_n^{*}$ in Lemma~\ref{Lemma:Q_characterization}.
Then we first consider the expansion
\begin{align*}
    (\hat{\bvar{L}}\tilde{\bvar{W}}_n - \bvar{L}) &=  (\bvar{L}_S - \bvar{L}) +
    (\hat{\bvar{L}}\tilde{\bvar{W}}_n - \bvar{L}_S).
\end{align*}
This first term captures the asymptotic bias of the omnibus embedding.  
Let $h = n(g-1) + i$ for some vertex $i\in[n]$ and graph $g\in[m]$.
Then we can write $(\bvar{L}_S - \bvar{L})_h$ as
\begin{align*}
    (\bvar{L}_S - \bvar{L})_h = (\bvar{S}^{(g)} - \sqrt{\bvar{C}^{(g)}})\bvar{X}_i
\end{align*}
which establishes the bias term in Theorem~\ref{Main Bias Theorem}.
Next, we consider the term $\hat{\bvar{L}}\tilde{\bvar{W}}_n - \bvar{L}_S$.
Consider the expansion
\begin{align*}
    (\hat{\bvar{L}}\tilde{\bvar{W}}_n - \bvar{L}_S) &= (\hat{\bvar{L}}\tilde{\bvar{W}}_n - \bvar{Z}\bvar{W}_n^{*T}) + (\bvar{Z}\bvar{W}_n^{*T} - \bvar{L}_S)\\
    &= (\hat{\bvar{L}} - \bvar{Z}\tilde{\bvar{V}})\tilde{\bvar{W}}_n^{T} + (\bvar{Z}\bvar{W}_n^{*T} - \bvar{L}_S)\\
    &:= \bvar{N} + \bvar{M}.
\end{align*}
To bound $\bvar{N}$, we adopt the expansion used in \textcite{OmniCLT}.
Using the definitions of $\tilde{\bvar{V}}$ as defined as in Lemma \ref{Vtil} and $\bvar{R}_1, \bvar{R}_2$, and $\bvar{R}_3$ as in Lemma \ref{Centered-R-Bounds},
consider the following decomposition 
\begin{align*}
    \hat{\bvar{L}} - \bvar{Z}\tilde{\bvar{V}} &= (\tilde{\bvar{A}} - \tilde{\bvar{P}})\bvar{U}_{\tilde{\bvar{P}}}\bvar{S}_{\tilde{\bvar{P}}}^{-1/2}\tilde{\bvar{V}} +  (\tilde{\bvar{A}} - \tilde{\bvar{P}})\bvar{U}_{\tilde{\bvar{P}}}(\tilde{\bvar{V}}\bvar{S}_{\tilde{\bvar{A}}}^{-1/2} - \bvar{S}_{\tilde{\bvar{P}}}^{-1/2}\tilde{\bvar{V}})\\
    &-\bvar{U}_{\tilde{\bvar{P}}}\bvar{U}_{\tilde{\bvar{P}}}^T(\tilde{\bvar{A}} - \tilde{\bvar{P}})\bvar{U}_{\tilde{\bvar{P}}}\tilde{\bvar{V}}\bvar{S}_{\tilde{\bvar{A}}}^{-1/2}\\
    &+(\bvar{I} - \bvar{U}_{\tilde{\bvar{P}}}\bvar{U}_{\tilde{\bvar{P}}}^T)(\tilde{\bvar{A}} - \tilde{\bvar{P}})\bvar{R}_3\bvar{S}_{\tilde{\bvar{A}}}^{-1/2}\\
    &+ \bvar{R}_1\bvar{S}_{\tilde{\bvar{A}}}^{1/2} + \bvar{U}_{\tilde{\bvar{P}}}\bvar{R}_2\\
    &-(\bvar{U}_{\tilde{\bvar{P}}}^{\perp}\bvar{S}_{\tilde{\bvar{P}}}^{\perp}\bvar{U}_{\tilde{\bvar{P}}}^{\perp})^T\bvar{U}_{\tilde{\bvar{A}}}\bvar{S}_{\tilde{\bvar{A}}}^{-1/2}.
\end{align*}
Using the supporting Lemmas to be developed it holds with high probability that
\begin{align*}
    \|(\tilde{\bvar{A}} - \tilde{\bvar{P}})\bvar{U}_{\tilde{\bvar{P}}}(\tilde{\bvar{V}}\bvar{S}_{\tilde{\bvar{A}}}^{-1/2} - \bvar{S}_{\tilde{\bvar{P}}}^{-1/2}\tilde{\bvar{V}})\| &\leq \frac{Cm^{3/2}d^{1/2}\log^{3/2}nm}{n} && (\text{Lemma \ref{Hoeffding}, \ref{Commuter}})\\
    \|\bvar{U}_{\tilde{\bvar{P}}}\bvar{U}_{\tilde{\bvar{P}}}^T(\tilde{\bvar{A}} - \tilde{\bvar{P}})\bvar{U}_{\tilde{\bvar{P}}}\tilde{\bvar{V}}\bvar{S}_{\tilde{\bvar{A}}}^{-1/2}\|_F&\leq Cd\sqrt{\frac{\log nm}{n}}&& (\text{Lemma \ref{Hoeffding}})\\
    \|\bvar{R}_1\bvar{S}_{\tilde{\bvar{A}}}^{1/2} + \bvar{U}_{\tilde{\bvar{P}}}\bvar{R}_2\|_F &\leq \frac{Cm^{3/2}\log nm}{\sqrt{n}}&& (\text{Lemma \ref{Centered-R-Bounds}})\\
    \|(\bvar{I} - \bvar{U}_{\tilde{\bvar{P}}}\bvar{U}_{\tilde{\bvar{P}}}^T)(\tilde{\bvar{A}} - \tilde{\bvar{P}})\bvar{R}_3\bvar{S}_{\tilde{\bvar{P}}}^{-1/2}\|_F &\leq  \frac{Cm^{5/4}\log nm}{\sqrt{n}}&& (\text{Lemma \ref{Centered-R-Bounds}})\\
    \|(\bvar{U}_{\tilde{\bvar{P}}}^{\perp}\bvar{S}_{\tilde{\bvar{P}}}^{\perp}\bvar{U}_{\tilde{\bvar{P}}}^{\perp})^T\bvar{U}_{\tilde{\bvar{A}}}\bvar{S}_{\tilde{\bvar{A}}}^{-1/2}\|_{F}&\leq C\frac{dm^{3/2}\log nm}{\sqrt{n}} && (\text{Lemma \ref{2d Eigenvectors Bound}})
\end{align*}
Therefore, we with high probability, we can write
\begin{align*}
    \|\bvar{N}\|_F = \|\hat{\bvar{L}} - \bvar{Z}\tilde{\bvar{V}}\|_F = \|(\tilde{\bvar{A}} - \tilde{\bvar{P}})\bvar{U}_{\tilde{\bvar{P}}}\bvar{S}_{\tilde{\bvar{P}}}^{-1/2}\|_F + O\left(\frac{m^{3/2}\log nm}{\sqrt{n}}\right).
\end{align*}
Next, notice
\begin{align*}
    \|(\tilde{\bvar{A}} - \tilde{\bvar{P}})\bvar{U}_{\tilde{\bvar{P}}}\bvar{S}_{\tilde{\bvar{P}}}^{-1/2}\|_{2,\infty} \leq \|(\tilde{\bvar{A}} - \tilde{\bvar{P}})\bvar{U}_{\tilde{\bvar{P}}}\|_{2,\infty}\|\bvar{S}_{\tilde{\bvar{P}}}^{-1/2}\|
\end{align*}
To bound this first term, let $\bvar{u}_j$ be the $j$-th column of $\tilde{\bvar{P}}$.
Then we have
\begin{align*}
    \|(\tilde{\bvar{A}} - \tilde{\bvar{P}})\bvar{U}_{\tilde{\bvar{P}}}\|_{2,\infty} \leq \sqrt{d}\max_{j\in[d]}\|(\tilde{\bvar{A}} - \tilde{\bvar{P}})\bvar{u}_{j}\|_{\infty} = \sqrt{d}\max_{j\in[d]}\max_{h\in[nm]}|(\tilde{\bvar{A}} - \tilde{\bvar{P}})\bvar{u}_j|_h
\end{align*}
An application of Hoeffding's inequality as in Lemma \ref{Hoeffding} shows that $|(\tilde{\bvar{A}} - \tilde{\bvar{P}})\bvar{u}_j|_h \leq Cm^{1/4}\sqrt{c\log nm}$ with probability at least $1 - (nm)^{-c}$.
Changing $c$ only changes the constant $C$ for this element-wise bound.
Hence, choosing $c$ sufficiently large and a union bound shows 
\begin{align*}
    \sqrt{d}\max_{j\in[d]}\max_{h\in[nm]}|(\tilde{\bvar{A}} - \tilde{\bvar{P}})\bvar{u}_j|_h \leq Cdm^{1/4}\sqrt{\log nm}
\end{align*}
with high probability.
Lastly, as $\|\bvar{S}_{\tilde{\bvar{P}}}^{-1/2}\|\leq C(n\sqrt{m})^{-1/2}$ and integrating over $\bvar{X}$ then establishes the bound
\begin{align*}
    \|(\tilde{\bvar{A}} - \tilde{\bvar{P}})\bvar{U}_{\tilde{\bvar{P}}}\bvar{S}_{\tilde{\bvar{P}}}^{-1/2}\|_{2,\infty}\leq Cd\sqrt{\frac{\log nm}{n}}.
\end{align*}
Therefore, with high probability
\begin{align*}
    \|\bvar{N}\|_{2,\infty} \leq \|\hat{\bvar{L}} - \bvar{Z}\tilde{\bvar{V}}\|_{2,\infty} \leq O\left(\frac{m^{3/2}\log nm}{\sqrt{n}}\right).
\end{align*}
Finally, to bound $\|\bvar{M}\|_{2,\infty}$, we use Lemma~\ref{Lemma:Q_characterization} directly
\begin{align*}
    \|\bvar{M}\|_{2,\infty} &\leq \|\bvar{Z}\bvar{W}_n^{*T} - \bvar{L}_S\|_{2,\infty}\\
    &= \left\|\left([\bvar{Z}|\bvar{Z}_{\perp}]\bvar{W}_n^{T} - \bvar{L}_{K}\right)\begin{bmatrix}
    \bvar{I}_{d\times d}\\ \bvar{0}_{q\times d}
    \end{bmatrix}\right\|_{2,\infty}\\
    &= \left\|[\bvar{Z}|\bvar{Z}_{\perp}]\bvar{W}_n^T\left(\bvar{I} - \bvar{W}_n\bvar{Q}\right)\begin{bmatrix}
    \bvar{I}_{d\times d}\\ \bvar{0}_{q\times d}
    \end{bmatrix}\right\|_{2,\infty}\\
    &\leq \|\bvar{Z}|\bvar{Z}_{\perp}\|_{2,\infty}\|\bvar{W}_n\bvar{Q} - \bvar{I}\|\\ 
    &\leq O\left(m\sqrt{\frac{\log nm}{n}}\right).
\end{align*}
\end{proof}

The proof of Theorem \ref{Main Bias Theorem} reveals that the bias is given by coordinate scaling of the matrix $\bvar{X}$.
In addition, we also establish a concentration rate of the corresponding residual term.

To support this result, we begin by stating two key Lemmas, Lemma \ref{Spectral Norm Bound} and Lemma \ref{Eval Order}.
Lemma \ref{Spectral Norm Bound} gives spectral norm control on the difference between $\tilde{\bvar{A}}$ and $\tilde{\bvar{P}}$.
This will allow us to use tools from perturbation theory including Weyl's inequality \parencite{HJ} to show the eigenvalues of $\tilde{\bvar{A}}$ are close to those of $\tilde{\bvar{P}}$. 
Lemma \ref{Eval Order} gives a lower bound on the growth of the eigenvalues of $\tilde{\bvar{P}}$ at a rate of $O(n\sqrt{m})$ from which an application of the Davis-Kahan Theorem \parencite{Davis_1963, DKvariant} will allow us to relate the eigenvectors of $\tilde{\bvar{A}}$ to those of $\tilde{\bvar{P}}$.

\begin{Lemma}\label{Spectral Norm Bound} \parencite[Adapted from][Lemma 2]{OmniCLT} 
Let $\tilde{\bvar{A}}\in\R^{nm\times nm}$ be the omnibus matrix of $\{\bvar{A}^{(g)}\}_{g=1}^m$ where $\{\bvar{A}^{(g)}\}_{g=1}^m\sim \text{ESRDPG}(F, n, \{\bvar{C}^{(g)}\}_{g=1}^m)$. Then w.h.p. $\|\tilde{\bvar{A}} - \tilde{\bvar{P}}\|\leq Cm\sqrt{n\log mn}$
\end{Lemma}

\Lemma\label{Eval Order} \parencite[Adapted from][Observation 2]{OmniCLT} Let $F$ be an inner product distribution on $\R^d$ and let $\bvar{X}_1, \ldots, \bvar{X}_n\overset{i.i.d.}{\sim} F$.
Define $\Delta_S = \bvar{S}^T(\bvar{I}\otimes \Delta)\bvar{S}$.
With probability $1 - d^2/(nm)^2$ for all $i\in[d]$ it holds that $|\lambda_{i}(\tilde{\bvar{P}}) - n\lambda_i(\Delta_S))| \leq C d^2m\sqrt{n\log nm}$.
Moreover, for all $i\in[d]$, $\lambda_i(\tilde{\bvar{P}}) \geq C\delta\sqrt{m}n$.

\begin{proof}
Let $\lambda_1(\tilde{\bvar{P}})\geq \lambda_2(\tilde{\bvar{P}})\geq \dots \geq \lambda_d(\tilde{\bvar{P}})>0$ be the positive $d$ eigenvalues of $\tilde{\bvar{P}}$. 
Then, $\lambda_d(\tilde{\bvar{P}}) = \lambda_d(\bvar{L}_K\bvar{I}_{d,q}\bvar{L}_K^T) = \lambda_d(\bvar{L}_K^T\bvar{L}_K\bvar{I}_{d,q})$.
Notice that $\bvar{L}_K^T\bvar{L}_K\bvar{I}_{d,q} - n\bvar{K}^T(\bvar{I}\otimes \Delta)\bvar{K}\bvar{I}_{d,q}$ takes the form \begin{align*}
    \bvar{L}_K^T\bvar{L}_K\bvar{I}_{d,q} - n\bvar{K}^T(\bvar{I}\otimes \Delta)\bvar{K}\bvar{I}_{d,q} = \sum_{g=1}^m
    \begin{bmatrix}
    \bvar{S}^{(g)T}\\
    \bvar{S}_{\perp}^{(g)T}
    \end{bmatrix}
    \left[\bvar{X}^T\bvar{X} - n\Delta\right]
    \begin{bmatrix}
    \bvar{S}^{(g)} | 
    \bvar{S}_{\perp}^{(g)}
    \end{bmatrix}\bvar{I}_{d,q}.
\end{align*}
Notice we can bound this sum as 
\begin{align*}
    \left\|\sum_{g=1}^m
    \begin{bmatrix}
    \bvar{S}^{(g)T}\\
    \bvar{S}_{\perp}^{(g)T}
    \end{bmatrix}
    \left[\bvar{X}^T\bvar{X} - n\Delta\right]
    \begin{bmatrix}
    \bvar{S}^{(g)} | 
    \bvar{S}_{\perp}^{(g)}
    \end{bmatrix}\right\|_F \leq \sum_{g=1}^m\|\bvar{S}^{(g)}\|_F^2\|\bvar{S}^{(g)}_{\perp}\|_F^2\|\bvar{I}_{d,q}\|_F\|_F^2\|\bvar{X}^T\bvar{X} - n\Delta\|_F
\end{align*}

Notice, $\left(\bvar{X}^T\bvar{X} - n\Delta\right)_{ij} = \sum_{k = 1}^n\left(\bvar{X}_{ki}\bvar{X}_{kj} - \Delta_{ij}\right)$ which is a sum of bounded, i.i.d random variables. 
Applying Hoeffding's inequality yields
\begin{align*}
    \prob\left[\big|\bvar{X}^T\bvar{X} - n\Delta\big|_{ij}\geq 2\sqrt{n\log nm}\right]\leq \frac{2}{n^2m^2}
\end{align*}
Therefore, applying a union bound over the matrix we have  $\|\bvar{X}^T\bvar{X} - n\Delta\|_F\leq 2d^2\sqrt{n\log nm}$ with high probability.
Moreover, as $\{\bvar{S}^{(g)}, \bvar{S}^{(g)}_{\perp}\}_{g=1}^m$ are independent of $n$ we see that $\|\bvar{S}^{(g)}\|_F^2\|\bvar{S}^{(g)}_{\perp}\|_F^2 \leq C$.
Therefore, with high probability 
\begin{align*}
    \left\|\sum_{g=1}^m
    \begin{bmatrix}
    \bvar{S}^{(g)T}\\
    \bvar{S}_{\perp}^{(g)T}
    \end{bmatrix}
    \left[\bvar{X}^T\bvar{X} - n\Delta\right]
    \begin{bmatrix}
    \bvar{S}^{(g)} | 
    \bvar{S}_{\perp}^{(g)}
    \end{bmatrix}\right\|_F \leq C(d+q)\sum_{g=1}^m\|\bvar{X}^T\bvar{X} - n\Delta\|_F \leq Cd^2(d+q)m\sqrt{n\log nm}.
\end{align*}
Recall from Lemma~\ref{Lemma:Q_characterization}, that
\begin{align*}
 \bvar{K}^T(\bvar{I}\otimes \Delta)\bvar{K}\bvar{I}_{d,q} = \bvar{S}^T(\bvar{I}\otimes \Delta)\bvar{S}\oplus -\bvar{S}_{\perp}^T(\bvar{I}\otimes \Delta)\bvar{S}_{\perp}
\end{align*}
is diagonal. 
Therefore, the positive $d$ eigenvalues of $\bvar{K}^T(\bvar{I}\otimes \Delta)\bvar{K}$ correspond to the $d$ eigenvalues of $\Delta_S$.
Using Weyl's inequality and bounding the spectral norm by the Frobenius norm, we have for $i\in[d]$
\begin{align*}
    |\lambda_i(\tilde{\bvar{P}}) - n\lambda_i(\Delta_S)| &= |\lambda_i(\bvar{L}_K^T\bvar{L}_K\bvar{I}_{d,q}) - \lambda_i(n\Delta_S)|\leq \left\|\bvar{L}_K^T\bvar{L}_K\bvar{I}_{d,q} - n\Delta_S\right\|_F \leq Cmd^2\sqrt{n\log nm}
\end{align*}
Using this result with the reverse triangle inequality, we see that for sufficiently large $n$
\begin{align*}
    \lambda_i(\tilde{\bvar{P}}) &\geq |n\lambda_d(\Delta_S) - Cmd^2\sqrt{n\log nm}|\geq Cn\delta\min_{i\in[d]}\left(\|\bvar{v}_i\|_1 + \sqrt{m}\|\bvar{v}_i\|_2\right)
\end{align*}
where $\delta = \min_{i}\Delta_{ii} > 0$.
Finally, $\min_{i\in[d]}\max_{g\in[m]}\bvar{C}_{ii}^{(g)} > 0$, this ensures $\|\bvar{v}_i\|_2 > 0$ and hence 
as $\min_{i\in[d]}\left(\|\bvar{v}_i\|_1 + \sqrt{m}\|\bvar{v}_i\|_2\right) \geq C\sqrt{m}$. 
Therefore 
\begin{align*}
    \lambda_{i}(\tilde{\bvar{P}}) \geq C\delta n\sqrt{m}. 
\end{align*}
\end{proof}

We note that this rate is slower than that presented in \textcite{OmniCLT} by a factor of $\sqrt{m}$ which results in the addition factor of $m$ in Theorem~\ref{Main Bias Theorem} compared to the expression presented in \textcite{OmniCLT}.
However, if $\|\bvar{v}_i\|_2 = \Omega(\sqrt{m})$, then we improve the lower bound $\lambda_d(\tilde{\bvar{P}})\geq O(nm)$ which improves the concentration in Theorem~\ref{Main Bias Theorem} to $O(\sqrt{m/n}\log nm)$, consistent with \textcite{OmniCLT}.
This improvement occurs provided the diagonal elements of $\bvar{C}^{(g)}$ are lower bounded by $\epsilon > 0$.
By allowing $\bvar{C}_{ii} = 0$ for all but one $g\in[m]$, this slower rate $O(n\sqrt{m})$ is a direct consequence in extending the model class of the ESRDPG. 

The next two results are cited directly in the proof of Theorem \ref{Main Bias Theorem}.
These proofs are analogous to the argument given in \textcite{PerfectClustering} and \textcite{OmniCLT}. 
Lemma \ref{Vtil} gives a concentration bound of the eigenvectors $\bvar{U}_{\tilde{\bvar{P}}}^T\bvar{U}_{\tilde{\bvar{A}}}$ while Lemma \ref{Hoeffding} are useful Hoeffding bounds used throughout. 

\begin{Lemma}\label{Vtil} \parencite[Adapted from][Proposition 16]{CommDetSBM} 
Let $\bvar{U}_{\tilde{\bvar{P}}}^T\bvar{U}_{\tilde{\bvar{A}}} = \bvar{V}_1\tilde{\Sigma}\bvar{V}_2^T$ be the singular value decomposition of $\bvar{U}_{\tilde{\bvar{P}}}^T\bvar{U}_{\tilde{\bvar{A}}}$ and let $\tilde{\bvar{V}} = \bvar{V}_1\bvar{V}_2^T$.
Then w.h.p.
\begin{align*}
\|\bvar{U}_{\tilde{\bvar{P}}}^T\bvar{U}_{\tilde{\bvar{A}}} - \tilde{\bvar{V}}\|_F\leq C\frac{dm\log nm}{n}\\
\|(\bvar{U}_{\tilde{\bvar{P}}}^{\perp})^T\bvar{U}_{\tilde{\bvar{A}}}\|_F\leq C\frac{dm\log nm}{n}
\end{align*}
\end{Lemma}


\begin{Lemma}\label{Hoeffding}
With the same notation above an application of Hoeffding's inequality gives with high probability
\begin{align*}
    \|\bvar{U}_{\tilde{\bvar{P}}}^T(\tilde{\bvar{A}} - \tilde{\bvar{P}})\|_F\leq Cm^{3/4}\sqrt{dn\log nm}\\
    \|\bvar{U}_{\tilde{\bvar{P}}}^T(\tilde{\bvar{A}} - \tilde{\bvar{P}})\bvar{U}_{\tilde{\bvar{P}}}\|_F\leq Cd\sqrt{m\log nm}.
\end{align*}
\end{Lemma}

Next, we state three key Lemmas used in the proof of Theorem \ref{Main Bias Theorem}. 
We closely follow a decomposition given in \textcite{CommDetSBM} and \textcite{OmniCLT}. 

\begin{Lemma}\label{Commuter}
Define $\tilde{\bvar{V}} = \bvar{V}_1\bvar{V}_2^T$ as in Lemma \ref{Vtil}. Then with high probability
\begin{align*}
    \|\tilde{\bvar{V}}\bvar{S}_{\tilde{\bvar{A}}} - \bvar{S}_{\tilde{\bvar{P}}}\tilde{\bvar{V}}\|_F&\leq Cdm^2\log nm\\
    \|\tilde{\bvar{V}}\bvar{S}_{\tilde{\bvar{A}}}^{1/2} - \bvar{S}_{\tilde{\bvar{P}}}^{1/2}\tilde{\bvar{V}}\|_F&\leq C\frac{m^{3/2}\log nm}{\sqrt{n}}\\
    \|\tilde{\bvar{V}}\bvar{S}_{\tilde{\bvar{A}}}^{-1/2} - \bvar{S}_{\tilde{\bvar{P}}}^{-1/2}\tilde{\bvar{V}}\|_F&\leq C\frac{m^{3/4}\log nm}{n^{3/2}}
\end{align*}
\end{Lemma}


We place a bound on the remaining residual terms in Lemma \ref{Centered-R-Bounds} and Lemma \ref{2d Eigenvectors Bound}. 
We do this directly following Lemma 5 in \textcite{OmniCLT}, using Lemmas \ref{Spectral Norm Bound},  \ref{Vtil}, \ref{Hoeffding}, \ref{Commuter} stated above.
These Lemmas introduce the leading order concentration rate and the argument follows directly from the above results combined with the strategy proposed in \cite{OmniCLT}. 

\begin{Lemma}\label{Centered-R-Bounds}
Define $\bvar{R}_1 = \bvar{U}_{\tilde{\bvar{P}}}\bvar{U}_{\tilde{\bvar{P}}}^T\bvar{U}_{\tilde{\bvar{A}}} - \bvar{U}_{\tilde{\bvar{P}}}\tilde{\bvar{V}}$, $\bvar{R}_2 = \tilde{\bvar{V}}\bvar{S}_{\tilde{\bvar{A}}}^{1/2} - \bvar{S}_{\tilde{\bvar{P}}}^{1/2}\tilde{\bvar{V}}$, and $\bvar{R}_3 = \bvar{U}_{\tilde{\bvar{A}}} -\bvar{U}_{\tilde{\bvar{P}}}\bvar{U}_{\tilde{\bvar{P}}}^T\bvar{U}_{\tilde{\bvar{A}}} + \bvar{R}_1$. 
Then with high probability 
\begin{align*}
\|\bvar{R}_1\bvar{S}_{\tilde{\bvar{A}}}^{1/2} + \bvar{U}_{\tilde{\bvar{P}}}\bvar{R}_2\|_F&\leq \frac{Cm^{3/2}\log mn}{\sqrt{n}}\\
\|(\bvar{I} - \bvar{U}_{\tilde{\bvar{P}}}\bvar{U}_{\tilde{\bvar{P}}}^T)(\tilde{\bvar{A}} - \tilde{\bvar{P}})\bvar{R}_3\bvar{S}_{\tilde{\bvar{A}}}^{-1/2}\|_F &\leq \frac{Cm\log nm}{\sqrt{n}}
\end{align*}
\end{Lemma}

 
 \begin{Lemma}\label{2d Eigenvectors Bound}
 Let $\tilde{\bvar{P}}^{\perp} = -(\bvar{U}_{\tilde{\bvar{P}}}^{\perp}\bvar{S}_{\tilde{\bvar{P}}}^{\perp}\bvar{U}_{\tilde{\bvar{P}}}^{\perp})^T\in\R^{nm\times nm}$ be the negative definite part of $\tilde{\bvar{P}}$.
 Then with high probability
 \begin{align*}
     \|\tilde{\bvar{P}}^{\perp}\bvar{U}_{\tilde{\bvar{A}}}\bvar{S}_{\tilde{\bvar{A}}}^{-1/2}\|_{2,\infty}\leq C\frac{dm\log nm}{n}\\
     \|\tilde{\bvar{P}}^{\perp}\bvar{U}_{\tilde{\bvar{A}}}\bvar{S}_{\tilde{\bvar{A}}}^{-1/2}\|_{F}\leq C\frac{dm^{3/2}\log nm}{\sqrt{n}}
 \end{align*}
 \end{Lemma}
 
\begin{proof}
 First recall that the rows of $\tilde{\bvar{U}}_{\tilde{\bvar{P}}}\tilde{\bvar{S}}_{\tilde{\bvar{P}}}^{1/2}$ are bounded in Euclidean norm by $1$.
 Using this fact and the fact $\|\tilde{\bvar{S}}_{\tilde{\bvar{P}}}^{1/2}\|\leq C(nm)^{1/2}$ and $\|\tilde{\bvar{S}}_{\tilde{\bvar{A}}}^{1/2}\|\leq C(nm)^{-1/2}$ we have 
  \begin{align*}
     \|\bvar{U}_{\tilde{\bvar{P}}}^{\perp}\bvar{S}_{\tilde{\bvar{P}}}^{\perp}\bvar{U}_{\tilde{\bvar{P}}}^{\perp})^T\bvar{U}_{\tilde{\bvar{A}}}\bvar{S}_{\tilde{\bvar{A}}}^{-1/2}\|_{2,\infty}\leq \|\bvar{U}_{\tilde{\bvar{P}}}^{\perp}|\bvar{S}_{\tilde{\bvar{P}}}^{\perp}|^{1/2}\|_{2,\infty}\||\bvar{S}_{\tilde{\bvar{P}}}^{\perp}|^{1/2}\|\|(\bvar{U}_{\tilde{\bvar{P}}}^{\perp})^T\bvar{U}_{\tilde{\bvar{A}}}\|\|\bvar{S}_{\tilde{\bvar{A}}}^{-1/2}\| \leq C\|\tilde{\bvar{U}}_{\tilde{\bvar{P}}}^T\bvar{U}_{\tilde{\bvar{A}}}\|_F
 \end{align*}
  Then applying Lemma \ref{Vtil} we have the result 
  \begin{align*}
      \|\tilde{\bvar{P}}^{\perp}\bvar{U}_{\tilde{\bvar{A}}}\bvar{S}_{\tilde{\bvar{A}}}^{-1/2}\|_{2,\infty}\leq C\|\tilde{\bvar{U}}_{\tilde{\bvar{P}}}^T\bvar{U}_{\tilde{\bvar{A}}}\|_F \leq C\frac{dm\log nm}{n}.
  \end{align*}
  We can use this result to directly establish the Frobenius norm bound.
  \begin{align*}
      \|\tilde{\bvar{P}}^{\perp}\bvar{U}_{\tilde{\bvar{A}}}\bvar{S}_{\tilde{\bvar{A}}}^{-1/2}\|_F = \sqrt{\sum_{i=1}^{nm}\|(\tilde{\bvar{P}}^{-}\bvar{U}_{\tilde{\bvar{A}}}\bvar{S}_{\tilde{\bvar{A}}}^{-1/2})_i\|_2^2}\leq \sqrt{nm}\|\tilde{\bvar{P}}^{-}\bvar{U}_{\tilde{\bvar{A}}}\bvar{S}_{\tilde{\bvar{A}}}^{-1/2}\|_{2,\infty} \leq C\frac{m^{3/2}\log nm}{\sqrt{n}}
  \end{align*}
 
 \end{proof}
 
This concludes our analysis of the first moment properties of the omnibus embedding under the ESRDPG. 
By first studying the spectral structure of the expected omnibus matrix $\tilde{\bvar{P}}$ we are able to express the omnibus embedding of $\tilde{\bvar{P}}$, $\bvar{Z}$, in terms of the latent positions $\bvar{X}$ and the corresponding scaling matrices $\{\bvar{S}^{(g)}\}_{g=1}^m$ which capture the bias of the omnibus embedding. 
Then by demonstrating spectral bound control on the difference of $\tilde{\bvar{A}}$ and $\tilde{\bvar{P}}$ and developing a lower bond on the eigenvalues of $\tilde{\bvar{P}}$ we successfully employ techniques from perturbation theory, as sketched by \textcite{OmniCLT}, to provide a uniform concentration rate of the residual term. 
Next, we turn our attention to the distributional proprieties of this residual term. 

\section{Second Moment}\label{Second Moment Appendix}

In this section we focus on the distributional properties of the residual terms introduced in Theorem \ref{Main Bias Theorem}.
We further factor this residual into three terms, an additional residual term, and two terms with an asymptotic distribution.
We will prove that this residual term converges in probability to 0 after being scaled by $\sqrt{n}$.
Next, following follow the approach introduced by \textcite{OmniCLT} \textit{mutatis mutandis}, we show one of these residual terms converges to a mixture of normal random variables and specify its variance explicitly. 
Finally, we establish this final residual term converges to a randomly projected vector in the direction of $\bvar{X}_i$, with mean zero.
We only include proofs in which the argument was fundamentally changed by the ESRDPG model. 
We first prove Theorem \ref{CLT}; highlighting results that will be justified later in the Appendix. 

\begin{proof}[Proof of Theorem~\ref{CLT}]
First consider the decomposition utilized in the proof of Theorem~\ref{Main Bias Theorem}
\begin{align*}
    (\hat{\bvar{L}}\tilde{\bvar{W}}_n - \bvar{L}_S) = (\hat{\bvar{L}} - \bvar{Z}\tilde{\bvar{V}})\tilde{\bvar{W}}_n + (\bvar{Z}\bvar{W}_n^{*T} - \bvar{L}_S) = \bvar{N} + \bvar{M}.
\end{align*}
We further expand $\bvar{N} = \bvar{H} + \bvar{R}^{(2)}\tilde{\bvar{W}}_n$
where where 
\begin{align*}
    \bvar{H} = (\tilde{\bvar{A}} - \tilde{\bvar{P}})\bvar{U}_{\tilde{\bvar{P}}}\bvar{S}_{\tilde{\bvar{P}}}^{-1/2}\tilde{\bvar{V}}\tilde{\bvar{W}}_n
\end{align*}
and $\bvar{R}^{(2)}$ is given by 
\begin{align*}
    \bvar{R}^{(2)} &= (\tilde{\bvar{A}}-\tilde{\bvar{P}}) \bvar{U}_{\tilde{\bvar{P}}}\left(\tilde{\bvar{V}} \bvar{S}_{\tilde{\bvar{A}}}^{-1/2}-\bvar{S}_{\tilde{\bvar{P}}}^{-1/2} \tilde{\bvar{V}}\right)-\bvar{U}_{\tilde{\bvar{P}}} \bvar{U}_{\tilde{\bvar{P}}}^{T}(\tilde{\bvar{A}}-\tilde{\bvar{P}}) \bvar{U}_{\tilde{\bvar{P}}} \tilde{\bvar{V}} \bvar{S}_{\tilde{\bvar{A}}}^{-1 / 2} \\ 
&+\left(\bvar{I}-\bvar{U}_{\tilde{\bvar{P}}} \bvar{U}_{\tilde{\bvar{P}}}^{T}\right)(\tilde{\bvar{A}}-\tilde{\bvar{P}}) \bvar{R}_{3} \bvar{S}_{\tilde{\bvar{A}}}^{-1/2}+\bvar{R}_{1} \bvar{S}_{\tilde{\bvar{A}}}^{1/2} + \bvar{U}_{\tilde{\bvar{P}}} \bvar{R}_{2} -(\bvar{U}_{\tilde{\bvar{P}}}^{\perp}\bvar{S}_{\tilde{\bvar{P}}}^{\perp}\bvar{U}_{\tilde{\bvar{P}}}^{\perp})^T\bvar{U}_{\tilde{\bvar{A}}}\bvar{S}_{\tilde{\bvar{A}}}^{-1/2}.
\end{align*}
This yields our second moment decomposition 
\begin{equation}\label{eq:second_moment_expansion}
\hat{\bvar{L}}\bvar{W}_n - \bvar{L}_S = \bvar{M} + \bvar{H} + \bvar{R}^{(2)}\tilde{\bvar{W}}_n.    
\end{equation}
In Lemma~\ref{Lemma:mu-distribution} and Lemma~\ref{Power Method CLT}, we first establish conditional on $\{\bvar{X}_i = \bvar{x}_i\}$ 
\begin{align*}
    \sqrt{n}\bvar{H}_h |\{\bvar{X}_i = \bvar{x}_i\} &\overset{D}{\longrightarrow} N(\bvar{0}, \Sigma_g^{(N)}(\bvar{x}_i))\\
    \sqrt{n}\bvar{M}_h |\{\bvar{X}_i = \bvar{x}_i\} &\overset{D}{\longrightarrow} N(\bvar{0}, \Sigma_g^{(M)}(\bvar{x}_i))
\end{align*}
Next, as $\tilde{\bvar{W}}_n$ is orthogonal, we show that $\sqrt{n}\bvar{R}_h^{(2)}\overset{\prob}{\longrightarrow}0$ in a series of subsequent Lemmas. 
\begin{align*}
    \sqrt{n}\bvar{R}_h^{(n)} =&\phantom{+} \sqrt{n}[(\tilde{\bvar{A}} - \tilde{\bvar{P}})\bvar{U}_{\tilde{\bvar{P}}}(\tilde{\bvar{V}}\bvar{S}_{\tilde{\bvar{A}}}^{-1/2} - \bvar{S}_{\tilde{\bvar{P}}}^{-1/2}\tilde{\bvar{V}})]_h && (\text{Lemma } \ref{root n: Easy Residual})\\
    &-\sqrt{n}[\bvar{U}_{\tilde{\bvar{P}}}\bvar{U}_{\tilde{\bvar{P}}}^T(\tilde{\bvar{A}} - \tilde{\bvar{P}})\bvar{U}_{\tilde{\bvar{P}}}\tilde{\bvar{V}}\bvar{S}_{\tilde{\bvar{A}}}^{-1/2}]_h&& (\text{Lemma } \ref{root n: Easy Residual})\\
    &+\sqrt{n}[(\bvar{I} - \bvar{U}_{\tilde{\bvar{P}}}\bvar{U}_{\tilde{\bvar{P}}}^T)(\tilde{\bvar{A}} - \tilde{\bvar{P}})\bvar{R}_3\bvar{S}_{\tilde{\bvar{A}}}^{-1/2}]_h && (\text{Lemma } \ref{root n: Exchangeable})\\
    &+ \sqrt{n}[\bvar{R}_1\bvar{S}_{\tilde{\bvar{A}}}^{1/2} + \bvar{U}_{\tilde{\bvar{P}}}\bvar{R}_2]_h &&(\text{Lemma } \ref{root n: Easy Residual})\\
    &-\sqrt{n}[(\bvar{U}_{\tilde{\bvar{P}}}^{\perp}\bvar{S}_{\tilde{\bvar{P}}}^{\perp}\bvar{U}_{\tilde{\bvar{P}}}^{\perp})^T\bvar{U}_{\tilde{\bvar{A}}}\bvar{S}_{\tilde{\bvar{A}}}^{-1/2}]_h && (\text{Lemma \ref{root n: Easy Residual}})
\end{align*}
Finally, we employ Slutsky's Theorem and the Lebesgue Dominated Convergence Theorem so integrating over all possible values of $\bvar{x}_i$ yields the result 
\begin{align*}
    \lim_{n\to\infty}\prob\left[\sqrt{n}(\bvar{M} + \bvar{N})_h \leq \bvar{x}\right] = \int_{\text{supp}(F)}\Phi(\bvar{x}; \bvar{0}, \Sigma_g^{(N)}(\bvar{y}) + \Sigma_g^{(M)}(\bvar{y}) + \Sigma_g^{(N,M)}(\bvar{y}))dF(\bvar{y}) 
\end{align*}
where $\Phi(\bvar{x}; \mu, \Sigma)$ is the normal cumulative distribution function with mean $\mu$ and covariance matrix $\Sigma$ evaluated at $\bvar{x}$ and $\Sigma_g^{(N,M)}(\bvar{y})$ is the covariance between $\bvar{M}_h$ and $\bvar{N}_h$. 
\end{proof}

\begin{Lemma}\label{Lemma:mu-distribution}
Let $\bvar{M}_h = (\bvar{Z}\bvar{W}_n^{*T} - \bvar{L}_S)_h$.
Then, conditioning on the event $\{\bvar{X}_i = \bvar{x}_i\}$, we have
\begin{align*}
    \sqrt{n}\bvar{M}_h|\{\bvar{X}_i = \bvar{x}_i\} \overset{D}{\longrightarrow} N(0, \Sigma_g^{(M)}(\bvar{x}_i))
\end{align*}
\end{Lemma}
\begin{proof}
For ease of notation, for all that follows condition on the event $\{\bvar{X}_i = \bvar{x}_i\}$.
Notice we can rewrite $\bvar{M}$ as 
\begin{align*}
    \bvar{M} &= \left(\bvar{Z}\bvar{W}_n^{*T} - \bvar{L}_S\right) = \left([\bvar{Z}|\bvar{Z}_{\perp}]\bvar{W}_n^T - \bvar{L}_K\right)\begin{bmatrix}
    \bvar{I}_{d\times d}\\\bvar{0}_{q\times d}\end{bmatrix} 
    = \bvar{L}_K\left(\bvar{Q}^{-1}\bvar{W}_n^T - \bvar{I}\right)\begin{bmatrix}
    \bvar{I}_{d\times d}\\\bvar{0}_{q\times d}\end{bmatrix}\\ 
    &= \bvar{L}_K\bvar{I}_{d,q}\left(\bvar{W}_n\bvar{Q} - \bvar{I}\right)^T\begin{bmatrix}
    \bvar{I}_{d\times d}\\\bvar{0}_{q\times d}\end{bmatrix}.
\end{align*}
Then, for $h = n(g-1) + i$ for some $i\in[n]$ and $g\in[m]$ we have
\begin{align*}
    \bvar{M}_h = [\bvar{I}_{d\times d}, \bvar{0}_{d\times q}]\left(\bvar{W}_n\bvar{Q} - \bvar{I}\right)[\bvar{S}^{(g)},  \bvar{S}^{(g)}_{\perp}]^T\bvar{x}_i
\end{align*}

For a positive definite matrix $\bvar{B}\in\R^{p\times p}$, let $ f_{\bvar{Q}}:\R^{p\times p}\to\R^{p\times p}$ be given by 
\begin{align*}
    f_{\bvar{Q}}(\bvar{B}) = [(\bvar{B}^{1/2}\bvar{I}_{d,q}\bvar{B}^{1/2})^{-1/2}]^T\bvar{B}^{1/2}
\end{align*} where $\bvar{B}^{1/2} = \bvar{U}_{\bvar{B}}|\bvar{S}_{\bvar{B}}|^{1/2}$ and $\bvar{B}^{-1/2} = \bvar{U}_{\bvar{B}}|\bvar{S}_{\bvar{B}}|^{-1/2}$ corresponding to the eigen-decomposition $\bvar{B} = \bvar{U}_{\bvar{B}}\bvar{S}_{\bvar{B}}\bvar{U}_{\bvar{B}}^T$.
Notice that $f_{\bvar{Q}}(\bvar{B}) = f_{\bvar{Q}}(c\bvar{B})$ for any $c\neq 0$ as 
\begin{align*}
    f_{\bvar{Q}}(\bvar{B}) &= [(\bvar{B}^{1/2}\bvar{I}_{d,q}\bvar{B}^{1/2})^{-1/2}]^T\bvar{B}^{1/2} = [((c\bvar{B})^{1/2}\bvar{I}_{d,q}(c\bvar{B})^{1/2}))^{-1/2}]^T(c\bvar{B})^{1/2}.
\end{align*}
Let $\Delta_K = \bvar{K}^T(\bvar{I}\otimes \Delta)\bvar{K}$ and with these observations, notice we can write 
\begin{align*}
    \sqrt{n}\bvar{M}_h &= \sqrt{n}[\bvar{I}_{d\times d},\bvar{0}_{d\times q}](\bvar{W}_nf_{\bvar{Q}}(\bvar{L}_K^T\bvar{L}_K) - f_{\bvar{Q}}(\Delta_K))\bvar{I}_{d,q}[\bvar{S}^{(g)},  \bvar{S}^{(g)}_{\perp}]^T\bvar{x}_i\\ 
    &= \sqrt{n}[\bvar{I}_{d\times d},\bvar{0}_{d\times q}](\bvar{W}_nf_{\bvar{Q}}(n^{-1}\bvar{L}_K^T\bvar{L}_K) - f_{\bvar{Q}}(\Delta_K))\bvar{I}_{d,q}[\bvar{S}^{(g)},  \bvar{S}^{(g)}_{\perp}]^T\bvar{x}_i
\end{align*}
We note that function $f_{\bvar{Q}}(\bvar{B})$ is only defined up to some ordering of the eigenvalues of $\bvar{B}$.
In our context, $f_{\bvar{Q}}(\Delta_K)$ is well defined as we fix an ordering of the eigenvalues of $\Delta$ at the outset.
Moreover, the rotation matrix $\bvar{W}_n$ ensures that $f_{\bvar{Q}}(n^{-1}\bvar{L}_K^T\bvar{L}_K)$ is rotated to align the with the eigenvectors of $\Delta_K$.

Recall from the proof of Lemma~\ref{Eval Order}, that $\|n^{-1}\bvar{L}_K^T\bvar{L}_K - \Delta_K\|_F \leq O\left(d^2m\sqrt{\frac{\log nm}{n}}\right)$ hence $n^{-1}\bvar{L}_K^T\bvar{L}_K|\{\bvar{X}_i = \bvar{x}_i\}\overset{\prob}{\longrightarrow}\Delta_K$.
Defining $\bvar{K}^{(g)} = [\bvar{S}^{(g)}|\bvar{S}_{\perp}^{(g)}]$ and scaling this sum by $\sqrt{n}$ yields
\begin{align*}
    \sqrt{n}\text{vec}\left[\sum_{g=1}^m\bvar{K}^{(g)T}(n^{-1}\bvar{X}^T\bvar{X} - \Delta)\bvar{K}^{(g)}\right] &=\frac{1}{\sqrt{n}}\sum_{j=1}^n\sum_{g=1}^m\text{vec}(\bvar{K}^{(g)T}(\bvar{X}_j\bvar{X}_j^T - \Delta)\bvar{K}^{(g)})\\
    &=\frac{1}{\sqrt{n}}\sum_{j=1}^n\sum_{g=1}^m(\bvar{K}^{(g)}\otimes \bvar{K}^{(g)})^T\text{vec}(\bvar{X}_j\bvar{X}_j^T - \Delta) \\
    &=\frac{1}{\sqrt{n}}\sum_{j\neq i}\sum_{g=1}^m(\bvar{K}^{(g)}\otimes \bvar{K}^{(g)})^T(\bvar{X}_j\otimes \bvar{X}_j - \text{vec}(\Delta))\\
    &+\frac{1}{\sqrt{n}}\sum_{g=1}^m(\bvar{K}^{(g)}\otimes \bvar{K}^{(g)})^T(\bvar{x}_i\otimes \bvar{x}_i - \text{vec}(\Delta)).
\end{align*}
As $\bvar{K}^{(g)}$ is bounded and independent of $n$, $n^{-1/2}\sum_{g=1}^m(\bvar{K}^{(g)}\otimes \bvar{K}^{(g)})^T(\bvar{x}_i\otimes \bvar{x}_i - \text{vec}(\Delta))\overset{\prob}{\longrightarrow}\bvar{0}$.
Then by the multivariate central limit theorem $\sqrt{n}\text{vec}(\bvar{L}_K^T\bvar{L}_K - \Delta_K)|\{\bvar{X}_i = \bvar{x}_i\} \overset{D}{\longrightarrow}N(0, \Sigma_{\otimes})$
where $\Sigma_{\otimes}$ is given by
\begin{align*}
    \Sigma_{\otimes} = \sum_{g=1}^m(\bvar{K}^{(g)}\otimes \bvar{K}^{(g)})^T\E[(\bvar{y}\otimes \bvar{y}- \text{vec}(\Delta))(\bvar{y}\otimes \bvar{y} - \text{vec}(\Delta))^T](\bvar{K}^{(g)}\otimes \bvar{K}^{(g)})
\end{align*}
for $\bvar{y}\sim F$.
Notice as $\bvar{yy}^T - \Delta$ is symmetric, the matrix $\E[(\bvar{y}\otimes \bvar{y}- \text{vec}(\Delta))(\bvar{y}\otimes \bvar{y} - \text{vec}(\Delta))^T]$ has rank $d(d+1)/2$ and as a result $\Sigma_{\otimes}$ is degenerate. 

As $f_{\bvar{Q}}(\bvar{B})$ is a function of the eigenvectors and values of $\bvar{B}$, and the eigenvectors and eigenvalues of $\bvar{B}$ are continuous functions of the entries of $\bvar{B}$ and is differentiable at $f_{\bvar{Q}}(\Delta_K)$.
With this observation and by the delta method stated in Theorem 3.1.5 of \textcite{Kollo}, we have 
\begin{align*}
    \sqrt{n}\text{vec}\left(\bvar{W}_nf_{\bvar{Q}}(n^{-1}\bvar{L}_K^T\bvar{L}_K) - f_{\bvar{Q}}(\Delta_K)\right) \overset{D}{\longrightarrow} N(\bvar{0},\Sigma_{\bvar{Q}})
\end{align*}
where $\Sigma_{\bvar{Q}}$ is given by
\begin{align*}
    \Sigma_{\bvar{Q}} = \left(\frac{d f_{\bvar{Q}}(\bvar{B})}{d\bvar{B}}\Big|_{\bvar{B} = \Delta_K}\right)^T\Sigma_{\otimes}\left(\frac{d f_{\bvar{Q}}(\bvar{B})}{d\bvar{B}}\Big|_{\bvar{B} = \Delta_K}\right).
\end{align*}
Therefore, the entries of $\sqrt{n}(\bvar{W}_n\bvar{Q} - \bvar{I})$ are asymptotically normally distributed with a degenerate covariance structure.
Therefore, letting $n\to\infty$, the elements of $[\bvar{I}_{d\times d}, \bvar{0}_{d\times q}]\left(\bvar{W}_n\bvar{Q} - \bvar{I}\right)[\bvar{S}^{(g)},  \bvar{S}^{(g)}_{\perp}]^T\bvar{x}_i$ are a linear combination of normal random variables with mean zero. 
Hence, 
\begin{align*}
\lim_{n\to\infty}\sqrt{n}\bvar{M}_h|\{\bvar{X}_i = \bvar{x}_i\} \overset{D}{\longrightarrow}N(0, \Sigma_g^{(M)}(\bvar{x}_i))    
\end{align*}
for covariance matrix $\Sigma_g^{(M)}(\bvar{x}_i)$ concluding the proof.
 \end{proof}

\begin{Lemma}\label{Power Method CLT}
Conditional on the event $\{\bvar{X}_i = \bvar{x}_i\}$ we have
\begin{align*}
    \sqrt{n}[(\tilde{\bvar{A}} - \tilde{\bvar{P}})\bvar{U}_{\tilde{\bvar{P}}}\bvar{S}_{\tilde{\bvar{P}}}^{-1/2}\tilde{\bvar{V}}\tilde{\bvar{W}}_n]_h |\{\bvar{X}_i = \bvar{x}_i\} &\overset{D}{\longrightarrow} N(\bvar{0}, \Sigma_g^{(N)}(\bvar{x}_i))
\end{align*} 
where the covariance matrix is given by
\begin{align*}
\Sigma_g^{(N)}(\bvar{x}_i) = \frac{1}{4}\Delta_S^{-1}\left[(\bvar{S}^{(g)} + m\bar{\bvar{S}})\tilde{\Sigma}_g(\bvar{x}_i)(\bvar{S}^{(g)} + m\bar{\bvar{S}}) + \sum_{k\neq g}\bvar{S}^{(k)}\tilde{\Sigma}_k(\bvar{x}_i)\bvar{S}^{(k)}\right]\Delta_S^{-1}
\end{align*}
and $\tilde{\Sigma}_g(\bvar{x}_i)$ is given by
\begin{align*}
    \tilde{\Sigma}_g(\bvar{x}_i) &= \E\left[(\bvar{x}_i^T\bvar{C}^{(g)}\bvar{X}_j - (\bvar{x}_i^T\bvar{C}^{(g)}\bvar{X}_j)^2)\bvar{X}_j\bvar{X}_j^T\right]
\end{align*}

\end{Lemma}

\begin{proof}
For notational convenience, for all that follows we will condition on the event $\{\bvar{X}_i = \bvar{x}_i\}$.
First notice that we can rewrite this term as follows
\begin{align*}
    &[(\tilde{\bvar{A}} - \tilde{\bvar{P}})\bvar{U}_{\tilde{\bvar{P}}}\bvar{S}_{\tilde{\bvar{P}}}^{-1/2}\tilde{\bvar{V}}\tilde{\bvar{W}}_n]_h \\
    &= (n\bvar{W}_n^*\bvar{S}_{\tilde{\bvar{P}}}^{-1}\bvar{W}_n^{*T})[\bvar{I}_{d\times d}, \bvar{0}_{d\times q}] (\bvar{W}_n\bvar{Q})\bvar{I}_{d,q}\left[\frac{1}{\sqrt{n}}(\tilde{\bvar{A}} - \tilde{\bvar{P}})(\bvar{I}\otimes \bvar{X})[\bvar{S}|\bvar{S}_{\perp}]\right]_h
\end{align*}
We will focus on each term in this product individually and then use Slutsky's Theorem to establish the result. 
First notice that we can write $\bvar{S}_{\tilde{\bvar{P}}} = \bvar{Z}^T\bvar{Z}$ where $\bvar{Z} = [\bvar{Z}|\bvar{Z}_{\perp}][\bvar{I}_{d\times d},\bvar{0}_{d\times q}]^T =\bvar{L}_K\bvar{Q}^{-1}[\bvar{I}_{d\times d},\bvar{0}_{d\times q}]^T$.
Therefore, we can write 
\begin{align*}
    \bvar{W}_n^{*}\bvar{S}_{\tilde{\bvar{P}}}\bvar{W}_n^{*T} &= \bvar{W}_n^{*}\begin{bmatrix}
    \bvar{I}_{d\times d}, 
    \bvar{0}_{d\times q}
    \end{bmatrix}
    \bvar{Q}\bvar{I}_{d,q}\bvar{L}_K^T\bvar{L}_K\bvar{I}_{d,q}\bvar{Q}\begin{bmatrix}
    \bvar{I}_{d\times d} \\ 
    \bvar{0}_{q\times d}
    \end{bmatrix}\bvar{W}_n^{*T}\\
    &= 
    \begin{bmatrix}
    \bvar{I}_{d\times d}, 
    \bvar{0}_{d\times q}
    \end{bmatrix}
    \bvar{W}_n\bvar{Q}\bvar{I}_{d,q}\bvar{L}_K^T\bvar{L}_K\bvar{I}_{d,q}\bvar{Q}^T\bvar{W}_n^T\begin{bmatrix}
    \bvar{I}_{d\times d}\\
    \bvar{0}_{q\times d}
    \end{bmatrix}
\end{align*}
where we use the fact $\bvar{Q}^{-1} = \bvar{I}_{d,q}\bvar{Q}^T\bvar{I}_{d,q}$.
In Lemma~\ref{Lemma:Q_characterization} we establish $\bvar{W}_n\bvar{Q}\overset{a.s.}{\longrightarrow}\bvar{I}$ and in Lemma~\ref{Eval Order} we establish $n^{-1}\bvar{L}_{K}^T\bvar{L}_K\overset{a.s.}{\longrightarrow}\Delta_K$.
Therefore, we see 
\begin{align*}
    n^{-1}\bvar{W}_n^{*}\bvar{S}_{\tilde{\bvar{P}}}\bvar{W}_n^T \overset{a.s.}{\longrightarrow} \begin{bmatrix}
    \bvar{I}_{d\times d}, 
    \bvar{0}_{d\times q}
    \end{bmatrix}
    \Delta_K\begin{bmatrix}
    \bvar{I}_{d\times d} \\ 
    \bvar{0}_{q\times d}
    \end{bmatrix} = \Delta_S.
\end{align*}
Therefore using the continuous mapping theorem establishes $n\bvar{W}_n^{*}\bvar{S}_{\tilde{\bvar{P}}}^{-1}\bvar{W}_n^{*T}\overset{a.s.}{\longrightarrow}\Delta_S^{-1}$.
Combining this limiting result with the other leading terms gives 
\begin{align*}
    (n\bvar{W}_n^*\bvar{S}_{\tilde{\bvar{P}}}^{-1}\bvar{W}_n^{*T})[\bvar{I}_{d\times d}, \bvar{0}_{d\times q}] (\bvar{W}_n\bvar{Q})\bvar{I}_{d,q} &\overset{a.s.}{\longrightarrow} \Delta_S^{-1}[\bvar{I}_{d\times d}, \bvar{0}_{d\times q}]
\end{align*}
Next, consider the decomposition of the remaining term 
\begin{align*}
    \frac{1}{\sqrt{n}}\left[(\tilde{\bvar{A}} - \tilde{\bvar{P}})(\bvar{I}\otimes \bvar{X})[\bvar{S}|\bvar{S}_{\perp}]\right]_h &=  \frac{1}{\sqrt{n}}\left[\sum_{k = 1}^m\left(\frac{\bvar{A}^{(g)} - \bvar{P}^{(g)}}{2} + \frac{\bvar{A}^{(k)} - \bvar{P}^{(k)}}{2}\right)\bvar{X}[\bvar{S}^{(k)}|\bvar{S}_{\perp}^{(k)}] \right]_i\\
    &= \frac{1}{\sqrt{n}}\sum_{k = 1}^m[\bvar{S}^{(k)}|\bvar{S}_{\perp}^{(k)}]^T\sum_{j=1}^n\left(\frac{\bvar{A}_{ij}^{(g)} - \bvar{P}_{ij}^{(g)}}{2} + \frac{\bvar{A}_{ij}^{(k)} - \bvar{P}_{ij}^{(k)}}{2}\right)\bvar{X}_j\\
    &= \sum_{k = 1}^m[\bvar{S}^{(k)}|\bvar{S}_{\perp}^{(k)}]^T\frac{1}{\sqrt{n}}\left\{\sum_{j\neq i}\left(\frac{\bvar{A}_{ij}^{(g)} - \bvar{x}_i^T\bvar{C}^{(g)}\bvar{X}_j}{2} + \frac{\bvar{A}_{ij}^{(k)} - \bvar{x}_i^T\bvar{C}^{(k)}\bvar{X}_j}{2}\right)\bvar{X}_j\right\}\\
    &-\sum_{k = 1}^m[\bvar{S}^{(k)}|\bvar{S}_{\perp}^{(k)}]^T\frac{1}{\sqrt{n}}\left\{\frac{\bvar{x}_i^T\bvar{C}^{(g)}\bvar{x}_i}{2} + \frac{\bvar{x}_i^T\bvar{C}^{(k)}\bvar{x}_i}{2}\right\}
\end{align*}
Notice as $\bvar{x}_i^T\bvar{C}^{(g)}\bvar{x}_i$ and $\bvar{x}_i^T\bvar{C}^{(k)}\bvar{x}_i$ are bounded by 1. 
Moreover, as $[\bvar{S}^{(k)}|\bvar{S}_{\perp}^{(k)}]$ is independent of $n$ for all $k\in[m]$
\begin{align*}
    -\sum_{k = 1}^m[\bvar{S}^{(k)}|\bvar{S}_{\perp}^{(k)}]^T\frac{1}{\sqrt{n}}\left\{\frac{\bvar{x}_i^T\bvar{C}^{(g)}\bvar{x}_i}{2} + \frac{\bvar{x}_i^T\bvar{C}^{(k)}\bvar{x}_i}{2}\right\} \overset{\prob}{\longrightarrow} \bvar{0} 
\end{align*}
Focusing on this second term, notice that we can write the sum of $n-1$ random variables as 
\begin{align*}
    \frac{1}{\sqrt{n}}\sum_{j\neq i}\left\{\left([\bvar{S}^{(g)}|\bvar{S}_{\perp}^{(g)}]^T + \sum_{k \neq g }\frac{1}{2}[\bvar{S}^{(k)}|\bvar{S}_{\perp}^{(k)}]^T\right)(\bvar{A}_{ij}^{(g)} - \bvar{x}_i^T\bvar{C}^{(g)}\bvar{X}_j)\bvar{X}_j + \sum_{k\neq g}\frac{1}{2}\bvar{S}^{(k)}(\bvar{A}_{ij}^{(k)} - \bvar{x}_i^T\bvar{C}^{(k)}\bvar{X}_j)\bvar{X}_j\right\} 
\end{align*}
From which the the multivariate central limit theorem gives
\begin{align*}
    \frac{1}{\sqrt{n}}\sum_{j\neq i}\Big\{\Big([\bvar{S}^{(g)}|\bvar{S}_{\perp}^{(g)}]^T &+ \sum_{k \neq g} \frac{1}{2}[\bvar{S}^{(k)}|\bvar{S}_{\perp}^{(k)}]^T\Big)(\bvar{A}_{ij}^{(g)} - \bvar{x}_i^T\bvar{C}^{(g)}\bvar{X}_j)\bvar{X}_j\\ 
    &+ \sum_{k\neq g}\frac{1}{2}\bvar{S}^{(k)}(\bvar{A}_{ij}^{(k)} - \bvar{x}_i^T\bvar{C}^{(k)}\bvar{X}_j)\bvar{X}_j\Big\} \overset{D}{\longrightarrow} N(0,\bar{\Sigma}_g(\bvar{x}_i))
\end{align*}
where the covariance matrix is given by
\begin{align*}
    \bar{\Sigma}_g(\bvar{x}_i) &= \Big([\bvar{S}^{(g)}|\bvar{S}_{\perp}^{(g)}]^T + \sum_{k \neq g} \frac{1}{2}[\bvar{S}^{(k)}|\bvar{S}_{\perp}^{(k)}]^T\Big)\tilde{\Sigma}_g(\bvar{x}_i)\Big([\bvar{S}^{(g)}|\bvar{S}_{\perp}^{(g)}]^T + \sum_{k \neq g} \frac{1}{2}[\bvar{S}^{(k)}|\bvar{S}_{\perp}^{(k)}]\Big)\\ 
    &+ \frac{1}{4}\sum_{k\neq g}[\bvar{S}^{(k)}|\bvar{S}_{\perp}^{(k)}]^T\tilde{\Sigma}_k(\bvar{x}_i)[\bvar{S}^{(k)}|\bvar{S}_{\perp}^{(k)}]
\end{align*}
and $\tilde{\Sigma}_k(\bvar{x}_i)$ is given by $\tilde{\Sigma}_k(\bvar{x}_i) = \E\left[(\bvar{x}_i^T\bvar{C}^{(k)}\bvar{X}_j - (\bvar{x}_i^T\bvar{C}^{(k)}\bvar{X}_j)^2)\bvar{X}_j\bvar{X}_j^T\right]$.
Notice by pre and most multiplying this covariance by $\Delta_S^{-1}[\bvar{I}_{d,\times d}, \bvar{0}_{d,\times q}]$ gives 
\begin{align*}
    \Delta_S^{-1}[\bvar{I}_{d,\times d}, \bvar{0}_{d,\times q}]\bar{\Sigma}_g(\bvar{x}_i)\Delta_S[\bvar{I}_{d,\times d}, \bvar{0}_{d,\times q}]^T\Delta_S^{-1} &= \Delta_S^{-1}\Big(\bvar{S}^{(g)} + \sum_{k \neq g} \frac{1}{2}\bvar{S}^{(k)}\Big)\tilde{\Sigma}_g(\bvar{x}_i)\Big(\bvar{S}^{(g)} + \sum_{k \neq g} \frac{1}{2}\bvar{S}^{(k)}\Big)\Delta_S^{-1}\\ 
    &+ \frac{1}{4}\Delta_S^{-1}\sum_{k\neq g}\bvar{S}^{(k)}\tilde{\Sigma}_k(\bvar{x}_i)\bvar{S}^{(k)}\Delta_S^{-1} 
\end{align*}

Therefore, applying the multivariate Slutsky's theorem then provides
\begin{align*}
    \sqrt{n}[(\tilde{\bvar{A}} - \tilde{\bvar{P}})\bvar{L}_S\tilde{\bvar{W}}_n\tilde{\bvar{S}}_{\tilde{\bvar{P}}}^{-1}\tilde{\bvar{W}}_n^T]_h \overset{D}{\longrightarrow} N\left(\bvar{0}, \Sigma_g^{(N)}(\bvar{x}_i)\right)
\end{align*}
where 
\begin{align*}
\Sigma_g^{(N)}(\bvar{x}_i) = \frac{1}{4}\Delta_S^{-1}\left[\Big(\bvar{S}^{(g)} + m\bar{\bvar{S}})\tilde{\Sigma}_g(\bvar{x}_i)(\bvar{S}^{(g)} + \bar{\bvar{S}}) +\sum_{k\neq g}\bvar{S}^{(k)}\tilde{\Sigma}_k(\bvar{x}_i)\bvar{S}^{(k)}\right]\Delta_S^{-1}
\end{align*}
which concludes the proof. 
\end{proof}

Having demonstrated the asymptotic normality of this term, we now turn to showing the remaining residual terms converge to zero in probability. 
As these Lemmas follow directly from \textcite{OmniCLT} and the bounds stated in Appendix \ref{First Moment Appendix}, we state them without proof and refer the reader to \textcite{OmniCLT}. 

\begin{Lemma}\label{root n: Easy Residual}
Let $h = n(g-1) + i$. Then with the notation as given in the proof of Theorem \ref{CLT} we have the following convergence results.  
\begin{align}
    &\sqrt{n}[(\tilde{\bvar{A}} - \tilde{\bvar{P}})\bvar{U}_{\tilde{\bvar{P}}}(\tilde{\bvar{V}}\bvar{S}_{\tilde{\bvar{A}}}^{-1/2} - \bvar{S}_{\tilde{\bvar{P}}}^{-1/2}\tilde{\bvar{V}})]_h \overset{\prob}{\longrightarrow}\bvar{0}\label{root n: Commuter}\\
    &\sqrt{n}[\bvar{U}_{\tilde{\bvar{P}}}\bvar{U}_{\tilde{\bvar{P}}}^T(\tilde{\bvar{A}} - \tilde{\bvar{P}})\bvar{U}_{\tilde{\bvar{P}}}\tilde{\bvar{V}}\bvar{S}_{\tilde{\bvar{A}}}^{-1/2}]_h \overset{\prob}{\longrightarrow}\bvar{0}\label{root n: Hoeffding}\\
    &\sqrt{n}[\bvar{R}_1\bvar{S}_{\tilde{\bvar{A}}}^{1/2} + \bvar{U}_{\tilde{\bvar{P}}}\bvar{R}_2]_h  \overset{\prob}{\longrightarrow}\bvar{0}\label{root n: Rs}\\
    &\sqrt{n}[(\bvar{U}_{\tilde{\bvar{P}}}^{\perp}\bvar{S}_{\tilde{\bvar{P}}}^{\perp}\bvar{U}_{\tilde{\bvar{P}}}^{\perp})^T\bvar{U}_{\tilde{\bvar{A}}}\bvar{S}_{\tilde{\bvar{A}}}^{-1/2}]_h{\longrightarrow} \bvar{0}\label{root n: 2d eigenvectors}
\end{align}
\end{Lemma}


\begin{Lemma}\label{root n: Exchangeable}
With the notation as used in the proof of Theorem \ref{CLT} we have the following convergence
\begin{align*}
    \sqrt{n}[(\bvar{I} - \bvar{U}_{\tilde{\bvar{P}}}\bvar{U}_{\tilde{\bvar{P}}}^T)(\tilde{\bvar{A}} - \tilde{\bvar{P}})\bvar{R}_3\bvar{S}_{\tilde{\bvar{A}}}^{-1/2}]_h \overset{\prob}{\longrightarrow}\bvar{0}
\end{align*}
\end{Lemma}


\section{Corollaries and Statistical Consequences}\label{Corollaries and Statistical Consequences Appendix}

Included below are proofs of the Corollaries utilizing the asymptotic joint distribution of omnibus node embeddings. 
These proofs largely follow direction from Theorem \ref{Main Bias Theorem} and \ref{CLT}. 
We derive the asymptotic covariances of each set of rows explicitly as it serves as a format for the development of further estimators that utilize the rows of $\hat{\bvar{L}}$.

\begin{proof}[Proof of Corollaries 3.4, 3.5]
Define the vector $\bvar{R}_{r_k} = (\hat{\bvar{L}}\tilde{\bvar{W}}_n - \bvar{L}_S)_{r_k}\in\R^d$ for a finite collection of rows indexed by $\{r_k\}_{k=1}^K\subset[nm]$. 
Consider the vector 
\begin{align*}
    \bvar{V} = [\bvar{R}_{r_1}^T|\bvar{R}_{r_2}^T|\dots |\bvar{R}_{r_K}^T]^T \in \R^{Kd\times 1}
\end{align*}
Utilizing the decomposition in the proof of Theorem \ref{CLT} each $\bvar{R}_{r_k}$ can be written as 
\begin{align*}
 \sqrt{n}\bvar{R}_{r_k} =   \sqrt{n}\bvar{H}_{r_k} + \sqrt{n}\bvar{M}_{r_k} + \sqrt{n}\bvar{R}_{r_k}^{(2)}\tilde{\bvar{W}}_n
\end{align*} 
so we can write 
\begin{align*}
    \sqrt{n}\bvar{V} = \sqrt{n}\begin{bmatrix}
    \bvar{H}_{r_1}\\
    \vdots\\
    \bvar{H}_{r_K}
    \end{bmatrix} + 
    \sqrt{n}\begin{bmatrix}
    \bvar{M}_{r_1}\\
    \vdots\\
    \bvar{M}_{r_K}
    \end{bmatrix} + 
    \sqrt{n}\begin{bmatrix}
    \bvar{R}^{(2)}_{r_1}\\
    \vdots\\
    \bvar{R}^{(2)}_{r_K}
    \end{bmatrix}\tilde{\bvar{W}}_n.
\end{align*}
This final term converges to zero as  $\sqrt{n}\bvar{R}_{r_k}^{(2)} \overset{\prob}{\longrightarrow}\bvar{0}$ for each $k\in[K]$.
Conditional on $\{\bvar{X}_{r_k} = \bvar{x}_{r_k}\}_{k=1}^K$, Lemma~\ref{Lemma:mu-distribution} and Lemma~\ref{Power Method CLT} establish that each of the vectors above converge to a multivariate normal distribution.
That is, conditional on $\{\bvar{X}_{r_k} = \bvar{x}_{r_k}\}_{k=1}^K$ 
\begin{align*}
    \sqrt{n}\begin{bmatrix}
    \bvar{H}_{r_1}\\
    \vdots\\
    \bvar{H}_{r_K}
    \end{bmatrix} + 
    \sqrt{n}\begin{bmatrix}
    \bvar{M}_{r_1}\\
    \vdots\\
    \bvar{M}_{r_K}
    \end{bmatrix} + 
    \sqrt{n}\begin{bmatrix}
    \bvar{R}^{(2)}_{r_1}\\
    \vdots\\
    \bvar{R}^{(2)}_{r_K}
    \end{bmatrix}\tilde{\bvar{W}}_n \overset{D}{\longrightarrow} N(\bvar{0}, \Sigma(\bvar{x}_{r_1}, \ldots, \bvar{x}_{r_K}))
\end{align*}
where the covariance $\Sigma(\bvar{x}_{r_1}, \ldots, \bvar{x}_{r_K})\in\R^{Kd\times Kd}$ can be decomposed as
\begin{align*}
\Sigma(\bvar{x}_{r_1}, \ldots, \bvar{x}_{r_K}) = \Sigma^{(N)}(\bvar{x}_{r_1}, \ldots, \bvar{x}_{r_K}) + \Sigma^{(M)}(\bvar{x}_{r_1}, \ldots, \bvar{x}_{r_K}) + \Sigma^{(N,M)}(\bvar{x}_{r_1}, \ldots, \bvar{x}_{r_K})
\end{align*}
Letting $r_k = n(g_k -1) + i_k$ for some $g_k\in[m]$ and $i_k\in[n]$, each $\Sigma^{(N)}, \Sigma^{(M)}, \Sigma^{(N,M)}$ is block diagonal with blocks $\Sigma_{g_k}^{(N)}(\bvar{x}_{r_k})$, $\Sigma_{g_k}^{(M)}(\bvar{x}_{r_k})$, and $\Sigma_{g_k}^{(N,M)}(\bvar{x}_{r_k})$, respectively.
Therefore, it suffices to specify the off diagonal blocks of $\Sigma^{(H)}, \Sigma^{(H)}, \Sigma^{(H,M)}$ whence integrating over possible values of $\bvar{x}_{r_k}$ will yield the result. 

We begin with specifying the off diagonal blocks of $\Sigma^{(N)}(\bvar{x}_{r_1}, \ldots, \bvar{x}_{r_K})$.
To simplify notation, we consider $K=2$ rows and note that this argument can be extended to the $K>2$ setting directly.
First recall we can $\sqrt{n}\bvar{H}_{r_k}$ as
\begin{align*}
(n\bvar{W}_n^{*}\bvar{S}_{\tilde{\bvar{P}}}^{-1}\bvar{W}_n^{*T})[\bvar{I}_{d\times d}, \bvar{0}_{d\times q}]\bvar{W}_n\bvar{Q}\bvar{I}_{d,q}
\left[\frac{1}{\sqrt{n}}(\tilde{\bvar{A}} - \tilde{\bvar{P}})(\bvar{I}\otimes \bvar{X})[\bvar{S}|\bvar{S}_{\perp}]\right]_{r_k}
\end{align*}
This first term converges to $\Delta_S^{-1}[\bvar{I}_{d\times d}, \bvar{0}_{d\times q}]$
for each $r_k$ we can use the multivariate Slutsky theorem as well as the distribution of the remaining term to establish the result. 

Recall we denote $r_1 = n(g_1 - 1) + i_1$ and $r_2 = n(g_2 - 1) + i_2$
There are four possible scenarios we need to consider depending on the pairs $(g_1, g_2)$ and $(i_1, i_2)$.
If $g_1 = g_2$ and $i_1 = i_2$ then $\hat{\bvar{L}}_{r_1} = \hat{\bvar{L}}_{r_2}$ and the off block diagonals are simply given by $\Sigma_{g_1}^{(N)}(\bvar{x}_{r_1})$.
Next, suppose $g_1 \neq g_2$ and $i_1 = i_2$.
Then, letting $\bvar{K}^{(k)} = [\bvar{S}^{(k)}|\bvar{S}_{\perp}^{(k)}]$ then what remains is the term
\begin{align*}
     \frac{1}{\sqrt{n}}\sum_{j = 1}^n\frac{1}{2}\left\{\left(\bvar{K}^{(g)} + m\bar{\bvar{K}}\right)^T(\bvar{A}_{i_1j}^{(g_1)} - \bvar{x}_{i_1}^T\bvar{C}^{(g_1)}\bvar{X}_j)\bvar{X}_j + \sum_{\ell\neq g}\bvar{K}^{(\ell)T}(\bvar{A}_{i_1j}^{(\ell)} - \bvar{x}_{i_1}^T\bvar{C}^{(\ell)}\bvar{X}_j)\bvar{X}_j\right\}
\end{align*}
The covariance for this term is given by 
\begin{align*}
    &\E\Big[\text{Cov}\Big(\frac{1}{2}\left\{\left(\bvar{K}^{(g_1)} + m\bar{\bvar{K}}\right)^T(\bvar{A}_{i_1j}^{(g_1)} - \bvar{x}_{i_1}^T\bvar{C}^{(g_1)}\bvar{X}_j)\bvar{X}_j + \sum_{\ell\neq g_1}\bvar{K}^{(\ell)T}(\bvar{A}_{i_1j}^{(\ell)} - \bvar{x}_{i_1}^T\bvar{C}^{(\ell)}\bvar{X}_j)\bvar{X}_j\right\}, \\
    &\frac{1}{2}\left\{\left(\bvar{K}^{(g_2)} + m\bar{\bvar{K}}\right)^T(\bvar{A}_{i_1j}^{(g_2)} - \bvar{x}_{i_1}^T\bvar{C}^{(g_2)}\bvar{X}_j)\bvar{X}_j + \sum_{\ell\neq g_2}\bvar{K}^{(\ell)T}(\bvar{A}_{i_1j}^{(\ell)} - \bvar{x}_{i_1}^T\bvar{C}^{(\ell)}\bvar{X}_j)\bvar{X}_j\right\}\Big|\bvar{X}_j\Big)\Big]\\
    &=\frac{1}{4}\E\left[(\bvar{K}^{(g_1)} + m\bar{\bvar{K}})\bvar{X}_j\text{Cov}(\bvar{A}_{i_1j}^{(g_1)}, \bvar{A}_{i_1j}^{(g_2)}|\bvar{X}_j)\bvar{X}_j^T(\bvar{K}^{(g_2)} + m\bar{\bvar{K}})\right]\\
    &+\frac{1}{4}\E\left[(\bvar{K}^{(g)} + m\bar{\bvar{K}})^T\bvar{X}_j\sum_{\ell \neq g_2}\text{Cov}(\bvar{A}_{i_1j}^{(g_1)},\bvar{A}_{i_1j}^{(\ell)}|\bvar{X}_j)\bvar{X}_j^T\bvar{K}^{(\ell)}\right]\\
    &+\frac{1}{4}\E\left[\sum_{\ell\neq g_1}\bvar{K}^{(\ell)T}\bvar{X}_j\text{Cov}(\bvar{A}_{i_1j}^{(\ell)}, \bvar{A}_{i_1j}^{(g_2)}|\bvar{X}_j)\bvar{X}_j^T(m\bar{\bvar{K}} + \bvar{K}^{(g_2)})\right]\\
    &+\frac{1}{4}\sum_{\ell\neq g_1}\sum_{\ell'\neq g_2}\E\left[\bvar{K}^{(\ell)T}\bvar{X}_j\text{Cov}(\bvar{A}_{i_1j}^{(\ell)}, \bvar{A}_{i_1j}^{(\ell')}|\bvar{X}_j)\bvar{X}_j^T\bvar{K}^{(\ell')}\right]\\
    &= \frac{1}{4}(\bvar{K}^{(g_1)} + m\bar{\bvar{K}})^T\tilde{\Sigma}_{g_1}(\bvar{x}_{i_1})\bvar{K}^{(g_1)} + \bvar{K}^{(g_2T)}\tilde{\Sigma}_{g_2}(\bvar{x}_{i_1})(m\bar{\bvar{K}} + \bvar{K}^{(g_2)})+ \frac{1}{4}\sum_{\ell \neq g_1,g_2}\bvar{K}^{(\ell)T}\tilde{\Sigma}_{\ell}(\bvar{x}_{i_1})\bvar{K}^{(\ell)} 
\end{align*}
Therefore, we see the rows of $\hat{\bvar{L}}$ corresponding to the same vertex $i_1 = i_2$ are correlated with the above covariance pre and post multiplied by $\Delta_S^{-1}[\bvar{I}_{d\times d}, \bvar{0}_{d\times q}]$.
Therefore, the covariance between rows of $\bvar{H}$ corresponding to the same vertex can be written
\begin{align*}
    \Sigma^{(R)}_{g_1g_2}(\bvar{x}_{i_1}) =  \frac{1}{4}\Delta_S^{-1}\Big((\bvar{S}^{(g_1)} &+ m\bar{\bvar{S}})\tilde{\Sigma}_{g_1}(\bvar{x}_{i_1})\bvar{S}^{(g_1)} + \bvar{S}^{(g_2)}\tilde{\Sigma}_{g_2}(\bvar{x}_{i_1})(m\bar{\bvar{S}} + \bvar{S}^{(g_2)})\\
    &+ \sum_{\ell \neq g_1,g_2}\bvar{S}^{(\ell)}\tilde{\Sigma}_{\ell}(\bvar{x}_{i_1})\bvar{S}^{(\ell)} \Big)\Delta_S^{-1}
\end{align*}
Next, consider the setting where $i_1\neq i_2$. 
Then we look to calculate the covariance 
\begin{align*}
    &\E\Big[\text{Cov}\Big(\frac{1}{2}\left\{\left(\bvar{K}^{(g_1)} + m\bar{\bvar{K}}\right)^T(\bvar{A}_{i_1j}^{(g_1)} - \bvar{x}_{i_1}^T\bvar{C}^{(g_1)}\bvar{X}_j)\bvar{X}_j + \sum_{\ell\neq g_1}\bvar{K}^{(\ell)^T}(\bvar{A}_{i_1j}^{(\ell)} - \bvar{x}_{i_1}^T\bvar{C}^{(\ell)}\bvar{X}_j)\bvar{X}_j\right\}, \\
    &\frac{1}{2}\left\{\left(\bvar{K}^{(g_2)} + m\bar{\bvar{K}}\right)^T(\bvar{A}_{i_2j}^{(g)} - \bvar{x}_{i_2}^T\bvar{C}^{(g_2)}\bvar{X}_j)\bvar{X}_j + \sum_{\ell\neq g_2}\bvar{K}^{(\ell)T}(\bvar{A}_{i_2j}^{(\ell)} - \bvar{x}_{i_2}^T\bvar{C}^{(\ell)}\bvar{X}_j)\bvar{X}_j\right\}\Big|\bvar{X}_j\Big)\Big]\\
    &= \frac{1}{4}\E\left[(\bvar{K}^{(g_1)} + m\bar{\bvar{K}})^T\text{Cov}(\bvar{A}_{i_1j}^{(g_1)}, \bvar{A}_{i_2j}^{(g_2)}|\bvar{X}_j)\bvar{X}_j^T(\bvar{K}^{(g_2)} + m\bar{\bvar{K}})\right]\\
    &+\frac{1}{4}\E\left[(\bvar{K}^{(g_1)} + m\bar{\bvar{K}})^T\bvar{X}_j\sum_{\ell \neq g_1}\text{Cov}(\bvar{A}_{i_1j}^{(g_1)},\bvar{A}_{i_2j}^{(\ell)}|\bvar{X}_j)\bvar{X}_j^T\bvar{K}^{(\ell)}\right]\\
    &+\frac{1}{4}\E\left[\sum_{\ell\neq g_2}\bvar{K}^{(\ell)T}\bvar{X}_j\text{Cov}(\bvar{A}_{i_1j}^{(\ell)}, \bvar{A}_{i_2j}^{(g_2)}|\bvar{X}_j)\bvar{X}_j^T(m\bar{\bvar{K}} + \bvar{K}^{(g_2)})\right]\\
    &+\frac{1}{4}\sum_{\ell\neq g_1}\sum_{\ell'\neq g_2}\E\left[\bvar{K}^{(\ell)T}\bvar{X}_j\text{Cov}(\bvar{A}_{i_1j}^{(\ell)}, \bvar{A}_{i_2j}^{(\ell')}|\bvar{X}_j)\bvar{X}_j^T\bvar{K}^{(\ell')}\right]\\
    &=0
\end{align*}
Therefore, we see that rows of $\bvar{H}$ that correspond to different vertices are asymptotically independent. 
Using the expressions given above, we can now define $\Omega_{gk}^{(N)}(\bvar{y})$ and $\Psi_{gk}^{(N)}(\bvar{y}_1, \bvar{y}_2)$
\begin{align*}
    \Omega_{gk}^{(N)}(\bvar{y}) = \begin{bmatrix}
    \Sigma_g^{(N)}(\bvar{y}) & \Sigma_{gk}^{(N)}(\bvar{y})\\
    \Sigma_{kg}^{(N)}(\bvar{y}) & \Sigma_k^{(N)}(\bvar{y})\\
    \end{bmatrix} \quad \quad 
    \Psi_{gk}^{(N)}(\bvar{y}_1, \bvar{y}_2) = \begin{bmatrix}
    \Sigma_g^{(N)}(\bvar{y}_1) & \bvar{0}\\
    \bvar{0} & \Sigma_k^{(N)}(\bvar{y}_2)\\
    \end{bmatrix}
\end{align*}

Moving to the $\bvar{M}$ term, consider the $r_k$ row 
\begin{align*}
    \sqrt{n}\bvar{M}_{r_k} = 
    [\bvar{I}_{d\times d}, \bvar{0}_{d\times q}][\sqrt{n}\left(\bvar{W}_n\bvar{Q} - \bvar{I}\right)]\bvar{I}_{d,q}[(\bvar{I}\otimes \bvar{X})[\bvar{S}|\bvar{S}_{\perp}]]_{r_k}
\end{align*}
Let $\bvar{M}_{\bvar{Q}} = [\bvar{I}_{d\times d}, \bvar{0}_{d\times q}][\sqrt{n}\left(\bvar{W}_n\bvar{Q} - \bvar{I}\right)]\bvar{I}_{d,q}$.
Then Lemma~\ref{Lemma:mu-distribution} establishes $\bvar{M}_{\bvar{Q}}$ has normally distributed entries with a degenerate covariance structure. 
As $\E[\bvar{M}_\bvar{Q}] = 0$, for any $g_1, g_2\in [m]$ and $i_1, i_2\in[n]$, we can write the covariance as
\begin{align*}
\Sigma^{(M)}_{g_1g_2}(\bvar{x}_{i_1}, \bvar{x}_{i_2}) &= \E\left[\bvar{M}_{\bvar{Q}}[\bvar{S}^{(g_1)}|\bvar{S}_{\perp}^{(g_1)}]^T\bvar{x}_{i_1}\bvar{x}_{i_2}^T[\bvar{S}^{(g_2)}|\bvar{S}_{\perp}^{(g_2)}]\bvar{M}_{\bvar{Q}}^T\right]\\
\end{align*}
which is non-zero for any combination of $g_1, g_2, i_1, i_2$.
Using the expressions given above, we can now define $\Omega_{gk}^{(M)}(\bvar{y})$ and $\Psi_{gk}^{(M)}(\bvar{y}_1, \bvar{y}_2)$
\begin{align*}
    \Omega_{gk}^{(M)}(\bvar{y}) = \begin{bmatrix}
    \Sigma_{gg}^{(M)}(\bvar{y}, \bvar{y}) & \Sigma_{gk}^{(M)}(\bvar{y}, \bvar{y})\\
    \Sigma_{kg}^{(M)}(\bvar{y}, \bvar{y}) & \Sigma_{kk}^{(M)}(\bvar{y}, \bvar{y})\\
    \end{bmatrix} \quad \quad 
    \Psi_{gk}^{(M)}(\bvar{y}_1, \bvar{y}_2) = \begin{bmatrix}
    \Sigma_{gg}^{(M)}(\bvar{y}_1, \bvar{y}_2) & \Sigma_{gk}^{(M)}(\bvar{y}_1, \bvar{y}_2)\\
    \Sigma_{kg}^{(M)}(\bvar{y}_2, \bvar{y}_1) & \Sigma_{kk}^{(M)}(\bvar{y}_2, \bvar{y}_2)\\
    \end{bmatrix}.
\end{align*}
Finally, let $\Sigma_{g_1g_2}^{(R,M)}(\bvar{x}_{i_1}, \bvar{x}_{i_2}) = \E[\bvar{H}_{r_1}\bvar{M}_{r_2}^T +  \bvar{M}_{r_2}\bvar{H}_{r_1}^T]$ and define 
\begin{align*}
    \Omega_{gk}^{(N,M)}(\bvar{y}) = \begin{bmatrix}
    \Sigma_{gg}^{(N,M)}(\bvar{y}, \bvar{y}) & \Sigma_{gk}^{(N,M)}(\bvar{y}, \bvar{y})\\
    \Sigma_{kg}^{(N,M)}(\bvar{y}, \bvar{y}) & \Sigma_{kk}^{(N,M)}(\bvar{y}, \bvar{y})\\
    \end{bmatrix} \quad \quad 
    \Psi_{gk}^{(N,M)}(\bvar{y}_1, \bvar{y}_2) = \begin{bmatrix}
    \Sigma_{gg}^{(N,M)}(\bvar{y}_1, \bvar{y}_2) & \Sigma_{gk}^{(N,M)}(\bvar{y}_1, \bvar{y}_2)\\
    \Sigma_{kg}^{(N,M)}(\bvar{y}_2, \bvar{y}_1) & \Sigma_{kk}^{(N,M)}(\bvar{y}_2, \bvar{y}_2)\\
    \end{bmatrix}.
\end{align*}
Having specified the off diagonal covariance structure of $\bvar{V}$, for any combination $(i_1, i_2, g_1, g_2)$ this concludes the proof. 
\end{proof}

\begin{proof}[Proof of Corollary 4.1]
Recall from the definition of $\bvar{S}^{(g)}$ given in Theorem~\ref{Main Bias Theorem}, we can write $\bar{\bvar{S}}$ as 
\begin{align*}
    \bar{\bvar{S}} = \frac{1}{m}\sum_{g=1}^m\left(\frac{\bvar{C}^{(g)}\bvar{C}_m^{-1/4} + \bvar{C}_m^{1/4}}{2}\right) = \frac{\bar{\bvar{C}}\bvar{C}_{m}^{-1/4} + \bvar{C}_{m}^{1/4}}{2} = \sqrt{\bvar{\bar{C}}}\left(\frac{\bar{\bvar{C}}^{1/2}\bvar{C}_{m}^{-1/4} + \bar{\bvar{C}}^{-1/2}\bvar{C}_{m}^{1/4}}{2}\right).
\end{align*}
As $\bar{\bvar{C}}$ and $\bvar{C}_m$ are diagonal, they commute and we may write
\begin{align*}
\|\bar{\bvar{S}}(\bvar{x}_k - \bvar{x}_{\ell})\| \leq \|2^{-1}(\bar{\bvar{C}}^{1/2}\bvar{C}_{m}^{-1/4} + \bar{\bvar{C}}^{-1/2}\bvar{C}_{m}^{1/4})\| \|\sqrt{\bar{\bvar{C}}}(\bvar{x}_k - \bvar{x}_{\ell})\|.
\end{align*}
$\bar{\bvar{C}}$ is full rank so $\|\sqrt{\bar{\bvar{C}}}(\bvar{x}_k - \bvar{x}_{\ell})\| \neq 0$ and we can write 
\begin{align*}
    \frac{\|\bar{\bvar{S}}(\bvar{x}_k - \bvar{x}_{\ell})\|}{\|\sqrt{\bar{\bvar{C}}}(\bvar{x}_k - \bvar{x}_{\ell})\|} \leq \|2^{-1}(\bar{\bvar{C}}^{1/2}\bvar{C}_{m}^{-1/4} + \bar{\bvar{C}}^{-1/2}\bvar{C}_{m}^{1/4})\|.
\end{align*}
Let $\bvar{v}_i = (\bvar{C}_{ii}^{(1)}, \bvar{C}_{ii}^{(2)}, \ldots, \bvar{C}_{ii}^{(m)})^T$ and notice the $(i,i)$-th element of $\bar{\bvar{C}}^{1/2}\bvar{C}_m^{-1/4}$ can be written as 
\begin{align*}
    (\bar{\bvar{C}}^{-1/2}\bvar{C}_m^{-1/4})_{ii} = \left(\frac{1}{m}\sum_{g=1}^m\bvar{C}^{(g)}_{ii}\right)^{1/2}\left(\frac{1}{m}\sum_{g=1}^m\bvar{C}^{(g)2}_{ii}\right)^{-1/4}
    = \left(\frac{\|\bvar{v}_{i}\|_1}{m}\right)^{1/2}\left(\frac{\|\bvar{v}_{i}\|_2^2}{m}\right)^{-1/4}
    = \sqrt{\frac{\|\bvar{v}_i\|_1}{\sqrt{m}\|\bvar{v}_i\|_2}}.
\end{align*}
With this and the norm ordering $\|\bvar{v}_i\|_2 \leq \|\bvar{v}_i\|_1\leq \sqrt{m}\|\bvar{v}_i\|_2$ we have
\begin{align*}
    \|2^{-1}(\bar{\bvar{C}}^{1/2}\bvar{C}_{m}^{-1/4} + \bar{\bvar{C}}^{-1/2}\bvar{C}_{m}^{-1/4})\| &= \max_{i\in[d]}\frac{1}{2}\left(\sqrt{\frac{\|\bvar{v}_i\|_1}{\sqrt{m}\|\bvar{v}_i\|_2}} + \sqrt{\frac{\sqrt{m}\|\bvar{v}_i\|_2}{\|\bvar{v}_i\|_1}}\right)
\end{align*}
For the lower bound, first notice 
\begin{align*}
    \|\sqrt{\bar{\bvar{C}}}(\bvar{x}_k - \bvar{x}_{\ell})\| \leq \|(2^{-1}(\bar{\bvar{C}}^{1/2}\bvar{C}_{m}^{-1/4} + \bar{\bvar{C}}^{-1/2}\bvar{C}_{m}^{1/4}))^{-1}\| \|\bar{\bvar{S}}(\bvar{x}_k - \bvar{x}_{\ell})\|
\end{align*}
and by rearranging gives gives 
\begin{align*}
    \|(2^{-1}(\bar{\bvar{C}}^{1/2}\bvar{C}_{m}^{-1/4} + \bar{\bvar{C}}^{-1/2}\bvar{C}_{m}^{1/4}))^{-1}\|^{-1} \leq \frac{\|\bar{\bvar{S}}(\bvar{x}_k - \bvar{x}_{\ell})\|}{\|\sqrt{\bar{\bvar{C}}}(\bvar{x}_k - \bvar{x}_{\ell})\|} 
\end{align*}
Recall for an invertible matrix $\|\bvar{M}^{-1}\| = 1/\lambda_{\min}(\bvar{M})$ and hence $\|\bvar{M}^{-1}\|^{-1} = \lambda_{\min}(\bvar{M})$.
Therefore, 
\begin{align*}
    \|(2^{-1}(\bar{\bvar{C}}^{1/2}\bvar{C}_{m}^{-1/4} + \bar{\bvar{C}}^{-1/2}\bvar{C}_{m}^{1/4}))^{-1}\|^{-1} = \min_{i\in[d]}\frac{1}{2}\left(\sqrt{\frac{\|\bvar{v}_i\|_1}{\sqrt{m}\|\bvar{v}_i\|_2}} + \sqrt{\frac{\sqrt{m}\|\bvar{v}_i\|_2}{\|\bvar{v}_i\|_1}}\right). 
\end{align*}
Let $f(x) = 2^{-1}(x^{1/2} + x^{-1/2})$ and notice $f(x)\geq 1$ for $x\geq 0$.
With this notation, we can rewrite the bounds.
Notice as $\|\bvar{v}_i\|_1 / \sqrt{m}\|\bvar{v}_i\|_2 \geq 0$, then we can write the lower bound 
\begin{align*}
    \min_{i\in[d]}f(\|\bvar{v}_i\|_1 / \sqrt{m}\|\bvar{v}_i\|_2) \leq \frac{\|\bar{\bvar{S}}(\bvar{x}_k - \bvar{x}_{\ell})\|}{\|\sqrt{\bar{\bvar{C}}}(\bvar{x}_k - \bvar{x}_{\ell})\|} \leq \max_{i\in[d]}f(\|\bvar{v}_i\|_1 / \sqrt{m}\|\bvar{v}_i\|_2).
\end{align*}
As $f(x)\geq 1$, the lower bound it achieved. 
For the upper bound, note $f(x)$ is monotonic decreasing for $0\leq x \leq 1$ and monotonic increasing for $1\leq x$.
However note the argument $\|\bvar{v}_i\|_1/\sqrt{m}\|\bvar{v}_2\|_2 \leq 1$.
Therefore, minimizing the argument is equivalent to maximizing $f(x)$.
Therefore, as $\|\bvar{v}_i\|_1/\sqrt{m}\|\bvar{v}_2\|_2 \geq m^{-1/2}$, the maximum value is achieved at $f(m^{-1/2}) = 2^{-1}(m^{-1/4} + m^{1/4})$ concluding the proof.
\end{proof}

\begin{proof}[Proof of Corollary 4.2]
Conditional on a vertices community under the RDPG is equivalent with conditioning on a community's latent vector. 
Conditional on this event, the event $\{\bvar{X}_i = \bvar{x}_i\}$, the distribution of $\bvar{X}_i$ reduces to a point mass over $\bvar{x}_i$.
Therefore, when integrating in the final step of the proof of Theorem \ref{CLT} is equivalent to evaluating the normal cumulative distribution at $\bvar{x}_i$.
This is the statement given in Corollary 4.2.
\end{proof}

\begin{proof}[Proof of Theorem 4.5]
Theorem \ref{Main Bias Theorem} and Theorem \ref{CLT} establish the asymptotic distribution of the rows of $\hat{\bvar{L}}$
\begin{align*}
    (\hat{\bvar{L}}\tilde{\bvar{W}}_n)_h &= \bvar{S}^{(g)}\bvar{X}_i + \bvar{H}_h + \bvar{M}_h + \tilde{\bvar{W}}_n^T\bvar{R}^{(r)}_h
\end{align*}
Express the statistic $\hat{\bvar{D}}_i = (\hat{\bvar{X}}^{(1)})_i - (\hat{\bvar{X}}^{(2)})_{n+i}$ as
\begin{align*}
    (\hat{\bvar{D}}\tilde{\bvar{W}}_n)_i &= (\bvar{S}^{(1)} - \bvar{S}^{(2)})\bvar{X}_i + [(\bvar{H}_i - \bvar{H}_{n+i}) + (\bvar{M}_i - \bvar{M}_{n+i})] +  \tilde{\bvar{W}}_n^T(\bvar{R}^{(2)}_i - \bvar{R}^{(2)}_{n+i}).
\end{align*}
Corollary 4.4 establishes that the asymptotic distribution of $\sqrt{n}[\hat{\bvar{D}}\tilde{\bvar{W}}_n - \bvar{X}(\bvar{S}^{(1)} - \bvar{S}^{(2)})]_i$ is a mixture of normal random variables with covariance given by $\Sigma_D(\bvar{X}_i)$. 
This inspires the test statistic $W_i = \hat{\bvar{D}}_i^T\Sigma_D^{-1}(\bvar{X}_i)\hat{\bvar{D}}_i$. 
Under $H_0$, $\bvar{S}^{(1)} = \bvar{S}^{(2)}$ and $W_i$ takes the form
\begin{align*}
    W_i = \bvar{W}_n^TW_i\bvar{W}_n &= [(\bvar{H}_i - \bvar{H}_{n+i}) + (\bvar{M}_i - \bvar{M}_{n+i})]^T\Sigma_D^{-1}(\bvar{X}_i)[(\bvar{H}_i - \bvar{H}_{n+i}) + (\bvar{M}_i - \bvar{M}_{n+i})]\\
    &+2[(\bvar{H}_i - \bvar{H}_{n+i}) + (\bvar{M}_i - \bvar{M}_{n+i})]^T\Sigma_D^{-1}(\bvar{X}_i)(\bvar{R}^{(2)}_i - \bvar{R}^{(2)}_{n+i})\\
    &+(\bvar{R}^{(2)}_i - \bvar{R}^{(2)}_{n+i})^T\Sigma_D^{-1}(\bvar{X}_i)(\bvar{R}^{(2)}_i - \bvar{R}^{(2)}_{n+i}). 
\end{align*}
Corollary 3.4 establishes that $\sqrt{n}[(\bvar{H}_i - \bvar{H}_{n+i}) + (\bvar{M}_i - \bvar{M}_{n+i})]$ converges in distribution to a mixture of normal random variables with covariance $\Sigma_D(\bvar{x}_i)$. 
Moreover, results in Appendix \ref{Second Moment Appendix} establishes $\sqrt{n}\bvar{R}^{(2)}_i\overset{\prob}{\longrightarrow} 0$.
These results, the fact that $\Sigma_D^{-1}(\bvar{X}_i)$ is bounded in $n$, and Slutsky's Theorem gives 
\begin{align*}
    &n[(\bvar{H}_i - \bvar{H}_{n+i}) + (\bvar{M}_i - \bvar{M}_{n+i})]^T\Sigma_D^{-1}(\bvar{X}_i)(\bvar{R}^{(2)}_i - \bvar{R}^{(2)}_{n+i}) \overset{\prob}{\longrightarrow}\bvar{0}\\
    &n(\bvar{R}^{(2)}_i - \bvar{R}^{(2)}_{n+i})^T\Sigma_D^{-1}(\bvar{X}_i)(\bvar{R}^{(2)}_i - \bvar{R}^{(2)}_{n+i}) \overset{\prob}{\longrightarrow}\bvar{0}
\end{align*}
All that remains is analyzing $n[(\bvar{H}_i - \bvar{H}_{n+i}) + (\bvar{M}_i - \bvar{M}_{n+i})]^T\Sigma_D^{-1}(\bvar{X}_i)[(\bvar{H}_i - \bvar{H}_{n+i}) + (\bvar{M}_i - \bvar{M}_{n+i})]$.
Corollary 4.4 establishes that $\sqrt{n}\Sigma_D^{-1/2}(\bvar{X}_i)[(\bvar{N}_i - \bvar{N}_{n+i}) + (\bvar{M}_i - \bvar{M}_{n+i})] \overset{D}{\longrightarrow} N(0, \bvar{I})$. 
Therefore, by the second order Delta Method,
\begin{align*}
    n[(\bvar{H}_i - \bvar{H}_{n+i}) + (\bvar{M}_i - \bvar{M}_{n+i})]^T\Sigma_D^{-1}(\bvar{X}_i)[(\bvar{H}_i - \bvar{H}_{n+i}) + (\bvar{M}_i - \bvar{M}_{n+i})] \overset{D}{\longrightarrow}\chi^2_d
\end{align*}
Under the alternative hypothesis, $\bvar{C}^{(1)} \neq \bvar{C}^{(2)}$ and $\bvar{S}^{(1)} \neq \bvar{S}^{(2)}$.
Following a similar analysis as above, the results from Appendix \ref{First Moment Appendix} and Appendix \ref{Second Moment Appendix} with Slutsky's Theorem and Corollary 4.4 gives 
\begin{align*}
&\lim_{n\to \infty}\prob\left[\sqrt{n}(W_i - \bvar{X}_i^T(\bvar{S}^{(1)} - \bvar{S}^{(2)})^T\Sigma_D^{-1}(\bvar{X}_i)(\bvar{S}^{(1)} - \bvar{S}^{(2)}))\leq x\right]\\
&= 
\lim_{n\to \infty}\prob\left[\sqrt{n}(2\bvar{X}_i^T(\bvar{S}^{(1)} - \bvar{S}^{(2)})^T\Sigma_D^{-1}(\bvar{X}_i)(\bvar{H}_i + \bvar{M}_i - \bvar{H}_{n+i} - \bvar{M}_{n+i})\bvar{X}_i\leq x\right]\\
&= \int_{\text{supp}(F)} \Phi(x; 0, 4\bvar{y}^T(\bvar{S}^{(1)} - \bvar{S}^{(2)})^T\Sigma_D^{-1}(\bvar{y})(\bvar{S}^{(1)} - \bvar{S}^{(2)})\bvar{y})dF(\bvar{y}).
\end{align*}
\end{proof}

\end{document}